\documentclass[reqno,11pt]{amsart}
\usepackage{amssymb, amsmath,latexsym,amsfonts,amsbsy, amsthm}
\usepackage{graphicx}
\usepackage{multirow}
\usepackage{tikz}
\usepackage{mathrsfs}
\usepackage{soul}
\allowdisplaybreaks[4]

\newtheorem{theorem}{Theorem}[section]
\newtheorem{lemma}[theorem]{Lemma}
\newtheorem{proposition}[theorem]{Proposition}

\newtheorem{remark}[theorem]{Remark}

\numberwithin{equation}{section}

\makeatletter
\def\@settitle{\begin{center}%
  \baselineskip14\p@\relax
  \bfseries
  \@title
  \end{center}%
}
\def\@setauthors{%
  \begingroup
  \def\thanks{\protect\thanks@warning}%
  \trivlist
  \centering\footnotesize \@topsep30\p@\relax
  \advance\@topsep by -\baselineskip
  \item\relax
  \author@andify\authors
  \def\\{\protect\linebreak}%
  \authors
  \ifx\@empty\contribs
  \else
    ,\penalty-3 \space \@setcontribs
    \@closetoccontribs
  \fi
  \endtrivlist
  \endgroup
}
\makeatother

\begin{document}
\title[Dynamical selection and stability of semi-similar Prandtl flows]
{{Semi-similar Prandtl Flows near a Stagnation Point: \\
Dynamical Selection and Nonlinear Stability}}

\author{Chen Gao}
\address{School of Mathematical Sciences, University of Science and Technology of China, Hefei 230026, China}
\email{gaochen@amss.ac.cn}

\author{Tao Tao}
\address{School of Mathematics, Shandong University, Jinan, Shandong 250100, China}
\email{taotao@sdu.edu.cn}

\author{Zhouping Xin}
\address{The Institute of Mathematical Sciences and Department of Mathematics, The Chinese University of Hong Kong, Shatin, NT, Hong Kong}
\email{zpxin@ims.cuhk.edu.hk}

\author{Shing-Tung Yau}
\address{Yau Mathematical Sciences Center and Qiuzhen College,
Tsinghua University, Beijing 100084, China}
\email{styau@mail.tsinghua.edu.cn}

\author{Chuankai Zhao}
\address{School of Mathematics, Statistics and Mechanics, Beijing University of Technology, Beijing 100124,
China}
\email{chuankaizhao@amss.ac.cn}

\date{}

\begin{abstract}
{
This paper studies time-dependent semi-similar Prandtl flows with
outer flow $U(x)=x^m$, where $m\geq1$.  The normal coordinate is scaled by the
local boundary-layer thickness, and time by the associated diffusion scale.
Every time-monotone solution in the Oleinik class extends globally in time.
The Falkner--Skan boundary layer globally attracts every such unsteady
semi-similar profile.
The convergence rate is exponential for $m=1$ and algebraic for $m>1$.  We also prove
nonlinear stability of each admissible time-dependent semi-similar trajectory
under small compatible perturbations that break the semi-similar ansatz.
Together, these results provide a dynamical justification of the time-dependent semi-similar class within the full Prandtl theory.

For semi-similar flows, the Crocco transformation reduces the problem to a
one-dimensional degenerate parabolic equation.  Combined with weighted
barriers, comparison principles propagate time monotonicity and lead to global
continuation.  Zero-order convergence and a one-sided gradient estimate then
produce uniform strict concavity after a transient period; this concavity
supplies the damping needed for derivative decay.  The same damping mechanism
extends to the full ultraparabolic perturbation equation,
allowing us to close the nonlinear stability estimates.
}
\end{abstract}

\maketitle

\section{Introduction}

{In this paper, we study the two-dimensional unsteady Prandtl system}
\begin{equation}\label{eq:prandtl-euler}
\left\{
\begin{aligned}
    &\partial_t u+u\partial_xu+v\partial_yu-\partial_{yy}u
        =\partial_t U+ U \partial_x U,\\
    &\partial_xu+\partial_yv=0,\\
    &(u,v)(t,x,0)=(0,0),\qquad
      \lim_{y\to+\infty}u(t,x,y)=U(t,x),\\
    &u(0,x,y)=u_0(x,y),\qquad u(t,0,y)=u_1(t,y),
\end{aligned}
\right.
\end{equation}
for $(t,x,y)\in (0,+\infty)\times(0,L)\times(0,+\infty)$. This system describes the leading-order dynamics of a viscous fluid near a solid boundary at high Reynolds numbers.
In the boundary layer, the no-slip condition generates strong shear and vorticity, while viscosity acts primarily in the normal direction.
Although the Prandtl system is formally derived from the Navier--Stokes equations, its mathematical analysis is delicate: diffusion is present only in the normal variable, whereas nonlinear transport acts along the boundary.
This anisotropic structure leads to a loss of tangential derivatives and makes the well-posedness theory highly sensitive to the structure of the flow.

There is a large literature on the well-posedness and singularity theory of the
Prandtl equations. Among these developments, the part most closely related to the 
present work is Oleinik's monotone theory, where the Crocco transformation and the maximum
principle provide a fundamental local well-posedness framework; see the
monograph of Oleinik and Samokhin \cite{OS}. For favorable pressure gradients,
Xin and Zhang \cite{XZ1} constructed global weak solutions.
Related structural well-posedness results in three space variables were obtained
by Liu, Wang and Yang \cite{LWY1}.  In Sobolev spaces, local well-posedness
under suitable structural assumptions was obtained by Alexandre, Wang, Xu and
Yang \cite{AWXY} through a Nash--Moser scheme and by Masmoudi and Wong
\cite{MW} through nonlinear energy estimates. 
{
A decisive advance in the global monotone theory was made by Xin, Zhang, and
Zhao~\cite{XZZ}, who established a global well-posedness framework for general
monotone Prandtl flows under favorable pressure gradients.  Their results include
uniqueness and continuous dependence of global weak solutions, as well as an interior regularity theory.  Jia, Lei, and
Yuan~\cite{JLY-dyn} subsequently supplied an up-to-wall regularity argument
and a classical-solution formulation in the normalized setting $U\equiv1$.
} Other well-posedness theories
have been developed in analytic and Gevrey classes
\cite{SC,Maekawa,GM,LY1,LMY}.  On the negative side, without monotonicity or
sufficient regularity, the Prandtl equation is ill-posed in Sobolev spaces
\cite{GD,LWY2}, and finite-time singularity formation was studied in the
unsteady setting \cite{DS,EE}.

Besides the general well-posedness theory, special Prandtl profiles arising
from similarity or semi-similarity reductions have long played a central role
in boundary-layer theory.  Such profiles identify natural length scales and
provide canonical states against which more general flows can be compared.  
In the steady two-dimensional theory, the classical Blasius and Falkner--Skan reductions transform the Prandtl equations into ordinary differential equations and produce basic boundary-layer profiles \cite{B,FS}.  
Weyl \cite{W} established fundamental existence results for these ordinary differential
equations.  The importance of the Blasius profile is further reflected in
far-field asymptotics and stability: for the pressure-free flat-plate problem,
steady Prandtl solutions converge to the Blasius profile as $x\to+\infty$, and
quantitative stability theories have been developed in recent works; see, for
instance, \cite{Serrin,Peletier,Iyer1,WZ2,JLY,GZ}.  {
Guo, Wang, and Zhang~\cite{GWZ} studied dynamical Prandtl solutions near
Blasius-like steady states.
}
{ Favorable Falkner--Skan profiles and other steady geometric configurations have also been studied; see \cite{Iyer2,GZZ,GX}. }

The unsteady case is different in an essential way.  A fully self-similar
unsteady boundary layer would reduce the three variables $(x,y,t)$ to a single
similarity variable, and such reductions are highly restrictive.  Classically,
when a transformation reduces $(x,y,t)$ only to two independent variables, the
corresponding solution is called semi-similar; see Schlichting and Gersten
\cite{SG}.  This terminology is natural in the unsteady setting: the
boundary-layer thickness gives the normal length scale, while viscosity makes
the shear near the wall diffuse on the corresponding time scale.  Thus a
similarity reduction need not remove the time variable completely.  After the
streamwise dependence has been fixed by the similarity ansatz, the reduced
profile may still evolve in a rescaled time variable.

{
Unsteady similar and semi-similar boundary layers have long been studied in the physics and fluid mechanics literature.  
Hassan~\cite{Hassan} introduced transformations for unsteady laminar boundary layers under special classes of outer pressure distributions, while Hayasi~\cite{Hayasi1,Hayasi2} gave a systematic treatment of similar and semi-similar reductions for unsteady quasi-two-dimensional incompressible boundary-layer equations.  
These works make clear that the unsteady analogue of a steady self-similar profile is often not an ordinary differential equation, but a reduced two-variable evolution problem.
Within this tradition, Nanbu~\cite{Nanbu} computed the transient Falkner--Skan flow generated by an impulsively started wedge and studied the dependence of wall shear and displacement thickness on the pressure-gradient parameter.  
Williams and Johnson~\cite{WilliamsJohnson} used semi-similar variables to analyze unsteady boundary-layer separation, while Williams and Rhyne~\cite{WilliamsRhyne} introduced a scaling for impulsively started wedge flows that connects the short-time Rayleigh regime with the long-time Falkner--Skan regime.  
Wang~\cite{Wang1985} later formulated the Stewartson and Lam--Crocco problems in a unified semi-similar framework for singular parabolic boundary-layer equations.
}

A classical example illustrating the transient-to-steady nature of unsteady
boundary layers is the impulsively started flat-plate problem.  At early times,
the boundary layer is governed by the Rayleigh-type diffusion created by the
sudden motion of the wall, whereas at large times it is expected to approach
the steady Blasius profile.  This transition was studied by Stewartson
\cite{Stewartson1,Stewartson2}, and further investigated by Akamatsu
\cite{Akamatsu}, Tokuda \cite{Tokuda}, and Dennis \cite{Dennis} through
asymptotic and numerical methods.  This example suggests a natural PDE question
for unsteady similarity reductions: whether the reduced evolution exists
globally and converges, as time tends to infinity, to a steady self-similar
boundary layer.

Sun \cite{Sun} recently proposed a unified similarity transformation for a
class of unsteady laminar boundary layers with $U=U(x)$.  In physical variables,
the stream function is written as
\[
    \psi(t,x,y)=U(x){\Delta}(x) f(r,\tau),
    \qquad
    r=\frac{y}{{\Delta}(x)},
    \qquad
    \tau=\frac{t}{{\Delta}^2(x)},
\]
where ${\Delta}(x)$ is the boundary-layer thickness and $\tau$ is the time
normalized by the diffusion time ${\Delta}^2(x)$.  This time scale is the key
point: it couples the unsteady evolution to the normal thickness of the boundary
layer and reduces the original variables $(x,y,t)$ to $(r,\tau)$.
Substitution into the unsteady Prandtl equation leads to an equation of the
form
\[
    f_{rrr}
    +\alpha f f_{rr}
    +\beta(1-f_r^2)
    =
    f_{\tau r}
    +\gamma \tau
    \bigl(f_\tau f_{rr}
    -f_r f_{\tau r}\bigr),
\]
where $\alpha$, $\beta$, and $\gamma$ are determined by $U$ and ${\Delta}$.  When
these coefficients are constants, the admissible outer flows are power laws
\[
    U(x)=Cx^m,
    \qquad
    m=\frac{\beta}{2\alpha-\beta},
\]
with
\[
    {\Delta}(x)
    =
    \left|
    \frac{(2\alpha-\beta)x}{U(x)}
    \right|^{1/2}.
\]
Thus Sun's transformation provides a natural semi-similar framework for a family of unsteady Prandtl flows with power-law outer Euler flows. 

{
The preceding semi-similar literature primarily provides formal reductions,
explicit or series representations, asymptotic expansions, and numerical
descriptions of special unsteady boundary layers.  The associated PDE theory
must address the degenerate Crocco evolution, whose diffusion coefficient
vanishes at \(\eta=1\) and whose wall condition is nonlinear.  The first issue
is global existence and uniqueness in the natural weighted strong class,
without a perturbative smallness assumption on the semi-similar profile.  The
second is quantitative convergence, including weighted normal derivatives, to
the stationary Falkner--Skan state.  The third is nonlinear stability of this
time-dependent orbit in the full \((\tau,\xi,\eta)\)-dependent Crocco equation.
The final issue is global propagation of the edge-degenerate weighted
regularity required for continuation and inversion of the Crocco transformation.
}

{
The last question is intrinsic to boundary-layer theory rather than an
auxiliary strengthening of global existence.  A complete global theory should
determine not only whether a solution exists for all time, but whether the
uniform regularity encoded by the natural strong Oleinik class is propagated
by the dynamics.
}

{
For boundary-layer modeling, the reduced profiles should represent physically
admissible flows rather than formal shapes.  This is why a rigorous
justification of the semi-similar reduction matters.  Global solvability shows
that the velocity profile remains monotone from the no-slip wall to the outer
flow throughout the evolution.  Long-time selection predicts that both the
profile and the scaled wall shear relax toward the Falkner--Skan regime.
In practice, an exact semi-similar scaling is also subject to streamwise
disturbances.  Stability shows that the predicted evolution survives small
departures from this idealized scaling.  Uniform normal estimates preserve
controlled matching with the outer flow across the far field.
}

{
The purpose of the present paper is to answer these questions for every fixed $m\geq1$ in Sun's power-law class, corresponding, after normalization, to the outer flow $U(x)=x^m$.

\medskip
\noindent\textbf{Main result, informally.}
}

{
For semi-similar initial data satisfying a one-sided time-monotonicity
condition and the natural weighted derivative bounds, the reduced Crocco
evolution is global for every fixed $m\geq1$.  The initial profile is not
required to be a small perturbation of the stationary Falkner--Skan profile and
may be far from it in the weighted topology.  The evolving profile converges,
together with its controlled derivatives, to that profile, exponentially for
$m=1$ and algebraically for $m>1$.
}

{
More importantly, we prove nonlinear stability of the entire genuinely
time-dependent semi-similar trajectory under general perturbations that break
the semi-similar ansatz.  Smallness is imposed only on the perturbation relative
to the prescribed evolving background.  Thus every admissible semi-similar
trajectory generates a stable family of global, large-amplitude structured
solutions of the full Prandtl system, possibly far from the stationary
self-similar state, while retaining the natural strong Oleinik regularity.
}

The global semi-similar dynamics produce a uniform strict-concavity structure after a transient period. This dynamically generated coercivity is precisely what stabilizes the evolving background against perturbations outside the semi-similar class.
The rest of the introduction gives the precise formulation of these results and then explains the analytical mechanism behind the proof.
\subsection{Global existence and large-time behavior of semi-similar power-law flows}

Fix $m\geq1$ and consider the normalized power-law outer flow $U(x)=x^m$.
Set
\[
    {\Delta_m}(x)=x^{\frac{1-m}{2}},
    \qquad
    r=\frac{y}{{\Delta_m}(x)}=x^{\frac{m-1}{2}}y,
    \qquad
    \tau=\frac{t}{{\Delta_m(x)^2}}=t x^{m-1}.
\]
We seek semi-similar solutions of the form $u^s(t,x,y)=x^m\partial_r\psi^s(\tau,r)$.
Substituting the preceding ansatz into \eqref{eq:prandtl-euler}, we obtain the following initial-boundary-value problem:
\begin{equation}\label{eq:semi-self-similar-profile-m}
\left\{
\begin{aligned}
&\partial_{rrr}\psi^s+\frac{m+1}{2}\psi^s\partial_{rr}\psi^s+m\bigl(1-(\partial_r\psi^s)^2\bigr)=\partial_{\tau r}\psi^s+(m-1)\tau\bigl(\partial_r\psi^s\,\partial_{\tau r}\psi^s-\partial_\tau\psi^s\,\partial_{rr}\psi^s\bigr),\\
&\psi^s(\tau,0)=\partial_r\psi^s(\tau,0)=0,\qquad\lim_{r\to+\infty}\partial_r\psi^s(\tau,r)=1,\\
&\partial_r\psi^s(0,r)=\widetilde u_0^s(r),
\end{aligned}
\right.
\end{equation}
where $(\tau,r) \in (0,\infty)\times(0,\infty)$.

We restrict attention to profiles that are strictly monotone in the normal variable: $\partial_{rr}\psi^s(\tau,r)>0$. Let
\[
    (\tau,\eta):=\left(t x^{m-1},\frac{u^s(t,x,y)}{x^m}\right)=\left(t x^{m-1},\psi^s_r(\tau,r)\right),
\]
and
\[
    w^s(\tau,\eta):=\frac{\partial_y u^s(t,x,y)}{x^{\frac{3m-1}{2}}}=\psi^s_{rr}(\tau,r).
\]
Then \eqref{eq:semi-self-similar-profile-m} is transformed into
\begin{equation}\label{eq:semi-self-similar-crocco-m}
\left\{
\begin{aligned}
&\bigl[1+(m-1)\tau\eta\bigr]\partial_\tau w^s =(w^s)^2\partial_{\eta\eta}w^s +m(\eta^2-1)\partial_\eta w^s -\frac{3m-1}{2}\eta w^s,\\
&w^s(\tau,1)=0,\qquad w^s(\tau,0)\partial_\eta w^s(\tau,0)+m=0,\\
&w^s(0,\eta)=w_0^s(\eta),
\end{aligned}
\right.
\end{equation}
where $(\tau,\eta)\in (0,\infty)\times (0,1)$.

{
The restriction $m\geq1$ is also dictated by the evolution structure of
\eqref{eq:semi-self-similar-crocco-m}.  If $m<1$, the coefficient
$1+(m-1)\tau\eta$ reaches zero at $\eta=1$ when
$\tau=(1-m)^{-1}$; for larger $\tau$, it vanishes at an interior point and
changes sign across that point.  Consequently, the reduced equation is no
longer a forward degenerate parabolic problem on the whole strip for arbitrary
$\tau$.  The global-in-$\tau$ problem in this regime is therefore of a
different type and is not considered in this paper.
}

Furthermore, the stationary state of \eqref{eq:semi-self-similar-profile-m} is obtained by setting $\psi^s(\tau,r)=f_m(r)$.
The corresponding profile satisfies the favorable Falkner--Skan problem
\begin{equation}\label{eq:falkner-skan-family}
\left\{
\begin{aligned}
&f_m'''+\frac{m+1}{2}f_m f_m'' +m\bigl(1-(f_m')^2\bigr)=0,\\
&f_m(0)=f_m'(0)=0, \qquad \lim_{r\to+\infty}f_m'(r)=1.
\end{aligned}
\right.
\end{equation}
Applying the steady Crocco transformation
\[
    \eta=f_m'(r), \qquad \overline w_m(\eta)=f_m''(r),
\]
to \eqref{eq:falkner-skan-family}, we find that $\overline w_m$ satisfies
\begin{equation}\label{eq:steady-self-similar-m}
\left\{
\begin{aligned}
&\overline w_m^2\overline w_m''
 +m(\eta^2-1)\overline w_m'
 -\frac{3m-1}{2}\eta\overline w_m=0,
 \qquad 0<\eta<1,\\
&\overline w_m(1)=0,
 \qquad
 \overline w_m(0)\overline w_m'(0)+m=0.
\end{aligned}
\right.
\end{equation}

{
\begin{lemma}\label{self-similar-properties}
For every fixed $m\geq1$, let $\overline w_m$ be the unique positive solution of \eqref{eq:steady-self-similar-m}.  
There exist constants $M_i=M_i(m)>0$, $\eta_0=\eta_0(m)\in(0,1)$, and $\mu_m\in(0,1)$ such that
\begin{equation}\label{eq:steady-self-similar-properties-m}
  \begin{aligned}
    M_1(1-\eta)\sigma_m(\eta) \leq\overline w_m (\eta) \leq{}&M_2(1-\eta)\sigma_m(\eta),\\
 -M_3\sigma_m (\eta)\leq\overline w_m'(\eta) \leq{}&-M_4\sigma_m (\eta),\\
 -M_5 \leq\overline w_m (\eta)\overline w_m'' (\eta) \leq{}&-M_6.
  \end{aligned}
\end{equation}
Moreover,
\begin{equation}\label{eq:steady-edge-curvature-family}
\begin{aligned}
 \overline w_m(\eta)
 &=\sqrt{m+1}\,(1-\eta)\sigma_m(\eta)
 \bigl(1+O(\sigma_m(\eta)^{-1})\bigr),\\
 \overline w_m'(\eta)
 &=-\sqrt{m+1}\,\sigma_m(\eta)
 \bigl(1+O(\sigma_m(\eta)^{-1})\bigr),\\
 \overline w_m''(\eta)
 &=-\frac{\sqrt{m+1}}{2(1-\eta)\sigma_m(\eta)}
 \bigl(1+O(\sigma_m(\eta)^{-1})\bigr),
\end{aligned}
\qquad \eta\to1^-.
\end{equation}
Here $M_2=\sqrt{m+1}$ and $\sigma_m(\eta)=\sqrt{-\ln[\mu_m(1-\eta)]}$.
\end{lemma}
}

The proof of Lemma~\ref{self-similar-properties} is given in Appendix~\ref{app:steady-profile}. 
For $m=1$, $M_2=\sqrt2$, consistently with the classical estimate in \cite{OS} (p.~109). 
The first theorem gives both the global semi-similar dynamics and their large-time convergence to the steady Falkner--Skan profile. 
\begin{theorem}\label{thm:ws-regularity}
Fix $m\geq1$. Assume that $w_0^s\in C^2([0,1))$ and that there exist
$E_1,E_2,E_3>0$ such that $w_0^s(\eta)$ satisfies
\begin{equation}\label{eq:initial-data}
\begin{gathered}
 \frac{1}{E_1}\overline w_m(\eta)\leq w_0^s(\eta)\leq E_1\overline w_m(\eta),\\
 \left|\partial_\eta w^s_0(\eta)\right| \leq E_2(-\overline w_m'(\eta)),\qquad
 \left|\partial_{\eta\eta} w^s_0(\eta)\right| \leq E_3(-\overline w_m''(\eta)),\\
\left.\partial_\tau w^s\right|_{\tau=0}=(w_0^s)^2\partial_{\eta\eta}w_0^s(\eta)+m (\eta^2-1)\partial_\eta w_0^s(\eta)-\frac{3m-1}{2}\eta w^s_0(\eta)\leq 0,
\end{gathered}
\end{equation}
and the compatibility condition
\begin{equation*}
    \left.\left(w_0^s\partial_\eta w_0^s+m\right)\right|_{\eta=0}=0.
\end{equation*}

Then the equation \eqref{eq:semi-self-similar-crocco-m} has a unique solution $w^s(\tau, \eta)\in C^{1,2}_{\text{par}}\left([0,+\infty)\times [0,1)\right)$ and there exists a positive constant $C=C(m,E_1,E_2,E_3)$ such that
\begin{equation}\label{eq:ws-estimates-all}
\begin{aligned}
 &\overline w_m(\eta)\leq w^s(\tau,\eta)\leq C\overline w_m(\eta), \qquad |\partial_\eta w^s(\tau,\eta)|
 \leq-Ce^{2m\tau}\overline w_m'(\eta),\\
 &|\partial_{\eta\eta}w^s(\tau,\eta)|
 \leq-C[1+(m-1)\tau]e^{2m\tau}\overline w_m''(\eta),\qquad
 -C\overline w_m(\eta)
 \leq\partial_\tau w^s(\tau,\eta)\leq0.
\end{aligned}
\end{equation}
Moreover, there exists $\beta_m>0$ such that for any $(\tau,\eta)\in [0,+\infty)\times [0,1)$, the solution satisfies
\begin{equation}\label{eq:ws-large-time}
\begin{aligned}
 &\left|w^s(\tau,\eta)-\overline w_m(\eta)\right|\leq C\rho_m (\tau)\overline w_m (\eta),\\
 &\left[1+(m-1)\tau\right]|\partial_\tau w^s(\tau,\eta)|\leq C\rho_m(\tau)\overline w_m(\eta),\\
 &|\partial_\eta w^s(\tau,\eta)-\overline w_m'(\eta)|\leq -C\rho_m(\tau) \overline w_m'(\eta),\\
 &|\partial_{\eta\eta}w^s(\tau,\eta)-\overline w_m''(\eta)|\leq -C\rho_m (\tau) \overline w_m''(\eta),
\end{aligned}
\end{equation}
where 
\begin{equation}\label{eq:m-time-weights}
 \rho_m(\tau)=
 \begin{cases}
  e^{-\beta_1\tau},&m=1,\\
  (1+\tau)^{-\beta_m},&m>1.
 \end{cases}
\end{equation}
\end{theorem}

{
\begin{theorem}[Finite-time loss of weighted $C^2$ edge control]
\label{thm:weighted-c2-blowup}
Fix $m\geq1$.  There exists a positive smooth initial datum $w_0^s$
satisfying all the hypotheses of Theorem~\ref{thm:ws-regularity} except the
initial time-monotonicity condition, such that the corresponding local
solution leaves the natural weighted $C^2$ class in finite time.  More
precisely, there exists $0<T^*<\infty$ such that
\[
 \limsup_{\tau\uparrow T^*}\quad \left(\bigg\|\frac{\partial_\eta w^s}{- \overline w_m'}\bigg\|_{L^\infty_\eta}+\bigg\|\frac{\partial_{\eta\eta} w^s}{- \overline w_m''}\bigg\|_{L^\infty_\eta}\right)=+\infty.
 \]
For the solution constructed in the proof, the relative amplitude remains
bounded up to $T^*$:
\[
 1\leq \frac{w^s(\tau,\eta)}{\overline w_m(\eta)}\leq3,
 \qquad 0\leq\tau<T^*,\quad 0\leq\eta<1.
\]
\end{theorem}

{
\medskip
\noindent\textit{Interpretation.}
The point of the theorem is the failure of global propagation in the strong
Oleinik weighted class.  The degenerate edge $\eta=1$ corresponds to the
physical far-field boundary $y=+\infty$.  The theorem does not assert that the
Prandtl velocity becomes singular at this boundary; it asserts that the
weighted first- and second-derivative estimates can fail while the relative
amplitude remains bounded.  The global theories of Xin, Zhang, and
Zhao~\cite{XZZ} and Jia, Lei, and Yuan~\cite{JLY-dyn} do not provide such
uniform weighted far-field estimates.  The counterexample shows that these
estimates cannot, for general positive data, be recovered from global
solvability and relative-amplitude bounds alone.  Thus a global theory within
the strong Oleinik weighted class requires additional structure.
}

}
\begin{theorem}\label{thm:higher-order-regularity}
Fix $m\geq1$ and let $w_0^s$ satisfy the hypotheses of Theorem~\ref{thm:ws-regularity}.
Additionally, assume that there exist constants $E_4>0$ and $E_5>0$ such that
\begin{equation}\label{eq:initial-high}
  \left|\partial_\eta^3w_0^s(\eta)\right|\leq \frac{E_4}{(1-\eta)^2 \sigma_m(\eta)},\qquad
  \left|\partial_\eta^4w_0^s(\eta)\right|\leq \frac{E_5}{(1-\eta)^3 \sigma_m(\eta)}.
\end{equation}
Assume further that the higher-order compatibility conditions hold at the corner $(\tau,\eta) = (0, 0)$.

Then the semi-similar solution $w^s(\tau, \eta)$ obtained in
Theorem~\ref{thm:ws-regularity} admits higher-order regularity
\begin{equation}
    w^s \in C^{2, 4}_{\text{par}} \left([0,+\infty)\times [0,1)\right).
\end{equation}
Furthermore, there exists a positive constant $C=C(m,E_1,E_2,E_3,E_4,E_5)$ such that, with the same decay rate $\beta_m>0$ as in Theorem~\ref{thm:ws-regularity}, the following estimates hold:
\begin{equation}\label{eq:ws-high-large-time}
\begin{aligned}
 \left[1+(m-1)\tau\right]^2 |\partial_{\tau\tau}w^s(\tau,\eta)| & \leq C \rho_m(\tau) (1-\eta)\sigma_m(\eta),\\
 \left[1+(m-1)\tau\right] |\partial_{\tau\eta}w^s (\tau,\eta)| &\leq C\rho_m(\tau)\sigma_m(\eta),\\
 \left[1+(m-1)\tau\right] |\partial_{\tau\eta\eta}w^s (\tau,\eta)| &\leq \frac{C\rho_m(\tau)} {(1-\eta)\sigma_m(\eta)},\\
 |\partial_\eta^3(w^s-\overline w_m) (\tau,\eta)| & \leq \frac{C\rho_m(\tau)}{(1-\eta)^2\sigma_m(\eta)},\\
 |\partial_\eta^4(w^s-\overline w_m) (\tau,\eta)| & \leq\frac{C\rho_m(\tau)}{(1-\eta)^3\sigma_m(\eta)},
\end{aligned}
\end{equation}
where $\sigma_m(\eta)=\sqrt{-\ln [\mu_m(1-\eta)]}$.
\end{theorem}

Through the inverse Crocco transformation, the preceding Crocco estimates yield the following Eulerian formulation.  
Let $\overline u_m$ be the Eulerian Falkner--Skan flow corresponding to $\overline w_m$.

\begin{theorem}\label{thm:eulerian-semi-self-similar}
Fix $m\geq1$.  
Let $u_0^s$ be semi-similar initial data and let $w_0^s$ be its Crocco profile.  Assume that $w_0^s$ satisfies the assumptions of Theorem~\ref{thm:ws-regularity} and the higher-order assumptions of Theorem~\ref{thm:higher-order-regularity}.
Then the Prandtl system \eqref{eq:prandtl-euler}, with $U(x)=x^m$, has a global semi-similar solution $u^s(t,x,y)$ with initial data $u_0^s$.  
It is reconstructed for $x>0$ by
\begin{equation}\label{eq:inverse-m-eulerian}
 y=x^{\frac{1-m}{2}} \int_0^{u^s/U}\frac{d\zeta}{w^s(\tau,\zeta)},
 \qquad \tau=t x^{m-1},
\end{equation}
and extends continuously to $x=0$ with the prescribed wall and initial conditions. 
Moreover, there exists a constant $C=C(m,E_1,E_2,E_3,E_4,E_5)>0$ such that, with the same decay rate $\beta_m$ as in Theorem~\ref{thm:ws-regularity}, the following estimates hold for $t\geq0$, $x>0$, and $y\geq0$:
\begin{equation}\label{eq:eulerian-high}
\begin{aligned}
\left|\partial_y^\ell
 \left(u^s(t,x,y)-\overline u_m(x,y)\right)\right|
&\leq C \rho_m(\tau)
 x^{\frac{(m-1)\ell}{2}}  \\
&\quad {}\times
 \left(x^m-\overline u_m(x,y)\right)
 \overline\sigma_m^{\ell+2}(x,y),
 \qquad 0\leq\ell\leq5, \\[0.4em]
\left[1+(m-1)\tau\right]
\left|\partial_t\partial_y^\ell
 \left(u^s(t,x,y)-\overline u_m(x,y)\right)\right|
&\leq C \rho_m(\tau)
 x^{\frac{(m-1)(\ell+2)}{2}} \\
&\quad {}\times
 \left(x^m-\overline u_m(x,y)\right)
 \overline\sigma_m^{\ell+2}(x,y),
 \qquad 0\leq\ell\leq3, \\[0.4em]
\left[1+(m-1)\tau\right]^2
\left|\partial_t^2\partial_y^\ell
 \left(u^s(t,x,y)-\overline u_m(x,y)\right)\right|
&\leq C \rho_m(\tau)
 x^{\frac{(m-1)(\ell+4)}{2}} \\
&\quad {}\times
 \left(x^m-\overline u_m(x,y)\right)
 \overline\sigma_m^{\ell+4}(x,y),
 \qquad 0\leq\ell\leq1,
\end{aligned}
\end{equation}
where
\begin{align*}
  \overline\sigma_m(x,y)=\sqrt{-\ln \left[\mu_m \left(1-\frac{\overline u_m (x,y)}{x^m}\right)\right]},
\end{align*}
\begin{equation*}
 \rho_m(tx^{m-1})=
 \begin{cases}
  e^{-\beta_1 t},&m=1,\\
  (1+tx^{m-1})^{-\beta_m},&m>1.
 \end{cases}
\end{equation*}
\end{theorem}

{
\begin{remark}[Meaning of the physical-time decay]
For $m>1$, the decay in Theorem~\ref{thm:eulerian-semi-self-similar} is
governed by $\tau=tx^{m-1}$.  Thus the convergence holds as $t\to+\infty$ for
each fixed $x>0$, and it is uniform on every region $x\geq x_0>0$.  It is not
a uniform-in-$x$ convergence statement up to the stagnation point $x=0$.
For $m=1$, the rescaled time is simply $\tau=t$.
\end{remark}
}

\begin{remark}\label{rem:oleinik-steady-tail}
The classical Falkner--Skan asymptotics presented in Oleinik--Samokhin~\cite[Section~2.4, formula~(2.4.12)]{OS} and Hartman~\cite[Chapter~XIV, Part~II, Section~9]{Hartman2002} imply the following formula for every fixed $m\geq1$:
\begin{equation}\label{eq:oleinik-steady-tail}
 1-\frac{\overline u_m(x,y)}{U(x)}
 \sim \overline M_m r^{-\frac{5m+1}{m+1}}
 \exp\left(-\frac{m+1}{4}r^2
 +\frac{m+1}{2}\bar d_m r\right),
 \qquad r\to+\infty,
\end{equation}
where $\overline M_m>0$, $r=x^{\frac{m-1}{2}}y$, and 
\[
 \bar d_m:=\int_0^{+\infty}\bigl(1-f_m'(r)\bigr)\,dr>0
\]
is the dimensionless displacement thickness. Since
$\overline\sigma_m/(\sqrt{m+1}\,r/2)\to1$ as $r\to+\infty$, after enlarging
the constant on bounded $r$-intervals,
\begin{equation}\label{eq:oleinik-weighted-tail}
 \bigl(x^m-\overline u_m(x,y)\bigr)\overline\sigma_m^\ell
 \leq C x^m (1+r)^{\ell-\frac{5m+1}{m+1}} \exp\left(-\frac{m+1}{4}r^2 +\frac{m+1}{2}\bar d_m r\right),
 \qquad 0\leq\ell\leq7.
\end{equation}
These sharp steady-profile tails will be used in the proof of Theorem~\ref{thm:eulerian-stability}.
\end{remark}

\subsection{Stability of semi-similar power-law flows}

We next allow perturbations which depend on the tangential variable and hence leave the semi-similar class.
Thus the stability problem is not restricted to the reduced semi-similar dynamics: the perturbations may depend on all Crocco variables and need not preserve the semi-similar ansatz.  
Set
\begin{equation}\label{eq:full-crocco-variables-m}
 (\tau,\xi,\eta)
 =\left(tx^{m-1},\ln x,\frac{u(t,x,y)}{x^m}\right),
\end{equation}
and
\[
 w(\tau,\xi,\eta)
 =x^{(1-m)/2}\frac{\partial_yu(t,x,y)}{x^m}.
\]
Then the Prandtl equation \eqref{eq:prandtl-euler} reduces to
\begin{equation}\label{eq:prandtl-crocco-m}
\left\{
\begin{aligned}
 &w^2\partial_{\eta\eta} w+m(\eta^2-1)\partial_\eta w -\frac{3m-1}{2}\eta w-\eta \partial_\xi w -[1+(m-1)\tau\eta]\partial_\tau w=0,\\
 &w(\tau,\xi,1)=0,\qquad
 w(\tau,\xi,0) \partial_\eta w(\tau,\xi,0) + m=0,\\
 &w(0,\xi,\eta)=w_0(\xi,\eta),\qquad
 \lim_{\xi\to-\infty}w(\tau,\xi,\eta)=w^-(\tau,\eta).
\end{aligned}
\right.
\end{equation}
Here $(\tau,\xi,\eta)\in(0,+\infty)\times(-\infty,\ln L)\times(0,1)$ and $w^{-}(\tau,\eta)$ satisfies
\begin{equation}
\left\{
\begin{aligned}
&(w^-)^2\partial_{\eta\eta}w^-+ m (\eta^2-1)\partial_\eta w^- -\frac{3m-1}{2} \eta w^- -[1+(m-1)\tau\eta]\partial_\tau w^-=0, \\
 &w^-(\tau,1)=0,\qquad w^-(\tau,0)\partial_\eta w^-(\tau,0)+m=0, \\
&w^-(0,\eta)=w^-_0(\eta),
\end{aligned}
\right.
\end{equation}
where $(\tau,\eta)\in (0,+\infty)\times (0,1)$.

For simplicity, in this paper we consider only $\lim_{\xi\to-\infty}w_0(\xi,\eta)=w_0^s(\eta)$ and $w^{-}(\tau,\eta)=w^s(\tau,\eta)$.

For a function $\varphi=\varphi(\tau,\xi,\eta)$, define the time-dependent
weighted norm
\begin{equation}\label{eq:X-norm-m}
\begin{aligned}
 \|\varphi(\tau,\cdot)\|_{\mathbb X_m}:={}&
 \left\|\frac{\varphi}{\overline w_m}\right\|_{L^\infty_{\xi,\eta}}
 +\left\|\frac{\partial_\eta\varphi}{-\overline w_m'}\right\|_{L^\infty_{\xi,\eta}}
 +\left\|\frac{\partial_{\eta\eta}\varphi}{-\overline w_m''}\right\|_{L^\infty_{\xi,\eta}}\\
 &+\left\|\frac{[1+(m-1)\tau]\partial_\tau\varphi}{\overline w_m}\right\|_{L^\infty_{\xi,\eta}}
 +\left\|\frac{\partial_\xi\varphi}{\overline w_m}\right\|_{L^\infty_{\xi,\eta}}.
\end{aligned}
\end{equation}
When the $\mathbb X_m$-norm is applied to an initial perturbation, the time derivative is understood as the initial time derivative induced by the equations. Specifically,
\begin{align*}
\partial_\tau w|_{\tau=0}:={}& w_0^2\partial_{\eta\eta}w_0 +m(\eta^2-1)\partial_\eta w_0-\frac{3m-1}{2}\eta w_0-\eta \partial_\xi w_0,\\
\partial_\tau w^s|_{\tau=0}:={}& (w_0^s)^2\partial_{\eta\eta}w_0^s +m(\eta^2-1)\partial_\eta w_0^s-\frac{3m-1}{2}\eta w_0^s.
\end{align*}

We now state the nonlinear stability result in Crocco variables.

\begin{theorem}\label{thm:semi-self-similar-stability}
Fix $m\geq1$ and $E^*\geq1$.  
Let $w_0^s(\eta)$ be admissible semi-similar initial data satisfying the assumptions of Theorems~\ref{thm:ws-regularity} and~\ref{thm:higher-order-regularity}, with $E_i\leq E^*$, $1\leq i\leq 5$.
Then there exist constants $\delta_0=\delta_0(m,E^*)>0$ and $C=C(m,E^*)>0$ such that the following holds.

For any initial data $w_0(\xi,\eta)$ compatible with the full initial-boundary value problem \eqref{eq:prandtl-crocco-m} and satisfying
\begin{equation}
    \left\|w_0-w_0^s\right\|_{\mathbb X_m} \leq \delta,
    \qquad 0<\delta\leq\delta_0,
    \nonumber
\end{equation}
where the initial time derivative appearing in the $\mathbb X_m$-norm is understood as above, \eqref{eq:prandtl-crocco-m} admits a global solution $w(\tau,\xi,\eta)$ satisfying
\begin{equation}\label{eq:smallness-condition}
    \left\|w(\tau,\cdot)-w^s(\tau,\cdot)\right\|_{\mathbb X_m}
    \leq C\delta \rho_m(\tau), \qquad \tau\geq 0,
\end{equation}
where 
\begin{equation*}
 \rho_m(\tau)=
 \begin{cases}
  e^{-\beta_1\tau},&m=1,\\
  (1+\tau)^{-\beta_m},&m>1.
 \end{cases}
\end{equation*}
\end{theorem}

\begin{remark}\label{rem:away-from-sss}
{
The semi-similar solution $w^s$ is not assumed to be close to the stationary
profile $\overline w_m$.  Smallness is imposed only on $w_0-w_0^s$, and the
allowed size may depend on the weighted bounds for $w^s$.  Thus
Theorem~\ref{thm:semi-self-similar-stability} establishes nonlinear stability
of each prescribed admissible time-dependent semi-similar solution, not only
of the stationary Falkner--Skan profile.  The choice $w^-=w^s$ as
$\xi\to-\infty$ is made for simplicity.  The same conclusion holds if the
initial entrance profile is sufficiently close to $w_0^s$; see
Remark~\ref{rem:relaxed-entrance}.
}
\end{remark}

{
\begin{remark}\label{rem:nonuniform-solution-map}
The initial datum $w_0^s$ for the semi-similar solution is assumed to be more
regular than the initial perturbation.  This distinction is necessary.  If
$w_0^s$ were assumed to have only the same regularity as the perturbation, the
stability estimate would imply uniform continuity of the solution map on
bounded sets in the weighted $C^2$ norm.
Theorem~\ref{thm:edge-nonuniformity-eta} constructs pairs of time-monotone
global semi-similar solutions for which the initial distance tends to zero but
the distance at every fixed positive time that is sufficiently small remains bounded
from below.
\end{remark}
}

The preceding theorem is formulated in Crocco variables, where the Oleinik weights and the perturbation norm are most natural.  
Returning to the original variables gives the following stability result for the Prandtl system.

\begin{theorem}\label{thm:eulerian-stability}
Fix $m\geq1$ and let $u^s(t,x,y)$ be the semi-similar Prandtl flow obtained in Theorem~\ref{thm:eulerian-semi-self-similar}. 
Assume that $w_0$ satisfies the hypotheses of Theorem~\ref{thm:semi-self-similar-stability} with $0<\delta\leq\delta_0$, and let $u_0(x,y)$ be the corresponding Eulerian initial datum under the inverse Crocco transformation.

Then the Prandtl system \eqref{eq:prandtl-euler} with $U(x)=x^m$ admits a global solution $u(t,x,y)$ with initial data $u_0$. 
Moreover, there exists a constant $C=C(m, E^*)>0$ such that the following estimates hold for all $t\geq0$, $0\leq x<L$, and $y\geq0$:
\begin{align*}
    |u(t,x,y)-u^s(t,x,y)|&\leq C\delta\rho_m(tx^{m-1})x^{\frac{3m-1}{2}}y \left(1+x^{\frac{m-1}{2}}y \right)\\
&\quad {}\times \exp\left\{-\left[\frac{m+1}{4}
 -C\delta\rho_m(tx^{m-1})\right]x^{m-1}y^2
 +\frac{m+1}{2}\bar d_m x^{\frac{m-1}{2}}y\right\}, \\[0.4em]
 \left|\partial_y^\ell\bigl(u(t,x,y)-u^s(t,x,y)\bigr)\right|
 &\leq C\delta\rho_m(tx^{m-1})
 x^{m+\frac{\ell(m-1)}{2}}
 \left(1+x^{\frac{m-1}{2}}y\right)^{\ell+2}\\
 &\quad {}\times\exp\left\{-\left[\frac{m+1}{4}
 -C\delta\rho_m(tx^{m-1})\right]x^{m-1}y^2\right.\\
 &\qquad\qquad\left.
 +\frac{m+1}{2}\bar d_m x^{\frac{m-1}{2}}y\right\},
 \qquad1\leq\ell\leq3, \\[0.4em]
 \left|\partial_x\partial_y^\ell
 \bigl(u(t,x,y)-u^s(t,x,y)\bigr)\right|
 &\leq C\delta\rho_m(tx^{m-1})
 x^{m-1+\frac{\ell(m-1)}{2}}
 \bigl(1+x^{\frac{m-1}{2}}y\bigr)^{\ell+2}\\
 &\quad {}\times\exp\left\{-\left[\frac{m+1}{4}
 -C\delta\rho_m(tx^{m-1})\right]x^{m-1}y^2\right.\\
 &\qquad\qquad\left.
 +\frac{m+1}{2}\bar d_m x^{\frac{m-1}{2}}y\right\},
 \qquad0\leq\ell\leq1, \\[0.4em]
 \left|\partial_t\partial_y^\ell \left(u(t,x,y)-u^s(t,x,y)\right)\right| &\leq C\delta\rho_m(tx^{m-1})
 [1+(m-1)tx^{m-1}]^{-1} x^{m+\frac{(\ell+2)(m-1)}{2}}  \bigl(1+x^{\frac{m-1}{2}}y\bigr)^{\ell+4}\\
  &\quad  {}\times \exp\left\{-\left[\frac{m+1}{4}
  -C\delta\rho_m(tx^{m-1})\right]x^{m-1}y^2 \right.\\
  &\qquad\qquad\left.
  +\frac{m+1}{2}\bar d_m x^{\frac{m-1}{2}}y\right\},\\
  &\qquad 0\leq\ell\leq1.
\end{align*}
Here 
\begin{equation*}
 \rho_m(tx^{m-1})=
 \begin{cases}
  e^{-\beta_1 t},&m=1,\\
  (1+tx^{m-1})^{-\beta_m},&m>1,
 \end{cases}
\end{equation*}
$\beta_m$ is the decay rate from Theorem~\ref{thm:semi-self-similar-stability}, and $\bar d_m>0$ is the displacement-thickness constant defined in Remark~\ref{rem:oleinik-steady-tail}.
\end{theorem}

{
\begin{remark}[Uniform normal far-field regularity]
\label{rem:normal-far-field-regularity}
The estimates above are uniform on the full Eulerian normal half-line
$y\in[0,+\infty)$ and give quantitative decay as $y\to+\infty$ for the normal
derivatives controlled by the theorem.  Under the Crocco transformation, this
is precisely the degenerate limit $\eta\uparrow1$.  Such uniform weighted
far-field control is stronger than global solvability in a weaker solution
class and is not supplied by the available global existence results.  The
counterexample in Theorem~\ref{thm:weighted-c2-blowup} shows that positivity
and relative-amplitude bounds alone do not propagate this control.  The
present theorem proves that, under the one-sided time-monotonicity and weighted
regularity assumptions, the corresponding solutions remain in the strong
Oleinik weighted class for all time.
\end{remark}
}

\vspace{2mm}
{
\medskip
\noindent\textit{Proof strategy.}
The proof has three parts.  First, comparison with the stationary profile
controls $w^s$, and the maximum principle preserves
$\partial_\tau w^s\leq0$.  Weighted barriers then control the first and second
normal derivatives on every finite time interval.  These estimates give the
global continuation of the semi-similar solution.

Second, write
$\widetilde w=w^s-\overline w_m$.  A weighted comparison gives
$0\leq\widetilde w\leq C\rho_m\overline w_m$.  The main edge estimate is
obtained from the logarithmic derivative
$P=\partial_\eta\widetilde w/\widetilde w$, which
satisfies
\[
\begin{aligned}
 &(w^s)^2(\partial_\eta P+P^2)+m(\eta^2-1)P
 +(w^s+\overline w_m)\overline w_m''-\frac{3m-1}{2}\eta\\
 &\qquad=[1+(m-1)\tau\eta]
 \frac{\partial_\tau\widetilde w}{\widetilde w}\leq0.
\end{aligned}
\]
This spatial Riccati inequality yields
\begin{equation}\label{eq:intro-riccati-gradient}
 \partial_\eta\widetilde w
 \leq\frac{K}{1-\eta}\widetilde w
 \leq C\rho_m(\tau)(-\overline w_m').
\end{equation}
This rules out the formation of a sharp transition layer near $\eta=1$.  Substitution in the
equation shows that, after a finite time,
\[
 w^s\partial_{\eta\eta}w^s\leq-c_*<0.
\]
This is the point where time monotonicity becomes coercivity.

The strict-concavity estimate gives a negative zero-order term in the equation
for $[1+(m-1)\tau]\partial_\tau w^s$.  The maximum principle then yields the
decay of the time derivative.  An integrating-factor argument gives the
two-sided first-normal-derivative estimate, and the equation recovers the
second derivative.  The remaining finite-order estimates follow from the
differentiated equations.

Third, the stability proof uses the same division into an initial finite-time
interval and a large-time interval.  Growing barriers control the initial
interval.  On the large-time interval, the perturbation equations inherit the
strict-concavity term and the barriers decay at the rate $\rho_m$.  The only
coupled estimates are those for the weighted normal and time derivatives.  We
first choose the transition time large and then the perturbation small, which
closes the bootstrap.
}

{
The rest of this paper is organized as follows.  Section~2 develops the
semi-similar Crocco theory, proving global well-posedness, large-time
convergence, and higher-order regularity, and also shows that weighted edge
control may be lost in finite time when the initial time-monotonicity condition
is removed.  Section~3 proves nonlinear stability of the semi-similar flows in
Crocco variables for every fixed $m\geq1$.  Section~4 reconstructs the
corresponding Eulerian Prandtl flows and derives the Eulerian forms of the main
estimates.  Appendix~A establishes the required estimates for the steady
favorable Falkner--Skan profiles, while Appendix~B constructs critical edge
packets and proves nonuniform dependence in the weighted topology.
}

\section{Semi-similar Crocco flows}

This section establishes global existence and large-time behavior of solutions to the semi-similar equation \eqref{eq:semi-self-similar-crocco-m}.
Throughout this section, write $\widetilde w(\tau,\eta)=w^s(\tau,\eta)-\overline w_m(\eta)$.

\subsection{A priori estimates and global well-posedness}

The following propositions give a priori estimates for equation \eqref{eq:semi-self-similar-crocco-m}.
Let $w^s\in C^{1,2}_{\text{par}}([0,+\infty)\times [0,1))\cap C([0,+\infty)\times [0,1])$ be a solution of \eqref{eq:semi-self-similar-crocco-m} whose initial data $w_0^s(\eta)$ satisfy \eqref{eq:initial-data}.

\begin{proposition}\label{prop:estimate-ws}
There exists a constant $D_1>1$, depending only on $E_1$, such that
\[
 \overline w_m(\eta)\leq w^s(\tau,\eta)
 \leq D_1\overline w_m(\eta)
\]
for all $(\tau,\eta)\in[0,+\infty)\times[0,1)$.
\end{proposition}

\begin{proof}
Since \(\left.\partial_\tau w^s\right|_{\tau=0}\leq0\), the equation at
\(\tau=0\) shows that \(w_0^s\) is a supersolution of the steady problem
\eqref{eq:steady-self-similar-m}.  The one-dimensional maximum principle
therefore gives \(w_0^s\geq\overline w_m\).
The difference $\widetilde w$ satisfies
\begin{equation*}
\left\{
\begin{aligned}
 &(w^s)^2\partial_{\eta\eta}\widetilde w
 +m(\eta^2-1)\partial_\eta\widetilde w
 +\left[(w^s+\overline w_m)\overline w_m''
 -\frac{3m-1}{2}\eta\right]\widetilde w -[1+(m-1)\tau\eta]\partial_\tau\widetilde w=0,\\
 &\widetilde w|_{\tau=0} \geq0,
 \qquad \widetilde w|_{\eta=1}=0, \qquad \left.\left(\partial_\eta\widetilde w
 -\frac{m}{w^s\overline w_m}\widetilde w\right)\right|_{\eta=0}=0.
\end{aligned}
\right.
\end{equation*}
Note that the two coefficients $(w^s+\overline w_m)\overline w_m''-\frac{3m-1}{2}\eta$ and $-\frac{m}{w^s\overline w_m}$ are strictly negative for $m\geq 1$.
The maximum principle gives $\widetilde w\geq0$, that is $w^s\geq \overline w_m$.

Let $D_1>1$ and $\psi_1=D_1 \overline w_m$.
Direct substitution and the steady equation give
\[
\psi_1^2\psi_1''+m(\eta^2-1)\psi_1'
-\frac{3m-1}{2}\eta\psi_1
=D_1(D_1^2-1)\overline w_m^2\overline w_m''\leq0.
\]
Choose $D_1>1$ sufficiently large that $w_0^s\leq D_1\overline w_m$. 
By \eqref{eq:initial-data}, it may be chosen depending only on $E_1$.  
Set $\Psi_1=w^s-\psi_1$.  
It satisfies
\[
\left\{
\begin{aligned}
 &(w^s)^2\partial_{\eta\eta}\Psi_1
 +m(\eta^2-1)\partial_\eta\Psi_1
 +\left[(w^s+\psi_1)\psi_1''-\frac{3m-1}{2}\eta\right]\Psi_1
 -[1+(m-1)\tau\eta]\partial_\tau\Psi_1\geq0,\\
 &\Psi_1|_{\tau=0}=w_0^s-D_1\overline w_m\leq0,
 \qquad \Psi_1 |_{\eta=1} =0,\\
 &\partial_\eta\Psi_1 |_{\eta=0}
 =\left.\left(-\frac{m}{w^s}+\frac{D_1m}{\overline w_m}\right)\right|_{\eta=0}
 \geq \left. \frac{m\Psi_1} {D_1w^s\overline w_m}\right|_{\eta=0}.
\end{aligned}
\right.
\]
The zero-order coefficient is strictly negative.  
The maximum principle gives $\Psi_1\leq0$, and therefore
\begin{equation*}
\overline w_m\leq w^s\leq D_1\overline w_m.
\end{equation*}
\end{proof}

\begin{proposition}
\label{prop:time-monotonicity}
Assume that $|w^s\partial_{\eta\eta}w^s(\tau,\eta)|\leq M(\tau)<+\infty$.
Then
\[
 \partial_\tau w^s\leq0
\]
for all $(\tau,\eta)\in [0,+\infty)\times [0,1)$.
\end{proposition}

\begin{proof}
Differentiating \eqref{eq:semi-self-similar-crocco-m} in $\tau$ gives
\begin{equation*}
\left\{
\begin{aligned}
 &\left[(w^s)^2\partial_{\eta\eta}
 +m(\eta^2-1)\partial_\eta
 +\left(2w^s\partial_{\eta\eta}w^s-\frac{5m-3}{2}\eta\right)
 -[1+(m-1)\tau\eta]\partial_{\tau}\right](\partial_\tau w^s)=0,\\
 &\partial_\tau w^s |_{\eta=1}=0,\qquad
 \left.\left(\partial_{\tau\eta}w^s
 -\frac{m}{(w^s)^2}\partial_\tau w^s\right)\right|_{\eta=0}=0,\\
 &\left.\partial_\tau w^s\right|_{\tau=0}=(w_0^s)^2\frac{d^2w_0^s}{d\eta^2}(\eta)+m (\eta^2-1)\frac{dw_0^s}{d\eta}(\eta)-\frac{3m-1}{2}\eta w^s_0(\eta)\leq 0.
\end{aligned}
\right.
\end{equation*}
The maximum principle gives $\partial_\tau w^s\leq0$.
\end{proof}

\begin{proposition}
\label{prop:weighted-first-derivative}
Assume that $|w^s\partial_{\eta\eta}w^s(\tau,\eta)|\leq M(\tau)<+\infty$.
Then there exists a constant $D_2>0$, depending only on $m,E_1,E_2$, such that
\[
 D_2e^{2m\tau}\overline w_m'(\eta)
 \leq\partial_\eta w^s(\tau,\eta)
 \leq-D_2\overline w_m'(\eta)
\]
for all $(\tau,\eta)\in [0,+\infty)\times [0,1)$.
\end{proposition}

\begin{proof}
For the lower bound, apply
\eqref{eq:semi-self-similar-crocco-m} to
$w^s\partial_\eta w^s$ and use the same equation to eliminate
$\partial_\tau w^s$.  This gives
\begin{equation*}
\begin{aligned}
 &(w^s)^2\partial_{\eta\eta}(w^s\partial_\eta w^s)
 +\left[m(\eta^2-1)
 -\frac{(m-1)\tau(w^s)^2}{1+(m-1)\tau\eta}\right]
 \partial_\eta(w^s\partial_\eta w^s)\\
 &\quad-[1+(m-1)\tau\eta]
 \partial_\tau(w^s\partial_\eta w^s)-\left[(m-1)\eta
 +\frac{m(m-1)\tau(\eta^2-1)}{1+(m-1)\tau\eta}\right]
 w^s\partial_\eta w^s\\
 &=\frac{3m-1}{2[1+(m-1)\tau\eta]}(w^s)^2
 -\frac{(m-1)\tau}{1+(m-1)\tau\eta}
 (w^s\partial_\eta w^s)^2.
\end{aligned}
\end{equation*}
The corresponding boundary conditions are
\[
 (w^s\partial_\eta w^s)|_{\eta=0}=-m,
 \qquad (w^s\partial_\eta w^s)|_{\eta=1}=0.
\]
The initial assumptions also give
\[
 |w_0^s(w_0^s)'|\leq C(1-\eta)\sigma_m^2(\eta).
\]

Let
\[
 \psi_2=D_2e^{2m\tau}(1-\eta)\sigma_m^2(\eta),
 \qquad
 \Psi_2=w^s\partial_\eta w^s+\psi_2,
\]
where $D_2>0$ will be fixed below.  Substitution of
$\Psi_2$ into the equation for
$w^s\partial_\eta w^s$ gives
\begin{align*}
  &\Bigg[(w^s)^2\partial_{\eta\eta}
 +\left(m(\eta^2-1)
 -\frac{(m-1)\tau(w^s)^2}{1+(m-1)\tau\eta}\right)\partial_\eta
 -[1+(m-1)\tau\eta]\partial_\tau\\
&\qquad -(m-1)\eta
 -\frac{(m-1)\tau}{1+(m-1)\tau\eta}
 \left[m(\eta^2-1)+2\psi_2\right]
 \Bigg]\Psi_2\\
=&\Big[(w^s)^2\partial_{\eta\eta}
 +m(\eta^2-1)\partial_\eta
 -[1+(m-1)\tau\eta]\partial_\tau-(m-1)\eta\Big]\psi_2
+\frac{3m-1}{2[1+(m-1)\tau\eta]}(w^s)^2\\
&-\frac{(m-1)\tau}{1+(m-1)\tau\eta}
 \left[(w^s)^2\partial_\eta\psi_2
 +m(\eta^2-1)\psi_2+\psi_2^2\right]
-\frac{(m-1)\tau}{1+(m-1)\tau\eta}\Psi_2^2.
\end{align*}
The terms on the right-hand side are estimated separately. 
\begin{align*}
&\Big[(w^s)^2\partial_{\eta\eta}
 +m(\eta^2-1)\partial_\eta
 -[1+(m-1)\tau\eta]\partial_\tau-(m-1)\eta\Big]\psi_2+\frac{3m-1}{2[1+(m-1)\tau\eta]}(w^s)^2\\
=&-\frac{D_2e^{2m\tau}}{1-\eta}(w^s)^2
 +mD_2e^{2m\tau}(\eta^2-1)(1-\sigma_m^2)\\
&\qquad-2m[1+(m-1)\tau\eta]\psi_2-(m-1)\eta\psi_2
 +\frac{3m-1}{2[1+(m-1)\tau\eta]}(w^s)^2\\
\leq&\left[-M_1^2
 +m(1+\eta)\left(1-\frac1{\sigma_m^2}\right)
 -2m[1+(m-1)\tau\eta]-(m-1)\eta
 +\frac{(3m-1)D_1^2M_2^2}{2D_2}\right]\psi_2\\
 \leq &-\frac{M_1^2}{2}\psi_2<0
\end{align*}
for sufficiently large $D_2$.  
Moreover, for sufficiently large $D_2$,
\begin{align*}
 &(w^s)^2\partial_\eta\psi_2
 +m(\eta^2-1)\psi_2+\psi_2^2\\
=&\psi_2\left[
 (w^s)^2\frac{1-\sigma_m^2}
 {(1-\eta)\sigma_m^2}
 +m(\eta^2-1)+\psi_2\right]\\
\geq & \psi_2(1-\eta)\sigma_m^2
 \bigl(D_2e^{2m\tau}-C\bigr)\geq0.
\end{align*}
 Hence the entire right-hand side is nonpositive.
Increasing $D_2$ once more,
\begin{align*}
 2\psi_2+m(\eta^2-1)&=(1-\eta)\bigl[2D_2e^{2m\tau}\sigma_m^2-m(1+\eta)\bigr]\\
 &\geq2(1-\eta)\bigl[D_2\sigma_m^2(0)-m\bigr]\geq0.
\end{align*}
Then the zero-order coefficient on the left-hand side of that identity is nonpositive.
For the same choice of $D_2$, the initial and boundary values satisfy
\begin{equation*}
\left\{
\begin{aligned}
&\Psi_2|_{\tau=0}
 =w_0^s(w_0^s)'+D_2(1-\eta)\sigma_m^2(\eta)\geq0,\\
&\Psi_2|_{\eta=1}=0,
 \qquad
 \Psi_2|_{\eta=0}
 =-m+D_2 e^{2m\tau}\sigma_m^2(0)\geq0.
\end{aligned}
\right.
\end{equation*}
The maximum principle gives $-\psi_2\leq w^s\partial_\eta w^s$.  Hence
\[
 \partial_\eta w^s
 \geq-\frac{D_2}{M_1}e^{2m\tau}\sigma_m
 \geq\frac{D_2}{M_1M_4}e^{2m\tau}\overline w_m'.
\]
After relabeling $D_2$,
\begin{equation}\label{eq:weighted-gradient-lower-m}
 \partial_\eta w^s
 \geq D_2e^{2m\tau}\overline w_m'.
\end{equation}

For the upper bound, fix $\tau$ and set
{
\[
 Z(\eta)=(1-\eta^2)\frac{\partial_\eta w^s}{w^s}.
\]
A direct differentiation, followed by the use of
\eqref{eq:semi-self-similar-crocco-m}, gives
\begin{align*}
(1-\eta^2)Z'=&-Z^2
+\left[-2\eta+\frac{m(1-\eta^2)^2}{(w^s)^2}\right]
Z\\
&+\frac{(1-\eta^2)^2}{(w^s)^2}
\left[\frac{3m-1}{2}\eta
+\frac{[1+(m-1)\tau\eta]\partial_\tau w^s}{w^s}\right].
\end{align*}
By Proposition~\ref{prop:estimate-ws},
\eqref{eq:steady-self-similar-properties-m}, and
Proposition~\ref{prop:time-monotonicity},
\[
 \frac{3m-1}{2}\eta
 +\frac{[1+(m-1)\tau\eta]\partial_\tau w^s}{w^s}
 \leq\frac{3m-1}{2}\eta.
\]
Consequently,
\begin{equation*}
 (1-\eta^2)Z'
 \leq -Z^2
 +\left(2+\frac{4m}{M_1^2\sigma_m^2(0)}\right)|Z|
 +\frac{2(3m-1)}{M_1^2\sigma_m^2(0)}.
\end{equation*}
At the wall,
\[
Z(0)=\left.\frac{\partial_\eta w^s}{w^s}\right|_{\eta=0}
 =-\frac{m}{w^s(\tau,0)^2}<0,
\]
and its absolute value is uniformly bounded by Proposition~\ref{prop:estimate-ws}.  
Choose $D_2$ so large that the right-hand side of the differential inequality for $Z$ is negative for $Z\geq D_2$ and that $Z(0)<D_2$.
The function $Z$ cannot cross the level $D_2$ from below: at a first crossing one would have $Z'\geq0$, whereas $Z'<0$. 
Thus $Z\leq D_2$.
After increasing $D_2$ by a factor depending only on $m,E_1$, the profile comparison gives
\begin{equation}\label{eq:weighted-gradient-upper-m}
 \partial_\eta w^s \leq-D_2\overline w_m'
\end{equation}
}
on $[0,+\infty)\times[0,1)$.  Increasing $D_2$ once more to combine
\eqref{eq:weighted-gradient-lower-m} and
\eqref{eq:weighted-gradient-upper-m} gives the asserted two-sided estimate.
\end{proof}

\begin{proposition}
\label{prop:time-derivative-bound}

Assume that $|w^s\partial_{\eta\eta}w^s(\tau,\eta)|\leq M(\tau)<+\infty$.
Then there exists a constant $D_3>0$, depending only on $m,E_1,E_2,E_3$, such that
\[
 -D_3w^s\leq\partial_\tau w^s\leq0
\]
for all $(\tau,\eta)\in [0,+\infty)\times [0,1)$.
\end{proposition}

\begin{proof}
Using \eqref{eq:semi-self-similar-crocco-m} to eliminate
$\partial_{\eta\eta}w^s$ from the equation for $\partial_\tau w^s$ in the
proof of Proposition~\ref{prop:time-monotonicity} gives
\begin{equation*}
\left\{
\begin{aligned}
&\Bigg[(w^s)^2\partial_{\eta\eta}
+m(\eta^2-1)\partial_\eta
-[1+(m-1)\tau\eta]\partial_\tau
\Bigg](\partial_\tau w^s)\\
&\qquad
 +\left[-2m(\eta^2-1)\frac{\partial_\eta w^s}{w^s}
 +\frac{m+1}{2}\eta\right]\partial_\tau w^s
=-\frac{2[1+(m-1)\tau\eta]}{w^s}
 \left(\partial_\tau w^s\right)^2\leq0,\\
&\partial_\tau w^s|_{\eta=1}=0,\qquad
\left.\left(\partial_{\tau\eta}w^s
-\frac{m}{(w^s)^2}\partial_\tau w^s\right)\right|_{\eta=0}=0,\\
& \left. \left|\partial_\tau w^s\right| \right|_{\tau=0} \leq C(1-\eta)\sigma_m(\eta).
\end{aligned}
\right.
\end{equation*}

 Set $\Psi_3=\partial_\tau w^s+D_3w^s$.  Substitution into the displayed
 equation for $\partial_\tau w^s$, together with
 \eqref{eq:semi-self-similar-crocco-m}, gives
 \begin{align*}
 &\Big[(w^s)^2\partial_{\eta\eta}
 +m(\eta^2-1)\partial_\eta
 -[1+(m-1)\tau\eta]\partial_\tau\Big]\Psi_3\\
 &\qquad+\left[-2m(\eta^2-1)\frac{\partial_\eta w^s}{w^s}
 +\frac{m+1}{2}\eta
 -4D_3[1+(m-1)\tau\eta]\right]\Psi_3\\
 =&-\frac{2[1+(m-1)\tau\eta]}{w^s}\Psi_3^2
 +2mD_3\big[(1-\eta^2)\partial_\eta w^s+\eta w^s\big]-2D_3^2[1+(m-1)\tau\eta]w^s\\
\leq & 2 D_3 w^s\left[m\left(1+\frac{2M_3D_2}{M_1}\right)-D_3\right]\leq 0
 \end{align*}
after increasing $D_3$.  
The same choice makes the zero-order coefficient on the left-hand side of that identity nonpositive.
Increasing $D_3$ once more, the initial data give $\Psi_3|_{\tau=0}\geq0$.  Moreover,
 \[
 \left\{
 \begin{aligned}
  &\Psi_3|_{\eta=1}=0,\\
  &\left.\left(\partial_\eta\Psi_3
  -\frac{m}{(w^s)^2}\Psi_3\right)\right|_{\eta=0}
  =-\frac{2mD_3}{w^s(\tau,0)}<0.
 \end{aligned}
 \right.
 \]
 The maximum principle gives $\Psi_3\geq0$.  Together with
 Proposition~\ref{prop:time-monotonicity}, this proves
 \[
  -D_3w^s\leq\partial_\tau w^s\leq0.
 \]
\end{proof}

\subsection*{Proof of the well-posedness assertion in Theorem~\ref{thm:ws-regularity}}
\begin{proof}
Set
\[
 \mathcal D=2 \left(D_3+\frac{2mM_3D_2}{M_1} +\frac{3m-1}{2}\right)
\]
and assume that 
\begin{equation*}
 \sup_{(\tau,\eta)\in[0,T]\times[0,1)}
 \frac{e^{-2m\tau}|w^s\partial_{\eta\eta}w^s(\tau,\eta)|}
 {1+(m-1)\tau}
 \leq\mathcal D.
\end{equation*}
Equation~\eqref{eq:semi-self-similar-crocco-m} gives
\begin{align*}
 w^s\partial_{\eta\eta}w^s
 ={}&[1+(m-1)\tau\eta]\frac{\partial_\tau w^s}{w^s}
 -m(\eta^2-1)\frac{\partial_\eta w^s}{w^s}
 +\frac{3m-1}{2}\eta.
\end{align*}
Propositions~\ref{prop:estimate-ws},~\ref{prop:weighted-first-derivative}, and~\ref{prop:time-derivative-bound} give
\[
 |w^s\partial_{\eta\eta}w^s(\tau,\eta)|
 \leq D_3 [1+(m-1)\tau]
 +\frac{2mM_3D_2}{M_1}e^{2m\tau}+\frac{3m-1}{2}
 \leq\frac{\mathcal D}{2} [1+(m-1)\tau]e^{2m\tau}
\]
for every $(\tau,\eta)\in[0,T]\times[0,1)$.  Hence
\[
 \sup_{(\tau,\eta)\in[0,T]\times[0,1)}
 \frac{e^{-2m\tau}|w^s\partial_{\eta\eta}w^s(\tau,\eta)|}
 {1+(m-1)\tau}
 \leq\frac{\mathcal D}{2}.
\]
{
Local existence and continuation are obtained by adapting the
argument of Oleinik--Samokhin~\cite[Section~4.4]{OS}, rather than by directly
invoking a theorem stated there.  Their regularization and compactness scheme,
together with the maximum-principle estimates, applies to
\eqref{eq:semi-self-similar-crocco-m} after replacing the coefficients and the
boundary operator by the present ones.  Positivity of \(w^s\), the wall
compatibility condition, and the weighted bounds above provide the uniform
estimates needed in this adaptation.  Continuity and the strict bootstrap
improvement then show that the bound cannot fail at a finite time.  The
corresponding continuation argument therefore yields a solution
\[
 w^s\in C^{1,2}_{\text{par}}([0,+\infty)\times[0,1))
\]
of \eqref{eq:semi-self-similar-crocco-m}.
}

It remains to recover the estimates in \eqref{eq:ws-estimates-all}.  The
profile and first-derivative estimates follow from
Propositions~\ref{prop:estimate-ws} and~
\ref{prop:weighted-first-derivative}.  Proposition~\ref{prop:time-derivative-bound}
and $w^s\leq D_1\overline w_m$ give
\[
 -D_1D_3\overline w_m(\eta)
 \leq\partial_\tau w^s(\tau,\eta)\leq0.
\]
Moreover, the estimate for $w^s\partial_{\eta\eta}w^s$ obtained from
Propositions~\ref{prop:estimate-ws},~\ref{prop:weighted-first-derivative},
and~\ref{prop:time-derivative-bound}, together with
$w^s\geq\overline w_m$ and
$-\overline w_m\overline w_m''\geq M_6$, gives
\[
 |\partial_{\eta\eta}w^s(\tau,\eta)|
 \leq-\frac{\mathcal D}{2M_6}
 [1+(m-1)\tau]e^{2m\tau}
 \overline w_m''(\eta).
\]
Thus \eqref{eq:ws-estimates-all} follows after enlarging $C$.

It remains to prove uniqueness in the positive classical class satisfying the
preceding weighted bounds.
If $w_1^s,w_2^s$ are two solutions, their difference satisfies
\begin{align*}
[1+(m-1)\tau\eta]\partial_\tau(w_1^s-w_2^s)
={}&(w_1^s)^2\partial_{\eta\eta}(w_1^s-w_2^s)
+m(\eta^2-1)\partial_\eta(w_1^s-w_2^s)\\
&+\left[(w_1^s+w_2^s)\partial_{\eta\eta}w_2^s
-\frac{3m-1}{2}\eta\right](w_1^s-w_2^s),
\end{align*}
with
\[
\left.\left[\partial_\eta(w_1^s-w_2^s)-\frac{m(w_1^s-w_2^s)}{w_1^sw_2^s}\right]\right|_{\eta=0}=0,
\qquad
  (w_1^s-w_2^s)|_{\eta=1}=0.
\]
The preceding estimates imply that the positive part of
\[
 (w_1^s+w_2^s)\partial_{\eta\eta}w_2^s
 -\frac{3m-1}{2}\eta
\]
is bounded on $[0,T]\times[0,1)$. 
Choosing $\Lambda>0$ sufficiently large and applying the maximum principle to $e^{-\Lambda\tau}(w_1^s-w_2^s)$ and $e^{-\Lambda\tau}(w_2^s-w_1^s)$ gives $w_1^s=w_2^s$ on $[0,T]$.

Thus the solution is global and unique in the positive classical class, which proves the global well-posedness assertion in Theorem~\ref{thm:ws-regularity}.
\end{proof}

\subsection{Large-time behavior}

The global semi-similar solution converges to the steady Falkner--Skan profile as time tends to infinity.

\begin{proposition}
\label{prop:zero-order-convergence}
Let $w^s$ be the global solution of \eqref{eq:semi-self-similar-crocco-m} whose initial data satisfy \eqref{eq:initial-data}. 
Then there exist positive constants $D_4=D_4(m,E_1)$ and $\beta_m=\beta_m(m)$ such that
\[
 \left|w^s(\tau,\eta)-\overline w_m(\eta)\right| \leq D_4\rho_m(\tau)\overline w_m(\eta)
\]
for all $(\tau,\eta)\in[0,+\infty)\times[0,1)$.
\end{proposition}

\begin{proof}
For $\widetilde w=w^s-\overline w_m$, one has
\begin{equation*}
\left\{
\begin{aligned}
 &(w^s)^2\partial_{\eta\eta}\widetilde w
 +m(\eta^2-1)\partial_\eta\widetilde w
 +\left[(w^s+\overline w_m)\overline w_m''
 -\frac{3m-1}{2}\eta\right]\widetilde w
 -[1+(m-1)\tau\eta]\partial_\tau\widetilde w=0,\\
 &\widetilde w|_{\eta=1}=0,\qquad
 \left.\left(\partial_\eta\widetilde w
 -\frac{m}{w^s\overline w_m}\widetilde w\right)\right|_{\eta=0}=0,\\
 &0\leq\widetilde w |_{\tau=0}\leq D_4\overline w_m(\eta),
\end{aligned}
\right.
\end{equation*}
where $D_4\geq E_1-1$.
Let $\psi_4=D_4\rho_m(\tau)\overline w_m(\eta)$ and choose $\beta_m>0$ so small that
\[
 0\leq-[1+(m-1)\tau\eta]
 \frac{\partial_\tau\rho_m(\tau)}{\rho_m(\tau)}\leq M_6.
\]
A direct calculation gives
\begin{align*}
 &(w^s)^2\partial_{\eta\eta}\psi_4
 +m(\eta^2-1)\partial_\eta\psi_4
 +\left[(w^s+\overline w_m)\overline w_m''
       -\frac{3m-1}{2}\eta\right]\psi_4
       -[1+(m-1)\tau\eta]\partial_\tau\psi_4\nonumber\\
 =&D_4\rho_m\left[
 w^s(w^s+\overline w_m)\overline w_m''
 -[1+(m-1)\tau\eta]
  \frac{\partial_\tau\rho_m}{\rho_m}\overline w_m
 \right]\\
\leq & \psi_4\left[2\overline w_m\overline w_m''-[1+(m-1)\tau\eta]
  \frac{\partial_\tau\rho_m}{\rho_m}\right]\\
\leq & \psi_4 (-2M_6+M_6) \leq 0.
\end{align*}
Set $\Psi_4=\widetilde w-\psi_4$.  Then
\begin{equation*}
\left\{
\begin{aligned}
 &(w^s)^2\partial_{\eta\eta}\Psi_4
 +m(\eta^2-1)\partial_\eta\Psi_4
 +\left[(w^s+\overline w_m)\overline w_m''
 -\frac{3m-1}{2}\eta\right]\Psi_4
 -[1+(m-1)\tau\eta]\partial_\tau\Psi_4\geq0,\\
 &\Psi_4|_{\eta=1}=0,\qquad \Psi_4 |_{\tau=0} \leq 0,\\
 &\left.\left(\partial_\eta\Psi_4
 -\frac{m}{w^s\overline w_m}\Psi_4\right)\right|_{\eta=0}
 =\left.D_4m\rho_m(\tau)
 \left(\frac1{w^s}+\frac1{\overline w_m}\right)\right|_{\eta=0}>0.\\
\end{aligned}
\right.
\end{equation*}
The maximum principle gives $\Psi_4\leq0$, that is $\widetilde w \leq \psi_4$. 
Combining this with $w^s\geq\overline w_m$ gives
\begin{equation}\label{eq:m-zero-decay}
\left|\widetilde w(\tau,\eta)\right| \leq D_4\rho_m(\tau)\overline w_m(\eta).
\end{equation}
\end{proof}

The following one-sided derivative estimate is used to obtain eventual strict concavity.

\begin{proposition}[One-sided gradient estimate]
\label{prop:riccati-gradient}
Let $w^s$ be the global solution of \eqref{eq:semi-self-similar-crocco-m} whose initial data satisfy \eqref{eq:initial-data}. 
Then there exist positive constants $D_5=D_5(m,E_1)$ and $\beta_m=\beta_m(m)$ such that
\begin{equation}\label{eq:m-riccati-gradient}
 \partial_\eta\bigl(w^s(\tau,\eta)-\overline w_m(\eta)\bigr)
 \leq-D_5\rho_m(\tau)\overline w_m'(\eta)
\end{equation}
for all $(\tau,\eta)\in[0,+\infty)\times[0,1)$.
\end{proposition}

\begin{proof}
Propositions~\ref{prop:estimate-ws} and~\ref{prop:time-monotonicity} give
$\widetilde w\geq0$ and $\partial_\tau\widetilde w=\partial_\tau w^s\leq0$.
Subtracting the
steady equation from \eqref{eq:semi-self-similar-crocco-m} gives
\[
 [1+(m-1)\tau\eta]\partial_\tau\widetilde w=(w^s)^2\partial_{\eta\eta}\widetilde w
 +m(\eta^2-1)\partial_\eta\widetilde w
 +\left[(w^s+\overline w_m)\overline w_m''
 -\frac{3m-1}{2}\eta\right]\widetilde w.
\]
Consequently,
\begin{equation}\label{eq:riccati-spatial-inequality}
 (w^s)^2\partial_{\eta\eta}\widetilde w
 +m(\eta^2-1)\partial_\eta\widetilde w
 +\left[(w^s+\overline w_m)\overline w_m''
 -\frac{3m-1}{2}\eta\right]\widetilde w\leq0.
\end{equation}

Choose $K>0$ below and put
\[
 Y=\partial_\eta\widetilde w-\frac{K}{1-\eta}\widetilde w.
\]
Then
\[
 \partial_{\eta\eta}\widetilde w
 =\partial_\eta Y+\frac{K}{1-\eta}Y
 +\frac{K(K+1)}{(1-\eta)^2}\widetilde w.
\]
Substitution in \eqref{eq:riccati-spatial-inequality} gives
\begin{align}
 &(w^s)^2\partial_\eta Y
 +\left[\frac{K(w^s)^2}{1-\eta}
       +m(\eta^2-1)\right]Y\nonumber\\
 &\quad+\left[
 \frac{K(K+1)(w^s)^2}{(1-\eta)^2}
 -mK(1+\eta)
 +(w^s+\overline w_m)\overline w_m''
 -\frac{3m-1}{2}\eta\right]\widetilde w\leq0.
 \label{eq:riccati-Y-inequality}
\end{align}
By Proposition~\ref{prop:estimate-ws} and
\eqref{eq:steady-self-similar-properties-m},
\[
 \frac{(w^s)^2}{(1-\eta)^2}
 \geq M_1^2\sigma_m(0)^2,
 \qquad
 \left|(w^s+\overline w_m)\overline w_m''\right|
 \leq(D_1+1)M_5.
\]
Choose $K$, depending only on $m$ and $E_1$, so large that
\[
 M_1^2\sigma_m(0)^2K(K+1)-2mK-(D_1+1)M_5
 -\frac{3m-1}{2}>0.
\]
Then the coefficient of $\widetilde w$ in
\eqref{eq:riccati-Y-inequality} is strictly
positive for every $\tau\geq0$ and $0\leq\eta<1$.

The wall conditions give
\[
 \partial_\eta\widetilde w(\tau,0)
 =\frac{m}{w^s(\tau,0)\overline w_m(0)}\widetilde w(\tau,0).
\]
Since $w^s(\tau,0)\geq\overline w_m(0)$, increasing $K$ if necessary so
that $K\geq m/\overline w_m(0)^2$ yields $Y(\tau,0)\leq0$.
Because $\widetilde w\geq0$, \eqref{eq:riccati-Y-inequality} now implies
\[
 \partial_\eta Y+
 \left[\frac{K}{1-\eta}
 +\frac{m(\eta^2-1)}{(w^s)^2}\right]Y\leq0.
\]
Fix $\tau\geq0$ and $0<r<1$.  Multiplication by the positive integrating
factor
\[
 \exp\left\{\int_0^\eta
 \left[\frac{K}{1-s}
 +\frac{m(s^2-1)}{(w^s(\tau,s))^2}\right]\,ds\right\}
\]
shows that $Y(\tau,\eta)\leq0$ for $0\leq\eta\leq r$.  Since $r<1$ is
arbitrary,
\[
 \partial_\eta\widetilde w(\tau,\eta)
 \leq\frac{K}{1-\eta}\widetilde w(\tau,\eta),
 \qquad 0\leq\eta<1.
\]
Equations~\eqref{eq:m-zero-decay} and
\eqref{eq:steady-self-similar-properties-m} give
\[
 \partial_\eta\widetilde w(\tau,\eta)
 \leq\frac{KD_4\rho_m(\tau)}{1-\eta}\overline w_m(\eta)
 \leq-\frac{KD_4M_2}{M_4}\rho_m(\tau)\overline w_m'(\eta).
\]
Increasing $D_5$ proves \eqref{eq:m-riccati-gradient}.
\end{proof}

{
\begin{proposition}[Dynamically generated strict concavity]
\label{prop:dynamical-strict-concavity}
Let $w^s$ be the global solution of
\eqref{eq:semi-self-similar-crocco-m} whose initial data satisfy
\eqref{eq:initial-data}.  There exist
$\tau_0=\tau_0(m,E_1,E_2,E_3)>0$ and $C_*=C_*(m)>0$ such that
\begin{equation}
 w^s(\tau,\eta)\partial_{\eta\eta}w^s(\tau,\eta)
 \leq-C_*,
 \qquad (\tau,\eta)\in[\tau_0,+\infty)\times[0,1).
 \label{eq:m-strict-concavity}
\end{equation}
In particular, strict concavity is not assumed at the initial time; it is
produced by the time-monotone evolution after a finite transient.
\end{proposition}

\begin{proof}
Proposition~\ref{prop:time-monotonicity} and the equation for $w^s$ give
\begin{align*}
 (w^s)^2\partial_{\eta\eta}w^s
 &\leq m(1-\eta^2)\partial_\eta w^s
 +\frac{3m-1}{2}\eta w^s\\
 &=\overline w_m^2\overline w_m''
 +m(1-\eta^2)\partial_\eta\widetilde w
 +\frac{3m-1}{2}\eta\widetilde w\\
 &\leq\overline w_m^2\overline w_m''
 +C\rho_m(\tau)\overline w_m.
\end{align*}
Here the last inequality follows from
Propositions~\ref{prop:zero-order-convergence} and~
\ref{prop:riccati-gradient}.  Taking $\tau_0$ sufficiently large gives
\[
 (w^s)^2\partial_{\eta\eta}w^s
 \leq-\frac{M_6}{2}\overline w_m.
\]
After increasing $\tau_0$, Proposition~\ref{prop:zero-order-convergence}
also gives $w^s\leq2\overline w_m$.  Hence
\eqref{eq:m-strict-concavity} holds with $C_*=M_6/4$.
\end{proof}
}

\begin{proposition}
\label{prop:optimal-derivative-convergence}
Let $w^s$ be the global solution of \eqref{eq:semi-self-similar-crocco-m} whose initial data satisfy \eqref{eq:initial-data}. 
Then there exist positive constants $D_6=D_6(m,E_1,E_2,E_3)$ and $\beta_m=\beta_m(m)$ such that
\begin{align}\label{eq:m-optimal-derivative-convergence}
 &[1+(m-1)\tau]|\partial_\tau w^s(\tau,\eta)|
 \leq D_6\rho_m(\tau)\overline w_m(\eta),\nonumber\\
 &|\partial_\eta w^s(\tau,\eta)-\overline w_m'(\eta)|
 \leq-D_6\rho_m(\tau)\overline w_m'(\eta),\nonumber\\
 &|\partial_{\eta\eta}w^s(\tau,\eta)-\overline w_m''(\eta)|
 \leq-D_6\rho_m(\tau)\overline w_m''(\eta)
\end{align}
for all $(\tau,\eta)\in[0,+\infty)\times[0,1)$.
\end{proposition}

\begin{proof}
{
Let $\tau_0$ and $C_*$ be given by
Proposition~\ref{prop:dynamical-strict-concavity}.  We use
\eqref{eq:m-strict-concavity} to obtain the decay of the time derivative.
}

Differentiating the equation for \(w^s\) with respect to \(\tau\) gives
\begin{equation*}
\left\{
\begin{aligned}
 &\left[(w^s)^2\partial_{\eta\eta}
 +m(\eta^2-1)\partial_\eta
 +\left(2w^s\partial_{\eta\eta}w^s-\frac{5m-3}{2}\eta\right)
 -[1+(m-1)\tau\eta]\partial_{\tau}\right](\partial_\tau w^s)=0,\\
 &\partial_\tau w^s |_{\eta=1}=0,\qquad
 \left.\left(\partial_{\tau\eta}w^s
 -\frac{m}{(w^s)^2}\partial_\tau w^s\right)\right|_{\eta=0}=0,\\
 &\left.\partial_\tau w^s\right|_{\tau=0}=(w_0^s)^2\frac{d^2w_0^s}{d\eta^2}(\eta)+m (\eta^2-1)\frac{dw_0^s}{d\eta}(\eta)-\frac{3m-1}{2}\eta w^s_0(\eta)\leq 0.
\end{aligned}
\right.
\end{equation*}

Let
\[
 \psi_6=\frac{D_6\rho_m(\tau)}{1+(m-1)\tau}\overline w_m(\eta)
\]
and choose \(\beta_m>0\) such that for \(\tau\geq\tau_0\),
\[
 \frac{m-1}{1+(m-1)\tau}
 -[1+(m-1)\tau\eta]
 \frac{\partial_\tau\rho_m(\tau)}{\rho_m(\tau)}
 \leq C_*.
\]
A direct calculation gives
\begin{align*}
 &(w^s)^2\partial_{\eta\eta}\psi_6+m(\eta^2-1)\partial_\eta\psi_6-[1+(m-1)\tau\eta]\partial_\tau\psi_6\\
 &+\left(2w^s\partial_{\eta\eta}w^s-\frac{5m-3}{2}\eta\right)\psi_6\\
 =&\frac{D_6\rho_m}{1+(m-1)\tau}\bigl[
 \bigl((w^s)^2-\overline w_m^2\bigr)\overline w_m''
 +2w^s\partial_{\eta\eta}w^s\overline w_m\\
 &+\frac{(m-1)(1-\eta)}{1+(m-1)\tau}\overline w_m
 -[1+(m-1)\tau\eta]
 \frac{\partial_\tau\rho_m}{\rho_m}\overline w_m
 \bigr]\\
 \leq & \psi_6 (-2C_*+C_*)\leq0.
\end{align*}
Take \(D_6\) sufficiently large and set
\[
 \Psi_6=\partial_\tau w^s+\psi_6.
\]
Then
\begin{equation*}
\left\{
\begin{aligned}
 &(w^s)^2\partial_{\eta\eta}\Psi_6
 +m(\eta^2-1)\partial_\eta\Psi_6
 -[1+(m-1)\tau\eta]\partial_\tau\Psi_6 +\left[
 2w^s\partial_{\eta\eta}w^s-\frac{5m-3}{2}\eta
 \right]\Psi_6\leq0,\\
 &\Psi_6 |_{\eta=1}=0,\\
 &\left.\left(\partial_\eta\Psi_6-\frac{m}{(w^s)^2}\Psi_6\right)
 \right|_{\eta=0}
 =-\left.\frac{D_6m\rho_m}{1+(m-1)\tau}
 \left(\frac1{\overline w_m}
 +\frac{\overline w_m}{(w^s)^2}\right)\right|_{\eta=0}<0,\\
 &\Psi_6|_{\tau=\tau_0} \geq0.
\end{aligned}
\right.
\end{equation*}
The maximum principle gives \(\Psi_6\geq0\).  Enlarging \(D_6\) on
\(0\leq\tau\leq\tau_0\) gives
\begin{equation}\label{eq:m-time-decay}
 [1+(m-1)\tau]|\partial_\tau w^s(\tau,\eta)|
 \leq D_6\rho_m(\tau)\overline w_m(\eta).
\end{equation}

It suffices to prove the remaining estimates for \(\tau\geq\tau_0\), since
on \(0\leq\tau\leq\tau_0\) they follow from the a priori estimates above
after increasing \(D_6\); here \(\rho_m\) is bounded from below by a positive
constant.  The difference equation can be written as
\begin{equation}\label{eq:m-first-derivative-equation}
 (w^s)^2\partial_{\eta\eta}\widetilde w
 +m(\eta^2-1)\partial_\eta\widetilde w=F,
\end{equation}
where
\[
 F=[1+(m-1)\tau\eta]\partial_\tau w^s
 +\frac{3m-1}{2}\eta\widetilde w
 -(w^s+\overline w_m)\overline w_m''\widetilde w.
\]
By Proposition~\ref{prop:zero-order-convergence} and
\eqref{eq:m-time-decay},
\[
 |F(\tau,\eta)|\leq C\rho_m(\tau)\overline w_m(\eta).
\]
Rewrite \eqref{eq:m-first-derivative-equation} as
\begin{equation}\label{eq:ode-pre}
    \partial_\eta (\varrho  \partial_\eta \widetilde{w})= A F,
\end{equation}
where
\begin{equation*}
  A=\frac{\varrho}{(w^s)^2}, \qquad \varrho=\exp\left(\int_{0}^{\eta} \frac{m(\zeta^2-1)}{(w^s)^2}\,d\zeta \right).
\end{equation*}
By increasing \(\tau_0\) if necessary, there exist \(\eta_*\in(0,1)\) and \(c>0\), independent of \(\tau\geq\tau_0\), such that
\begin{align*}
  -\partial_\eta(\varrho\sigma_m)&=\frac{\varrho}{(1-\eta)\sigma_m} \left(\frac{m(1+\eta)(1-\eta)^2\sigma_m^2}{(w^s)^2}-\frac{1}{2}\right)\\
  &\geq c\, \frac{\varrho}{(1-\eta)\sigma_m}
\end{align*}
for $\eta_*\leq\eta<1$. 
On the compact interval \(0\leq\eta\leq\eta_*\), the profile comparison implies that \(\varrho\), \(\varrho^{-1}\), and \(A\) are uniformly bounded.  The endpoint bounds in Proposition~\ref{prop:weighted-first-derivative} give
\[
 \lim_{\eta\to1^-}\varrho\,\partial_\eta\widetilde w=0.
\]
Splitting the integral at \(\eta_*\), using the differential inequality above
on \([\eta_*,1)\), using the uniform bounds on \([0,\eta_*]\), and integrating
\eqref{eq:ode-pre}, we obtain
\begin{align*}
  \left|\varrho \partial_\eta \widetilde{w} \right|&\leq  \int_{\eta}^{1} \left|A F\right|\,d\zeta\\
  &\leq C \rho_m(\tau) \int_{\eta}^{1} \frac{\varrho}{(1-\zeta)\sigma_m(\zeta)}\,d\zeta \\
  &\leq C \rho_m(\tau) \varrho \sigma_m,
\end{align*}
which implies
\begin{equation}\label{eq:derivative-space-decay}
        \left|\partial_\eta \widetilde{w}\right| \leq C \sigma_m \rho_m(\tau)\leq -D_6 \rho_m(\tau) \overline w_m'(\eta).
    \end{equation}
Finally, \eqref{eq:m-first-derivative-equation} yields
    \begin{equation*}
        \left|\partial_{\eta\eta}\widetilde{w}\right|=\left|\frac{F+m(1-\eta^2)\partial_\eta \widetilde{w}}{(w^s)^2}\right|\leq -D_6\rho_m(\tau) \overline w_m''(\eta).
    \end{equation*}
\end{proof}

\subsection*{Proof of the large-time behavior in Theorem~\ref{thm:ws-regularity}}

\begin{proof}
The large-time behavior \eqref{eq:ws-large-time} follows directly from Proposition~\ref{prop:zero-order-convergence} and Proposition~\ref{prop:optimal-derivative-convergence}.  The decay rate \(\beta_m\) is taken to be the minimum of the rates chosen in Propositions~\ref{prop:zero-order-convergence},~\ref{prop:riccati-gradient}, and~\ref{prop:optimal-derivative-convergence}.
\end{proof}

\subsection{Proof of Theorem~\ref{thm:higher-order-regularity}}

\begin{proof}
\medskip\noindent\textbf{Mixed time-normal derivative.}
For \(\partial_\tau(w^s\partial_\eta w^s)\), direct differentiation and
elimination of \(\partial_{\tau\tau}w^s\) through the equation for
\(\partial_\tau w^s\) give
\begin{equation*}
\left\{
\begin{aligned}
&\Bigg[(w^s)^2\partial_{\eta\eta}
+m(\eta^2-1)\partial_\eta -\frac{(m-1)\tau(w^s)^2}{1+(m-1)\tau\eta}\partial_\eta
-[1+(m-1)\tau\eta]\partial_\tau -2(m-1)\eta \\
&\quad+\frac{2[1+(m-1)\tau\eta]\partial_\tau w^s}{w^s}+\frac{(m-1)\tau}{1+(m-1)\tau\eta}
\left(m(1-\eta^2)+2w^s\partial_\eta w^s\right)\Bigg]
\partial_\tau(w^s\partial_\eta w^s)=f_1,\\
&\left.\partial_\tau(w^s\partial_\eta w^s)\right|_{\eta=0}=0, \qquad \left.\partial_\tau(w^s\partial_\eta w^s)\right|_{\eta=1}=0,\\
&\left.\left|\partial_\tau(w^s\partial_\eta w^s)\right|\right|_{\tau=0} \leq -C\overline w_m \overline w_m',
\end{aligned}
\right.
\end{equation*}
where
\begin{align*}
f_1=&(m-1)w^s\partial_\tau w^s +\frac{2\partial_\tau w^s}{w^s}\left[m(\eta^2-1)\partial_\eta(w^s\partial_\eta w^s) -(m-1)\eta w^s\partial_\eta w^s\right]\\
&+\frac{(m-1)\tau w^s\partial_\tau w^s} {1+(m-1)\tau\eta} \left[2(\partial_\eta w^s)^2 -\frac{2m(\eta^2-1)}{w^s}\partial_\eta w^s -(m-1)\eta\right].
\end{align*}
Propositions~\ref{prop:estimate-ws},~\ref{prop:time-derivative-bound}, and~
\ref{prop:optimal-derivative-convergence}, together with
Lemma~\ref{self-similar-properties}, give
\[
|f_1(\tau,\eta)|\leq \frac{C\rho_m(\tau)}{1+(m-1)\tau} (-\overline w_m\overline w_m').
\]
For sufficiently large \(\tau\), the zero-order coefficient is nonpositive.
Indeed, the Falkner--Skan equation gives
 \[
  \overline w_m\overline w_m'
  =-\frac{m+1}{2}f_m\overline w_m-m(1-\eta^2).
 \]
 The preceding estimates imply
 \begin{align*}
 \left|w^s\partial_\eta w^s
 -\overline w_m\overline w_m'\right|&\leq
|w^s-\overline w_m|\,|\partial_\eta w^s|
 +\overline w_m|\partial_\eta(w^s-\overline w_m)|\\
 &\leq C\rho_m(\tau)(-\overline w_m\overline w_m').
 \end{align*}
 Consequently,
 \begin{equation*}
 m(1-\eta^2)+2w^s\partial_\eta w^s\leq-\frac{m+1}{2}f_m\overline w_m
 +\overline w_m\overline w_m'+2C\rho_m(\tau)(-\overline w_m\overline w_m')\leq0
 \end{equation*}
 after fixing \(\tau_0\) so large that \(2C\rho_m(\tau)\leq1\) for
 \(\tau\geq\tau_0\).  Since \(\partial_\tau w^s\leq0\), every term in the
 zero-order coefficient
 \[
 -2(m-1)\eta
 +\frac{2[1+(m-1)\tau\eta]\partial_\tau w^s}{w^s}
 +\frac{(m-1)\tau}{1+(m-1)\tau\eta}
 \left[m(1-\eta^2)+2w^s\partial_\eta w^s\right]
 \]
 is nonpositive for \(\tau\geq\tau_0\).

On the fixed slab \(0\leq\tau\leq\tau_0\),
Propositions~\ref{prop:estimate-ws},~\ref{prop:weighted-first-derivative},
and~\ref{prop:time-derivative-bound} control
the coefficients, the source \(f_1\), and the initial and boundary data in the
weight \(-\overline w_m\overline w_m'\).  Comparing
\(\pm\partial_\tau(w^s\partial_\eta w^s)\) with
\(-Ce^{\kappa\tau}\overline w_m\overline w_m'\), where \(\kappa\)
is sufficiently large, and applying the maximum principle gives
\begin{equation}\label{eq:weighted-mixed-product-finite-time}
[1+(m-1)\tau]
\left|\partial_\tau(w^s\partial_\eta w^s)\right|
\leq C(-\overline w_m\overline w_m'),
\qquad 0\leq\tau\leq\tau_0.
\end{equation}

Consider the comparison function
\[
\psi_7=-\frac{D_7e^\Lambda}{1+(m-1)\tau}
\overline w_m\overline w_m'
e^{-\Lambda\rho_m(\tau)}.
\]
A direct calculation, followed by the steady equation and its derivative,
gives
\begin{align*}
&\Bigg[(w^s)^2\partial_{\eta\eta}
+\left(m(\eta^2-1)
-\frac{(m-1)\tau(w^s)^2}{1+(m-1)\tau\eta}\right)\partial_\eta
-[1+(m-1)\tau\eta]\partial_\tau\\
&\quad-2(m-1)\eta
+\frac{2[1+(m-1)\tau\eta]\partial_\tau w^s}{w^s}
+\frac{(m-1)\tau}{1+(m-1)\tau\eta}
 \left(m(1-\eta^2)+2w^s\partial_\eta w^s\right)
\Bigg]\psi_7\\
&=\psi_7\Bigg[
 \frac{\bigl((w^s)^2-\overline w_m^2\bigr)
 (\overline w_m\overline w_m')''}
 {\overline w_m\overline w_m'}
 +\frac{3m-1}{2}\frac{\overline w_m}{\overline w_m'}
 +\frac{(m-1)(1-\eta)}{1+(m-1)\tau}
 +\frac{2[1+(m-1)\tau\eta]\partial_\tau w^s}{w^s}\\
&\qquad
 +\frac{(m-1)\tau}{1+(m-1)\tau\eta}
 \left(m(1-\eta^2)+2w^s\partial_\eta w^s
 -(w^s)^2
 \frac{(\overline w_m\overline w_m')'}
 {\overline w_m\overline w_m'}\right)
 +\Lambda[1+(m-1)\tau\eta]\partial_\tau\rho_m(\tau)
 \Bigg]\\
&\leq\psi_7\Bigg[
-\frac{(m-1)\tau(-\overline w_m\overline w_m')}
 {1+(m-1)\tau\eta}
-\frac{m+1}{2[1+(m-1)\tau]}(1-\eta)
+C\rho_m(\tau)\\
&\qquad+C\rho_m(\tau)
\frac{(m-1)\tau(-\overline w_m\overline w_m')}
 {1+(m-1)\tau\eta}
+\Lambda[1+(m-1)\tau\eta]\partial_\tau\rho_m(\tau)
\Bigg]\\
&\leq-\frac{cD_7\rho_m(\tau)}{1+(m-1)\tau}
 (-\overline w_m\overline w_m').
\end{align*}
Increase \(\tau_0\) so that \(C\rho_m(\tau)\leq 1/2\). If \(1+(m-1)\tau\eta\geq [1+(m-1)\tau]/2\), the term containing \(\partial_\tau\rho_m\) absorbs \(C\rho_m\) for sufficiently large \(\Lambda\). Otherwise, \(\eta<1/2\) and \((m-1)\tau/[1+(m-1)\tau\eta]>(1-\eta)^{-1}\); hence the first negative term absorbs the remaining errors because \(-\overline w_m\overline w_m'/(1-\eta)\) has a positive lower bound on \([0,1/2]\). For \(m=1\), the conclusion follows directly from \(\partial_\tau\rho_1=-\beta_1\rho_1\).

Set
\[
\Psi_{7,\pm}=\psi_7
\pm\partial_\tau(w^s\partial_\eta w^s),
\]
and take \(D_7\) sufficiently large. Then 
\begin{equation*}
\left\{
\begin{aligned}
&\Bigg[(w^s)^2\partial_{\eta\eta}
+\left(m(\eta^2-1)
-\frac{(m-1)\tau(w^s)^2}{1+(m-1)\tau\eta}\right)\partial_\eta
-[1+(m-1)\tau\eta]\partial_\tau\\
&\quad-2(m-1)\eta
+\frac{2[1+(m-1)\tau\eta]\partial_\tau w^s}{w^s}
+\frac{(m-1)\tau}{1+(m-1)\tau\eta}
 \left(m(1-\eta^2)+2w^s\partial_\eta w^s\right)
\Bigg]\Psi_{7,\pm}\leq0,\\
&\Psi_{7,\pm}|_{\eta=0}\geq0,
\qquad \Psi_{7,\pm}|_{\eta=1}=0,
\qquad \Psi_{7,\pm}|_{\tau=\tau_0}\geq0.
\end{aligned}
\right.
\end{equation*}
The maximum principle gives \(\Psi_{7,\pm}\geq0\), and therefore
\begin{equation}\label{eq:weighted-mixed-product-high-m}
[1+(m-1)\tau]
\left|\partial_\tau(w^s\partial_\eta w^s)\right|
\leq C(-\overline w_m\overline w_m'),
\qquad \tau\geq\tau_0.
\end{equation}
Combining \eqref{eq:weighted-mixed-product-finite-time} and
\eqref{eq:weighted-mixed-product-high-m} gives, for every \(\tau\geq0\),
\begin{align*}
[1+(m-1)\tau]|w^s\partial_{\tau\eta}w^s|&\leq[1+(m-1)\tau]
\left|\partial_\tau(w^s\partial_\eta w^s)\right|
+[1+(m-1)\tau]|\partial_\tau w^s\partial_\eta w^s|\\
&\leq C(-\overline w_m\overline w_m').
\end{align*}

\medskip\noindent\textbf{Second time derivative.}
For \(\partial_{\tau\tau}w^s\), differentiating the equation for
\(\partial_\tau w^s\) once more gives
\begin{equation*}
\left\{
\begin{aligned}
 &\bigg[(w^s)^2\partial_{\eta\eta}
 +m(\eta^2-1)\partial_\eta
 -[1+(m-1)\tau\eta]\partial_\tau\bigg]
 \partial_{\tau\tau}w^s\\
 &\qquad+\bigg[2w^s\partial_{\eta\eta}w^s
 -\frac{7m-5}{2}\eta
 +\frac{4[1+(m-1)\tau\eta]\partial_\tau w^s}{w^s}\bigg]
 \partial_{\tau\tau}w^s=f_2,\\
 &\partial_{\tau\tau}w^s|_{\eta=1}=0,\qquad
 \left.\left(\partial_\eta\partial_{\tau\tau}w^s
 -\frac{m}{(w^s)^2}\partial_{\tau\tau}w^s\right)\right|_{\eta=0}
 =-\left.\frac{2m(\partial_\tau w^s)^2}{(w^s)^3}\right|_{\eta=0},
\end{aligned}
\right.
\end{equation*}
where
\[
 f_2=\frac{4m(\eta^2-1)}{w^s}
 \partial_\tau w^s\partial_{\tau\eta}w^s
 +\left(6\partial_{\eta\eta}w^s
 -\frac{2(5m-3)\eta}{w^s}\right)(\partial_\tau w^s)^2.
\]
The preceding estimates,
together with \eqref{eq:weighted-mixed-product-finite-time} and
\eqref{eq:weighted-mixed-product-high-m}, give
\begin{equation*}
  |f_2(\tau,\eta)|\leq\frac{C\rho_m(\tau)}{[1+(m-1)\tau]^2}\overline w_m(\eta).
\end{equation*}

On \(0\leq\tau\leq\tau_0\), the preceding estimates,
\eqref{eq:initial-high}, and the corner compatibility conditions control the
coefficients, the source, and the initial and boundary data.  Comparison with
\(C e^{\kappa\tau}\overline w_m\), for sufficiently large \(\kappa\),
gives
 \begin{equation}\label{eq:second-time-finite-m}
 |\partial_{\tau\tau}w^s(\tau,\eta)|
 \leq C e^{\kappa\tau}\overline w_m(\eta),
 \qquad0\leq\tau\leq\tau_0.
 \end{equation}

 For \(\tau\geq\tau_0\), set
 \[
  \psi_8=\frac{D_8\rho_m(\tau)}{[1+(m-1)\tau]^2}
  \overline w_m(\eta).
\]
A direct calculation, followed by the steady equation, gives
\begin{align*}
 &\bigg[(w^s)^2\partial_{\eta\eta}
 +m(\eta^2-1)\partial_\eta
 -[1+(m-1)\tau\eta]\partial_\tau\\
 &\qquad +2w^s\partial_{\eta\eta}w^s
 -\frac{7m-5}{2}\eta
 +\frac{4[1+(m-1)\tau\eta]\partial_\tau w^s}{w^s}\bigg]\psi_8\\
 =&\frac{D_8\rho_m}{[1+(m-1)\tau]^2}\Bigg[
 \bigl((w^s)^2-\overline w_m^2\bigr)\overline w_m''\\
 &\qquad+\Bigg(2w^s\partial_{\eta\eta}w^s
 +\frac{2(m-1)(1-\eta)}{1+(m-1)\tau}
 -[1+(m-1)\tau\eta]\frac{\partial_\tau\rho_m}{\rho_m}
 +\frac{4[1+(m-1)\tau\eta]\partial_\tau w^s}{w^s}\Bigg)
 \overline w_m\Bigg]\\
 \leq&\frac{D_8\rho_m}{[1+(m-1)\tau]^2}\overline w_m
 \Bigg[2w^s\partial_{\eta\eta}w^s
 +\frac{2(m-1)}{1+(m-1)\tau}
 -[1+(m-1)\tau\eta]\frac{\partial_\tau\rho_m}{\rho_m}\Bigg]\\
 \leq&-\frac{C_*D_8\rho_m(\tau)}{[1+(m-1)\tau]^2}
 \overline w_m(\eta).
\end{align*}
Here the common decay rate from Theorem~\ref{thm:ws-regularity} is decreased,
if necessary, and \(\tau_0\) is increased so that
\[
 \frac{2(m-1)}{1+(m-1)\tau}
 -[1+(m-1)\tau\eta]\frac{\partial_\tau\rho_m}{\rho_m}
 \leq C_*,\qquad \tau\geq\tau_0;
\]
together with \eqref{eq:m-strict-concavity}, this gives the last inequality.
Set
 \[
 \Psi_{8,\pm}=\psi_8\pm\partial_{\tau\tau}w^s,
 \]
and take \(D_8\) sufficiently large.  
Then
 \begin{equation*}
 \left\{
 \begin{aligned}
 &\Bigg[(w^s)^2\partial_{\eta\eta}
 +m(\eta^2-1)\partial_\eta
 -[1+(m-1)\tau\eta]\partial_\tau\\
 &\qquad
 +2w^s\partial_{\eta\eta}w^s
 -\frac{7m-5}{2}\eta
 +\frac{4[1+(m-1)\tau\eta]\partial_\tau w^s}{w^s}\Bigg]
 \Psi_{8,\pm}\leq0,\\
 &\Psi_{8,\pm}|_{\eta=1}=0,
 \qquad \Psi_{8,\pm}|_{\tau=\tau_0}\geq0,\\
 &\left.\left(\partial_\eta\Psi_{8,\pm}
 -\frac{m}{(w^s)^2}\Psi_{8,\pm}\right)\right|_{\eta=0} \leq0.
 \end{aligned}
 \right.
 \end{equation*}
The maximum principle gives
 \[
 |\partial_{\tau\tau}w^s(\tau,\eta)|
 \leq\frac{D_8\rho_m(\tau)}{[1+(m-1)\tau]^2}\overline w_m(\eta),
 \qquad \tau\geq\tau_0.
 \]
 Combining this estimate with \eqref{eq:second-time-finite-m} and enlarging
 the constant gives, for every \(\tau\geq0\),
 \begin{equation}\label{eq:second-time-closed-high-m}
  [1+(m-1)\tau]^2|\partial_{\tau\tau}w^s(\tau,\eta)|
 \leq C\rho_m(\tau)\overline w_m(\eta).
\end{equation}

{
\medskip\noindent\textbf{Remaining normal derivatives.}
The equation for \(\partial_\tau w^s\) can now be written as
\begin{equation}\label{eq:time-normal-ode-high}
 \begin{aligned}
  &(w^s)^2\partial_{\eta\eta}\partial_\tau w^s
  +m(\eta^2-1)\partial_\eta\partial_\tau w^s\\
  &\qquad=[1+(m-1)\tau\eta]\partial_{\tau\tau}w^s
  -\left(2w^s\partial_{\eta\eta}w^s
  -\frac{5m-3}{2}\eta\right)\partial_\tau w^s.
 \end{aligned}
\end{equation}
Its right-hand side satisfies
\begin{equation}\label{eq:time-normal-source-high}
 \begin{aligned}
 &\left|[1+(m-1)\tau\eta]\partial_{\tau\tau}w^s
 -\left(2w^s\partial_{\eta\eta}w^s
 -\frac{5m-3}{2}\eta\right)\partial_\tau w^s\right|\\
 &\qquad\leq \frac{C\rho_m(\tau)}{1+(m-1)\tau}\overline w_m(\eta).
 \end{aligned}
\end{equation}
Indeed, \(1+(m-1)\tau\eta\leq1+(m-1)\tau\), and
\eqref{eq:second-time-closed-high-m} controls the first term.  The remaining
terms are controlled by Proposition~\ref{prop:optimal-derivative-convergence}
and the uniform bound for \(w^s\partial_{\eta\eta}w^s\) on the finite slab,
together with Proposition~\ref{prop:dynamical-strict-concavity} and
Proposition~\ref{prop:optimal-derivative-convergence} for large time.

Assume first that \(\tau\geq\tau_0\).  For each fixed \(\tau\), apply to
\eqref{eq:time-normal-ode-high} the
integrating-factor argument following \eqref{eq:ode-pre}.  The mixed-product
estimate above gives
\([1+(m-1)\tau]|\partial_{\tau\eta}w^s|
\leq C(-\overline w_m')\), which supplies the vanishing endpoint term in
that argument.  Integrating from \(\eta\) to \(1\) and using
\eqref{eq:time-normal-source-high} therefore gives
\begin{equation}\label{eq:first-time-normal-high}
 [1+(m-1)\tau]|\partial_{\tau\eta}w^s|
 \leq C\rho_m(\tau)\sigma_m(\eta)
 \leq C\rho_m(\tau)(-\overline w_m'(\eta)).
\end{equation}
Substitution of \eqref{eq:first-time-normal-high} and
\eqref{eq:time-normal-source-high} into
\eqref{eq:time-normal-ode-high}, followed by the profile bounds in
Lemma~\ref{self-similar-properties}, gives
\begin{equation}\label{eq:second-time-normal-high}
 [1+(m-1)\tau]|\partial_{\tau\eta\eta}w^s|
 \leq \frac{C\rho_m(\tau)}{(1-\eta)\sigma_m(\eta)}
 \leq C\rho_m(\tau)(-\overline w_m''(\eta)).
\end{equation}
On \(0\leq\tau\leq\tau_0\), the same two estimates follow from the
finite-time weighted parabolic regularity based on \eqref{eq:initial-high};
enlarging the constant combines the two time ranges.

It remains to prove the third- and fourth-order normal estimates.
Differentiating the steady equation once and twice, and using
Lemma~\ref{self-similar-properties}, gives
\begin{equation}\label{eq:steady-higher-normal-weights}
 \begin{aligned}
 |\overline w_m''(\eta)|
 &\leq \frac{C}{(1-\eta)\sigma_m(\eta)},\qquad
 |\overline w_m'''(\eta)|
 \leq \frac{C}{(1-\eta)^2\sigma_m(\eta)},\\
 |\overline w_m''''(\eta)|
 &\leq \frac{C}{(1-\eta)^3\sigma_m(\eta)}.
 \end{aligned}
\end{equation}
For example, after one differentiation the only terms on the right-hand side
of the equation for \(\overline w_m^2\overline w_m'''\) have size at most
\(C\sigma_m(\eta)\); after a second differentiation the corresponding terms have
size at most \(C\sigma_m(\eta)/(1-\eta)\).  Since
\(\overline w_m^2\geq C^{-1}(1-\eta)^2\sigma_m(\eta)^2\), these are precisely the bounds in
the last two parts of \eqref{eq:steady-higher-normal-weights}.

The estimates already proved imply
\begin{equation}\label{eq:normal-weight-summary}
\begin{aligned}
 &C^{-1}(1-\eta)\sigma_m(\eta)\leq w^s
 \leq C(1-\eta)\sigma_m(\eta),
 \qquad |\partial_\eta w^s|\leq C\sigma_m(\eta),\\
 &|\partial_{\eta\eta}w^s|
 \leq\frac{C}{(1-\eta)\sigma_m(\eta)},
 \qquad |\widetilde w|\leq C\rho_m(\tau)(1-\eta)\sigma_m(\eta),\\
 &|\partial_\eta\widetilde w|\leq C\rho_m(\tau)\sigma_m(\eta),
 \qquad |\partial_{\eta\eta}\widetilde w|
 \leq\frac{C\rho_m(\tau)}{(1-\eta)\sigma_m(\eta)}.
\end{aligned}
\end{equation}
Subtracting the steady equation from the equation for \(w^s\), we have
\begin{equation}\label{eq:normal-difference-high}
 \begin{aligned}
 &(w^s)^2\partial_{\eta\eta}\widetilde w
 +m(\eta^2-1)\partial_\eta\widetilde w\\
 &\qquad=[1+(m-1)\tau\eta]\partial_\tau w^s
 +\frac{3m-1}{2}\eta\widetilde w
 -(w^s+\overline w_m)\overline w_m''\widetilde w.
 \end{aligned}
\end{equation}
Equations~\eqref{eq:steady-higher-normal-weights} and
~\eqref{eq:normal-weight-summary} give
\begin{equation}\label{eq:coefficient-normal-weights}
 \begin{aligned}
 \left|(w^s+\overline w_m)\overline w_m''\right|&\leq C,\\
 \left|\partial_\eta\bigl[(w^s+\overline w_m)\overline w_m''\bigr]\right|
 &\leq\frac{C}{1-\eta},\\
 \left|\partial_{\eta\eta}\bigl[(w^s+\overline w_m)
 \overline w_m''\bigr]\right|&\leq\frac{C}{(1-\eta)^2}.
 \end{aligned}
\end{equation}
Direct differentiation of the right-hand side of
\eqref{eq:normal-difference-high}, together with
\((m-1)\tau\leq1+(m-1)\tau\), \eqref{eq:m-time-decay},
\eqref{eq:first-time-normal-high}, \eqref{eq:second-time-normal-high},
\eqref{eq:normal-weight-summary}, and \eqref{eq:coefficient-normal-weights}.
Here \eqref{eq:m-time-decay} controls
\((m-1)\tau\partial_\tau w^s\), while
\eqref{eq:first-time-normal-high}--\eqref{eq:second-time-normal-high} control
the differentiated time terms.  Consequently,
\begin{align}
 &\left|\partial_\eta\left\{
 \begin{aligned}
 &[1+(m-1)\tau\eta]\partial_\tau w^s
 +\frac{3m-1}{2}\eta\widetilde w\\[-2pt]
 &\quad -(w^s+\overline w_m)\overline w_m''\widetilde w
 \end{aligned}\right\}\right|
 \leq C\rho_m(\tau)\sigma_m(\eta),
 \label{eq:first-source-normal-weight}\\
 &\left|\partial_{\eta\eta}\left\{
 \begin{aligned}
 &[1+(m-1)\tau\eta]\partial_\tau w^s
 +\frac{3m-1}{2}\eta\widetilde w\\[-2pt]
 &\quad -(w^s+\overline w_m)\overline w_m''\widetilde w
 \end{aligned}\right\}\right|
 \leq\frac{C\rho_m(\tau)\sigma_m(\eta)}{1-\eta}.
 \label{eq:second-source-normal-weight}
\end{align}

Differentiating \eqref{eq:normal-difference-high} once yields the exact
identity
\begin{align}
 (w^s)^2\partial_\eta^3\widetilde w
 ={}&\partial_\eta\left\{
 \begin{aligned}
 &[1+(m-1)\tau\eta]\partial_\tau w^s
 +\frac{3m-1}{2}\eta\widetilde w\\[-2pt]
 &\quad -(w^s+\overline w_m)\overline w_m''\widetilde w
 \end{aligned}\right\}
 \nonumber\\
 &
 -\bigl(2w^s\partial_\eta w^s+m(\eta^2-1)\bigr)
  \partial_{\eta\eta}\widetilde w
 -2m\eta\partial_\eta\widetilde w.
 \label{eq:third-normal-identity}
\end{align}
Here \(|m(\eta^2-1)|\leq C(1-\eta)\) and \(|2m\eta|\leq C\).  Thus
\eqref{eq:normal-weight-summary} and \eqref{eq:first-source-normal-weight}
show that
the right-hand side of \eqref{eq:third-normal-identity} is bounded by
\(C\rho_m(\tau)\sigma_m(\eta)\).  Dividing by
\((w^s)^2\geq C^{-1}(1-\eta)^2\sigma_m(\eta)^2\) gives
\begin{equation}\label{eq:third-normal-high}
 |\partial_\eta^3\widetilde w|
 \leq\frac{C\rho_m(\tau)}{(1-\eta)^2\sigma_m(\eta)}.
\end{equation}

Differentiating \eqref{eq:normal-difference-high} twice gives
\begin{align}
 (w^s)^2\partial_\eta^4\widetilde w={}&
 \partial_{\eta\eta}\Bigl\{[1+(m-1)\tau\eta]\partial_\tau w^s
 +\frac{3m-1}{2}\eta\widetilde w
 -(w^s+\overline w_m)\overline w_m''\widetilde w\Bigr\}
 \nonumber\\
 &
 -\bigl(4w^s\partial_\eta w^s+m(\eta^2-1)\bigr)
  \partial_\eta^3\widetilde w
 \nonumber\\
 &-\bigl(2(\partial_\eta w^s)^2
 +2w^s\partial_{\eta\eta}w^s
 +4m\eta\bigr)\partial_{\eta\eta}\widetilde w
 -2m\partial_\eta\widetilde w.
 \label{eq:fourth-normal-identity}
\end{align}
The coefficients multiplying \(\partial_\eta^3\widetilde w\) and
\(\partial_{\eta\eta}\widetilde w\) are bounded, respectively, by
\[
 C(1-\eta)\sigma_m(\eta)^2
 \quad\hbox{and}\quad
 C\sigma_m(\eta)^2.
\]
Consequently, \eqref{eq:second-source-normal-weight},
\eqref{eq:normal-weight-summary}, and \eqref{eq:third-normal-high} bound the
right-hand side of \eqref{eq:fourth-normal-identity} by
\[
 \frac{C\rho_m(\tau)\sigma_m(\eta)}{1-\eta}.
\]
Dividing once more by
\((w^s)^2\geq C^{-1}(1-\eta)^2\sigma_m(\eta)^2\) proves
\begin{equation}\label{eq:fourth-normal-high}
 |\partial_\eta^4\widetilde w|
 \leq\frac{C\rho_m(\tau)}{(1-\eta)^3\sigma_m(\eta)}.
\end{equation}
This proves all the estimates in \eqref{eq:ws-high-large-time}.

For completeness, the differentiations above can first be justified by
standard difference quotients or parabolic boundary regularity on compact
sub-strips bounded away from \(\eta=1\), where the equation is uniformly
parabolic.  At \(\eta=0\), differentiating the Robin condition gives
\[
 \left(\partial_{\tau\eta}w^s
 -\frac{m}{(w^s)^2}\partial_\tau w^s\right)_{\eta=0}=0,
\]
and a second time differentiation gives
\[
 \left(\partial_{\tau\tau\eta}w^s
 -\frac{m}{(w^s)^2}\partial_{\tau\tau}w^s\right)_{\eta=0}
 =-\left.\frac{2m(\partial_\tau w^s)^2}{(w^s)^3}\right|_{\eta=0}.
\]
After time derivatives at \(\tau=0\) are eliminated by the equation, the
corresponding corner identities are exactly the higher-order compatibility
conditions assumed in the statement.  These conditions give regularity up to
\((\tau,\eta)=(0,0)\).  The estimates above are uniform under exhaustion of
the whole domain and hence yield
\(w^s\in C^{2,4}_{\mathrm{par}}([0,+\infty)\times[0,1))\).
}
\end{proof}

{
\subsection{Finite-time loss of weighted edge control without the initial time-monotonicity condition}
\label{sec:weighted-c2-counterexample}

We prove Theorem~\ref{thm:weighted-c2-blowup}.  We first write the equation in
an edge coordinate and compare the weighted norms in the two coordinate
systems.  We then choose a periodic datum near the edge and study its
translated limit.  The characteristic flow of the limit equation shows that
the solution cannot remain in the weighted $C^2$ edge class for all time.

\medskip
\noindent\emph{The time and edge variables.}
Set
\begin{equation}\label{eq:edge-common-time}
 s=\int_0^\tau\frac{d\theta}{1+(m-1)\theta}.
\end{equation}
Thus \(s=\tau\) if \(m=1\), whereas
\[
 s=\frac1{m-1}\log[1+(m-1)\tau],
 \qquad 1+(m-1)\tau=e^{(m-1)s}\quad(m>1).
\]
Equation~\eqref{eq:semi-self-similar-crocco-m} becomes
\begin{equation}\label{eq:edge-common-time-equation}
 w_s=\frac1{\alpha(s,\eta)}
 \left[w^2w_{\eta\eta}+m(\eta^2-1)w_\eta
 -\frac{3m-1}{2}\eta w\right],
\end{equation}
where $\alpha(s,\eta)=\eta+(1-\eta)e^{-(m-1)s}$. Introduce
\begin{equation}\label{eq:edge-coordinate-chi}
 \chi=\log\bigl(2\sqrt{-\log(1-\eta)}\bigr),
 \qquad
 h(s,\chi)=\frac{w(s,\eta)}{\overline w_m(\eta)}.
\end{equation}
We use \(\chi\) for the normal edge coordinate to distinguish it from the
tangential coordinate \(\xi=\log x\) used in Section~3.  By definition,
\[
 e^{2\chi}=4[-\log(1-\eta)],\qquad
 \chi_\eta=\frac1{2(1-\eta)[-\log(1-\eta)]},
\]
and
\[
 \chi_{\eta\eta}
 =\frac{-\log(1-\eta)-1}
 {2(1-\eta)^2[-\log(1-\eta)]^2}.
\]
Substituting \(w=\overline w_m h\) into
\eqref{eq:edge-common-time-equation} and using the steady equation, we obtain
\begin{equation}\label{eq:edge-relative-equation}
 h_s=\nu h^2h_{\chi\chi}+(d_1h^2+d_0)h_\chi
      +g h(h^2-1),
\end{equation}
where
\begin{equation}\label{eq:edge-relative-coefficients}
\begin{aligned}
 \nu&=\frac{\overline w_m^2\chi_\eta^2}{\alpha},\qquad
 d_1=\frac1\alpha\bigl(\overline w_m^2\chi_{\eta\eta}
           +2\overline w_m\overline w_m'\chi_\eta\bigr),\\
 d_0&=\frac{m(\eta^2-1)\chi_\eta}{\alpha},\qquad
 g=\frac{\overline w_m\overline w_m''}{\alpha}.
\end{aligned}
\end{equation}
Here \(\eta\) is a function of \(\chi\).  The estimates in
Appendix~\ref{app:steady-profile} and the formulas above imply, uniformly for
\(0\leq s\leq T\),
\begin{equation}\label{eq:edge-coefficient-limits}
\begin{aligned}
 \nu(s,\chi)&=(m+1)e^{-2\chi}
       +O_T((1+\chi)e^{-4\chi}),
 &\nu_\chi(s,\chi)&=O_T((1+\chi)e^{-2\chi}),\\
 d_1(s,\chi)&=-\frac{m+1}{2}+O_T((1+\chi)e^{-2\chi}),
 &d_0(s,\chi)&=O_T(e^{-2\chi}),\\
 g(s,\chi)&=-\frac{m+1}{2}+O_T((1+\chi)e^{-2\chi})
\end{aligned}
\end{equation}
as \(\chi\to+\infty\).  These estimates, including the bound for
\(\nu_\chi\), are proved in Appendix~\ref{app:edge-counterexamples} from the
exact formulas in \eqref{eq:edge-relative-coefficients}.

\medskip
\noindent\emph{Equivalence of the weighted norms.}
Fix \(\eta_b\in(0,1)\) sufficiently close to \(1\), and set
\begin{equation}\label{eq:Y2-norm-definition}
 \|w\|_{\mathcal Y_m^2}
 :=\left\|\frac{w}{\overline w_m}\right\|_\infty
 +\left\|\frac{w_\eta}{-\overline w_m'}\right\|_\infty
 +\left\|\frac{w_{\eta\eta}}{-\overline w_m''}\right\|_\infty.
\end{equation}
{
We require interior continuity and finiteness of the weighted norm above; no
trace at \(\eta=1\) is imposed on the weighted derivative ratios.
}

\begin{lemma}
\label{lem:natural-edge-norm-equivalence}
There exists \(C=C(m,\eta_b)>0\) such that every
\(w\in C^2([0,1))\cap C([0,1])\) with \(w(1)=0\) satisfies
\begin{equation}\label{eq:natural-edge-norm-equivalence}
\begin{aligned}
 C^{-1}\|w\|_{\mathcal Y_m^2}
 \leq{}&\|w\|_{C^2([0,\eta_b])}
 +\|h\|_\infty+\|h_\chi\|_\infty
 +\|e^{-2\chi}h_{\chi\chi}\|_\infty\leq C\|w\|_{\mathcal Y_m^2}.
\end{aligned}
\end{equation}
\end{lemma}

\begin{proof}
On \([0,\eta_b]\), the result follows from the usual \(C^2\) estimates.  We
only need to consider \(\eta_b\leq\eta<1\).  The chain rule gives
\begin{equation}\label{eq:Y2-chain-rule}
 \frac{w_\eta}{\overline w_m'}
 =h+\frac{\overline w_m\chi_\eta}{\overline w_m'}h_\chi,
\end{equation}
and
\begin{equation}\label{eq:Y2-chain-rule-second}
 \frac{w_{\eta\eta}}{\overline w_m''}
 =h+
 \frac{2\overline w_m'\chi_\eta+
       \overline w_m\chi_{\eta\eta}}{\overline w_m''}h_\chi
 +e^{2\chi}\frac{\overline w_m\chi_\eta^2}
                       {\overline w_m''}
       \bigl(e^{-2\chi}h_{\chi\chi}\bigr).
\end{equation}
The coefficients of the derivatives of \(h\) satisfy
\begin{equation}\label{eq:Y2-chain-coefficients}
 \frac{\overline w_m\chi_\eta}{\overline w_m'}
 =-2e^{-2\chi}+o(e^{-2\chi}),
 \qquad
 \frac{2\overline w_m'\chi_\eta+
       \overline w_m\chi_{\eta\eta}}{\overline w_m''}\longrightarrow1,
 \qquad
 e^{2\chi}\frac{\overline w_m\chi_\eta^2}
                   {\overline w_m''}\longrightarrow-2.
\end{equation}
These limits give the first inequality in
\eqref{eq:natural-edge-norm-equivalence}.

For the reverse inequality, set
\[
 p_1=\frac{w_\eta}{\overline w_m'},\qquad
 p_2=\frac{w_{\eta\eta}}{\overline w_m''}.
\]
Differentiating \(p_1\) gives
\[
 (p_1)_\chi=\frac{\overline w_m''}{\overline w_m'\chi_\eta}(p_2-p_1),
 \qquad
 \|(p_1)_\chi\|_{L^\infty(\eta_b,1)}
 \leq C(\|p_1\|_\infty+\|p_2\|_\infty).
\]
Since \(w(1)=\overline w_m(1)=0\), we also have
\[
 h(\eta)=\frac1{\overline w_m(\eta)}
 \int_\eta^1(-\overline w_m'(\zeta))p_1(\zeta)\,d\zeta.
\]
It follows that
\begin{align*}
 |p_1(\eta)-h(\eta)|
 &\leq\frac{\|(p_1)_\chi\|_\infty}{\overline w_m(\eta)}
 \int_\eta^1(-\overline w_m'(\zeta))
 [\chi(\zeta)-\chi(\eta)]\,d\zeta\\
 &\leq Ce^{-2\chi(\eta)}
 (\|p_1\|_\infty+\|p_2\|_\infty).
\end{align*}
For the last inequality, set $\ell=-\log(1-\zeta)$.  The integral divided by
$\overline w_m(\eta)$ is then bounded by
\[
 \frac{C}{(1-\eta)\sqrt{-\log(1-\eta)}}
 \int_{-\log(1-\eta)}^\infty
 e^{-\ell}\sqrt\ell
 \log\frac{\ell}{-\log(1-\eta)}\,d\ell
 \leq\frac C{-\log(1-\eta)}=4Ce^{-2\chi(\eta)}.
\]
Since the coefficient of \(h_\chi\) in \eqref{eq:Y2-chain-rule} is comparable
to \(-e^{-2\chi}\), we obtain
\[
 \|h_\chi\|_{L^\infty(\eta_b,1)}
 \leq C(\|p_1\|_\infty+\|p_2\|_\infty).
\]
Equation~\eqref{eq:Y2-chain-rule-second} and the last limit in
\eqref{eq:Y2-chain-coefficients} now give a bound for
\(e^{-2\chi}h_{\chi\chi}\).  The usual \(C^2\) estimates apply on
\([0,\eta_b]\).  Since all the identities are linear in \(w\), the same
argument applies to differences.
\end{proof}

We use the standard local existence and continuation result for the Crocco
equation.  Every positive smooth datum satisfying the wall compatibility
condition gives a positive strong solution for a short time.  The solution
can be continued past a finite time $S$ if
\[
 0<c\leq h(s,\chi)\leq C
\]
and
\[
 \sup_{0\leq s<S}
 \left(
  \|h_\chi(s,\cdot)\|_{L^\infty(\mathbb R)}
  +\|e^{-2\chi}h_{\chi\chi}(s,\cdot)\|_{L^\infty(\mathbb R)}
 \right)<+\infty.
\]
We use this continuation criterion without repeating the local construction.

\medskip
\noindent\emph{Choice of the initial datum.}
Choose a smooth cutoff \(\vartheta\) such that
\[
 \vartheta(\chi)=0\quad(\chi\leq R),\qquad
 \vartheta(\chi)=1\quad(\chi\geq R+1),
 \qquad0\leq\vartheta\leq1,
\]
and define
\begin{equation}\label{eq:periodic-tail-datum}
 h_0(\chi)=1+\vartheta(\chi)[1+\sin(3\chi)],
 \qquad
 w_0^s(\eta)=\overline w_m(\eta)h_0(\chi(\eta)).
\end{equation}
Then \(1\leq h_0\leq3\).  For \(\chi\geq R+1\),
\begin{equation}\label{eq:periodic-tail-derivatives}
 h_0=2+\sin(3\chi),\qquad
 h_{0,\chi}=3\cos(3\chi),\qquad
 e^{-2\chi}h_{0,\chi\chi}=-9e^{-2\chi}\sin(3\chi).
\end{equation}
The datum agrees with \(\overline w_m\) near \(\eta=0\).  Thus it satisfies
the same Robin condition and corner compatibility conditions as the steady
solution.  Lemma~\ref{lem:natural-edge-norm-equivalence} gives
\(\|w_0^s\|_{\mathcal Y_m^2}<+\infty\).

At \(s=0\), equation~\eqref{eq:edge-relative-equation} gives
\begin{equation}\label{eq:periodic-tail-initial-time}
 \frac{\partial_\tau w^s(0,\eta)}{\overline w_m(\eta)}
 =\nu h_0^2h_{0,\chi\chi}+(d_1h_0^2+d_0)h_{0,\chi}
  +g h_0(h_0^2-1).
\end{equation}
The right-hand side is bounded because \(e^{2\chi}\nu\) and
\(e^{-2\chi}h_{0,\chi\chi}\) are bounded.  By
\eqref{eq:edge-coefficient-limits}, its value at
\(3\chi=\pi+2k\pi\) is \(\frac{3}{2}(m+1)+o(1)>0\), while its value at
\(3\chi=2k\pi\) is \(-9(m+1)+o(1)<0\).  Therefore the initial time derivative
changes sign infinitely many times near \(\eta=1\).

\medskip
\noindent\emph{The translated limit and finite-time loss.}
Let \(s_{\max}>0\) be the maximal existence time in
\(\mathcal Y_m^2\) given by the local theory.  We show that
\begin{equation}\label{eq:edge-blowup-time-s}
 s_{\max}\leq s_*:=\frac1{m+1}\log\frac32.
\end{equation}
Assume instead that \(s_{\max}>s_*\), and choose
\(T\in(s_*,s_{\max})\).
{
The constants \(h=1\) and \(h=3\) are respectively a solution and a strict
supersolution of \eqref{eq:edge-relative-equation}, since \(g<0\).  At
\(\eta=0\), the boundary condition gives
\[
 h_\eta=\frac{m(h^2-1)}{\overline w_m^2h},
\]
which excludes a boundary minimum below \(1\) or a boundary maximum above
\(3\).  The standard comparison argument, with a cutoff or strict-barrier
approximation on the unbounded edge domain, therefore gives
}
\begin{equation}\label{eq:periodic-zero-order-bound}
 1\leq h(s,\chi)\leq3,
 \qquad 0\leq s\leq T,\quad\chi\in\mathbb R.
\end{equation}
Since \(T<s_{\max}\), the definition of \(s_{\max}\),
\eqref{eq:Y2-norm-definition}, and
Lemma~\ref{lem:natural-edge-norm-equivalence} give, for some \(M<\infty\),
\begin{equation}\label{eq:periodic-assumed-bound}
 \sup_{0\leq s\leq T,\ \chi\in\mathbb R}
 \bigl(|h|+|h_\chi|+|e^{-2\chi}h_{\chi\chi}|\bigr)\leq M.
\end{equation}
{
Indeed, \eqref{eq:edge-relative-coefficients}--\eqref{eq:edge-coefficient-limits}
show that \(e^{2\chi}\nu\), \(d_1\), \(d_0\), and \(g\) are uniformly bounded
on every fixed tail.  Hence \eqref{eq:periodic-assumed-bound} gives there
\[
 |\nu h^2h_{\chi\chi}|
 \leq C|e^{-2\chi}h_{\chi\chi}|,
 \qquad
 |(d_1h^2+d_0)h_\chi|+|gh(h^2-1)|\leq C.
\]
Equation~\eqref{eq:edge-relative-equation} therefore yields a uniform bound
for \(h_s\) on the product of \([0,T]\) with that tail.

Set
\[
 X_n=\frac{2\pi n}{3},\qquad h_n(s,z)=h(s,X_n+z).
\]
On each compact cylinder, \(X_n+z\) lies in a fixed tail for all sufficiently
large \(n\).  Since \((h_n)_s=h_s(s,X_n+z)\) and
\((h_n)_z=h_\chi(s,X_n+z)\), the preceding bounds are uniform in \(n\) and
show that the family \((h_n)\) is equi-Lipschitz jointly in \(s\) and \(z\).  Passing to a
diagonal subsequence, we obtain
\begin{equation}\label{eq:periodic-compactness}
 h_n\longrightarrow H\quad\hbox{locally uniformly on }[0,T]\times\mathbb R,
 \qquad H\in W^{1,\infty}_{\rm loc},\qquad \|H_z\|_\infty\leq M,
\end{equation}
and the choice of \(X_n\) gives \(H(0,z)=2+\sin(3z)\).
}

{
Write
\[
 \nu_n(s,z)=\nu(s,X_n+z),\qquad
 d_{j,n}(s,z)=d_j(s,X_n+z)\quad(j=0,1),\qquad
 g_n(s,z)=g(s,X_n+z).
\]
On every compact cylinder, \(\nu_n\to0\), \((\nu_n)_z\to0\),
\(d_{1,n}\to-(m+1)/2\), \(d_{0,n}\to0\), and
\(g_n\to-(m+1)/2\), all uniformly.  We cannot pass to the limit in the
diffusion term pointwise.  Instead, for
\(\varphi\in C_c^\infty((0,T)\times\mathbb R)\), integration by parts gives
\begin{align}\label{eq:periodic-diffusion-vanishes}
 \iint\nu_nh_n^2(h_n)_{zz}\varphi
 =-\iint(h_n)_z\bigl[(\nu_n)_zh_n^2\varphi
 +2\nu_nh_n(h_n)_z\varphi+\nu_nh_n^2\varphi_z\bigr]
 \longrightarrow0.
\end{align}
Here we used the uniform bounds for \(h_n\) and \((h_n)_z\), together with the
uniform convergence of \(\nu_n\) and \((\nu_n)_z\) on the support of
\(\varphi\).

The distributional limit of the principal transport factor is obtained
without any convergence assertion for \((h_n)_z\).  Indeed,
\begin{align*}
 \iint h_n^2(h_n)_z\varphi
 &=-\frac13\iint h_n^3\varphi_z
 \longrightarrow-\frac13\iint H^3\varphi_z
 =\iint H^2H_z\varphi.
\end{align*}
Thus \(h_n^2(h_n)_z\to H^2H_z\) in distributions.  To treat the remaining
transport terms, set \(r_{1,n}=d_{1,n}+(m+1)/2\).  If \(Q\) is a compact
cylinder containing the support of \(\varphi\), then
\[
 \left|\iint
   \bigl(r_{1,n}h_n^2+d_{0,n}\bigr)(h_n)_z\varphi\right|
 \leq C_Q\bigl(\|r_{1,n}\|_{L^\infty(Q)}
              +\|d_{0,n}\|_{L^\infty(Q)}\bigr)\|\varphi\|_{L^1}
 \longrightarrow0.
\]
Consequently,
\[
 (d_{1,n}h_n^2+d_{0,n})(h_n)_z
 \longrightarrow-\frac{m+1}{2}H^2H_z
 \qquad\hbox{in }\mathcal D'((0,T)\times\mathbb R).
\]
Finally, \(g_nh_n(h_n^2-1)\to-(m+1)H(H^2-1)/2\) locally uniformly, while
local uniform convergence of \(h_n\) implies
\((h_n)_s\to H_s\) in distributions.  Passing term by term in the equation
therefore gives
\begin{equation}\label{eq:periodic-limit-equation}
 H_s+\frac{m+1}{2}H^2H_z=\frac{m+1}{2}H(1-H^2)
\end{equation}
in the sense of distributions.  Since \(H_z\in L^\infty\), the equation also
gives \(H_s\in L^\infty\).  Hence \(H\) is Lipschitz in space and time.
}

Let \(\Gamma\) solve the characteristic equation
\[
 \frac{d}{ds}\Gamma(s,z)=\frac{m+1}{2}H(s,\Gamma(s,z))^2,
 \qquad \Gamma(0,z)=z.
\]
The velocity is Lipschitz in \(z\), so for each \(s\) the map
\(z\mapsto\Gamma(s,z)\) is strictly increasing.  Set
\(H_0(z)=2+\sin(3z)\).
{
The characteristic identity holds first for almost every initial point and,
by continuity in \(z\), for every \(z\).  It gives
}
\[
 H(s,\Gamma(s,z))
 =\frac{e^{(m+1)s/2}H_0(z)}
 {\sqrt{1+H_0(z)^2(e^{(m+1)s}-1)}},
 \qquad
 \Gamma(s,z)=z+\frac12\log[1+H_0(z)^2(e^{(m+1)s}-1)].
\]
Therefore
\begin{equation}\label{eq:periodic-flow-derivative-explicit}
 \Gamma_z(s,z)=
 \frac{1+(e^{(m+1)s}-1)H_0(z)[H_0(z)+H_0'(z)]}
 {1+H_0(z)^2(e^{(m+1)s}-1)}.
\end{equation}
At \(z_*=\pi/3\), \(H_0(z_*)=2\) and \(H_0'(z_*)=-3\), so
\[
 \Gamma_z(s,z_*)=
 \frac{3-2e^{(m+1)s}}{1+4(e^{(m+1)s}-1)}<0
 \qquad\hbox{for }s>\frac1{m+1}\log\frac32.
\]
This contradicts the strict monotonicity of the flow map and proves
\eqref{eq:edge-blowup-time-s}.

Returning to the variable \(\tau\), let \(T^*\) be the maximal weighted
existence time.  Then
\begin{equation}\label{eq:edge-blowup-time-tau}
 T^*\leq
 \begin{cases}
  \displaystyle\frac12\log\frac32,&m=1,\\[2mm]
  \displaystyle\frac{(3/2)^{(m-1)/(m+1)}-1}{m-1},&m>1.
 \end{cases}
\end{equation}
The bounds in \eqref{eq:periodic-zero-order-bound} show that the amplitude
remains bounded.
If
\[
 \limsup_{\tau\uparrow T^*}
 \|w^s(\tau,\cdot)\|_{\mathcal Y_m^2}<+\infty,
\]
then Lemma~\ref{lem:natural-edge-norm-equivalence} and
\eqref{eq:periodic-zero-order-bound} give all the bounds in the continuation
criterion.  The solution could then be extended past \(T^*\), which
contradicts its maximality.  Therefore
\[
\limsup_{\tau\uparrow T^*}
 \|w^s(\tau,\cdot)\|_{\mathcal Y_m^2}=+\infty,
\]
which proves Theorem~\ref{thm:weighted-c2-blowup}.
{
Thus the obstruction is a steepening of the edge profile in the weighted
derivative topology.  It is not an amplitude blow-up, since
\eqref{eq:periodic-zero-order-bound} remains valid up to the maximal weighted
existence time.
}
}

\section{Nonlinear stability in Crocco variables}

Theorem~\ref{thm:semi-self-similar-stability} is proved by a bootstrap
argument.  Throughout the section, we write
\[
 \widehat w(\tau,\xi,\eta)=w(\tau,\xi,\eta)-w^s(\tau,\eta).
\]
By Theorem~\ref{thm:ws-regularity}, there exist \(C_*=C_*(m,E^*)>0\) and \(\tau_0>0\) such that
\[
 w^s\partial_{\eta\eta}w^s\leq-C_*,
 \qquad \tau\geq\tau_0.
\]
The decay exponent in \(\rho_m\) is decreased, if necessary, so that
\begin{equation}\label{eq:stability-decay-choice-m}
 0\leq-[1+(m-1)\tau\eta]\frac{\rho_m'(\tau)}{\rho_m(\tau)}
 \leq\frac{C_*}{4},
 \qquad \tau\geq\tau_0.
\end{equation}

Assume that there exist $N_1,\ldots,N_4>0$ such that
\begin{equation}\label{eq:bootstrap-m}
    \begin{aligned}
        \left|\frac{w(\tau,\xi,\eta)-w^s(\tau,\eta)}{w^s(\tau,\eta)}\right|&\leq N_1\delta\rho_m(\tau),\\
        \left|\frac{\partial_\xi w(\tau,\xi,\eta)}{w^s(\tau,\eta)}\right|&\leq N_1\delta\rho_m(\tau),\\
        \left|\frac{\partial_\eta w(\tau,\xi,\eta)-\partial_\eta w^s(\tau,\eta)}{-\overline w_m'(\eta)}\right|&\leq N_2\delta\rho_m(\tau),\\
         \left|\frac{[1+(m-1)\tau](\partial_\tau w(\tau,\xi,\eta)-\partial_\tau w^s(\tau,\eta))}{w^s(\tau,\eta)}\right|&\leq N_3\delta\rho_m(\tau),\\
         \left|\frac{\partial_{\eta\eta}w(\tau,\xi,\eta)-\partial_{\eta\eta}w^s(\tau,\eta)}{-\overline w_m''(\eta)}\right|&\leq N_4\delta\rho_m(\tau).
    \end{aligned}
\end{equation}
Here $N_i\delta\leq1/2$, $1\leq i\leq4$.
By the estimates for $w^s$, the weights generated by $w^s$ and \(\overline w_m\) are equivalent up to fixed constants.  
The estimates below improve \eqref{eq:bootstrap-m}.  
All constants below depend at most on $m$, $\tau_0$ and $E^*$, and are independent of $N_i$, $\delta$, and the variables unless explicitly stated otherwise.

\begin{proposition}\label{prop:stability-zero}
Let \(w\) be a smooth solution of \eqref{eq:prandtl-crocco-m} satisfying
\eqref{eq:bootstrap-m}.  There exists
\(K_1=K_1(m,E^*,\tau_0)>0\) such that
\begin{equation}\label{eq:zero-output-m}
 |\widehat w(\tau,\xi,\eta)|
 \leq K_1\delta\rho_m(\tau)w^s(\tau,\eta).
\end{equation}
\end{proposition}

\begin{proof}
Subtracting \eqref{eq:semi-self-similar-crocco-m} from
\eqref{eq:prandtl-crocco-m} gives
\begin{equation}\label{eq:perturbation-direct-m}
\left\{
\begin{aligned}
 &w^2\partial_{\eta\eta}\widehat w
 +m(\eta^2-1)\partial_\eta\widehat w
 -\frac{3m-1}{2}\eta\widehat w
 -\eta\partial_\xi\widehat w
 -[1+(m-1)\tau\eta]\partial_\tau\widehat w\\
 &\qquad +(w+w^s)\partial_{\eta\eta}w^s\widehat w=0,\\
 &\widehat w|_{\eta=1}=0,
 \qquad
 \left.\left(\partial_\eta\widehat w
 -\frac{m}{ww^s}\widehat w\right)\right|_{\eta=0}=0,\\
 &\widehat w|_{\tau=0}=w_0-w_0^s,
 \qquad \left.\widehat w\right|_{\xi\to-\infty}=0.
\end{aligned}
\right.
\end{equation}
Let \(L_1\) denote the differential operator in the first line of \eqref{eq:perturbation-direct-m}, including its zero-order term.
Set
\[
 \psi_1=
 \begin{cases}
  \delta e^{\kappa_1\tau}w^s,
  &0\leq\tau\leq\tau_0,\\
  K_1\delta\rho_m(\tau)w^s,
  &\tau_0<\tau<+\infty.
 \end{cases}
\]

For \(0\leq\tau\leq\tau_0\), substitution gives
\begin{align*}
 L_1 \psi_1 =&\delta e^{\kappa_1\tau}
 \left[(w^2-(w^s)^2)\partial_{\eta\eta}w^s
 +(w+w^s)w^s\partial_{\eta\eta}w^s
 -\kappa_1[1+(m-1)\tau\eta]w^s\right]\\
 =&\delta e^{\kappa_1\tau}w^s
 \left[
 \frac{w(w+w^s)}{(w^s)^2}w^s\partial_{\eta\eta}w^s
 -\kappa_1[1+(m-1)\tau\eta]
 \right].
\end{align*}
Theorem~\ref{thm:ws-regularity} and the bootstrap bounds \eqref{eq:bootstrap-m} imply that
\[
\frac{w(w+w^s)}{(w^s)^2}
w^s\partial_{\eta\eta}w^s
\]
is uniformly bounded. Choose \(\kappa_1>0\) sufficiently large. Then
\[
L_1 \psi_1 \leq-\delta e^{\kappa_1\tau}w^s, \qquad 0\leq\tau\leq\tau_0.
\]
At \(\eta=0\),
\[
 \left.\left(\partial_\eta-\frac{m}{ww^s}\right)
 \psi_1\right|_{\eta=0}
 =-m\delta e^{\kappa_1\tau}
 \left(\frac1{w^s}+\frac1w\right)<0.
\]
Set \(\Phi_{1,\pm}=\psi_1\pm\widehat w\).
The comparison functions satisfy
\[
\left\{
\begin{aligned}
 &L_1\Phi_{1,\pm}
 \leq-\delta e^{\kappa_1\tau}w^s<0,\\
 &\left.\Phi_{1,\pm}\right|_{\eta=1}=0,\qquad
 \left.\left(\partial_\eta\Phi_{1,\pm}
 -\frac{m}{ww^s}\Phi_{1,\pm}\right)\right|_{\eta=0}<0,\\
 &\left.\Phi_{1,\pm}\right|_{\xi\to-\infty}
 =\delta e^{\kappa_1\tau}w^s\geq0,\\
 &\left.\Phi_{1,\pm}\right|_{\tau=0}
 =\delta w_0^s\pm(w_0-w_0^s)\geq0.
\end{aligned}
\right.
\]
The zero-order coefficient is bounded on the finite time slab, so an
exponential conjugation permits the maximum principle there.  Hence
\begin{equation}\label{eq:zero-finite-m}
 |\widehat w|\leq\delta e^{\kappa_1\tau}w^s,
 \qquad0\leq\tau\leq\tau_0.
\end{equation}

For \(\tau\geq\tau_0\), substitution gives
\begin{align*}
 L_1 \psi_1 &=K_1\delta\rho_m w^s
 \left[
 \frac{w(w+w^s)}{(w^s)^2}w^s\partial_{\eta\eta}w^s
 -[1+(m-1)\tau\eta]\frac{\rho_m'}{\rho_m}
 \right]\\
 &\leq \psi_1 \left(-\frac{3}{4}C_*+\frac{C_*}{4}\right)\leq 0,
\end{align*}
using \eqref{eq:stability-decay-choice-m} and \eqref{eq:bootstrap-m}.
At \(\eta=0\),
\[
 \left.\left(\partial_\eta-\frac{m}{ww^s}\right)
 \psi_1\right|_{\eta=0}
 =-mK_1\delta\rho_m
 \left(\frac1{w^s}+\frac1w\right)<0.
\]
The comparison functions satisfy
\[
\left\{
\begin{aligned}
 &L_1\Phi_{1,\pm}
 \leq 0,\\
 &\left.\Phi_{1,\pm}\right|_{\eta=1}=0,\qquad
 \left.\left(\partial_\eta\Phi_{1,\pm}
 -\frac{m}{ww^s}\Phi_{1,\pm}\right)\right|_{\eta=0}<0,\\
 &\left.\Phi_{1,\pm}\right|_{\xi\to-\infty}
 =K_1\delta\rho_m w^s\geq0,\\
 &\left.\Phi_{1,\pm}\right|_{\tau=\tau_0}
 \geq\left[K_1\rho_m(\tau_0)-e^{\kappa_1\tau_0}\right]
 \delta w^s=0
\end{aligned}
\right.
\]
for $K_1=e^{\kappa_1\tau_0} \rho_m^{-1}(\tau_0)$.
The maximum principle gives
\[
 |\widehat w|\leq K_1\delta\rho_m(\tau)w^s,
 \qquad\tau\geq\tau_0.
\]
Together with the definition of \(K_1\), this proves \eqref{eq:zero-output-m}.
\end{proof}

\begin{proposition}\label{prop:stability-tangential}
Let \(w\) be a smooth solution of \eqref{eq:prandtl-crocco-m} satisfying
\eqref{eq:bootstrap-m}.  There exists
\(K_2=K_2(m,E^*,\tau_0)>0\) such that
\begin{equation}\label{eq:tangential-output-m}
 |\partial_\xi\widehat w(\tau,\xi,\eta)|
 \leq K_2\delta\rho_m(\tau)w^s(\tau,\eta).
\end{equation}
\end{proposition}

\begin{proof}
Set \(\widehat q=\partial_\xi\widehat w\).  Differentiation of
\eqref{eq:prandtl-crocco-m} with respect to \(\xi\) gives
\begin{equation}\label{eq:tangential-direct-m}
\left\{
\begin{aligned}
 &w^2\partial_{\eta\eta}\widehat q
 +m(\eta^2-1)\partial_\eta\widehat q
 -\eta\partial_\xi\widehat q -[1+(m-1)\tau\eta]\partial_\tau\widehat q\\
 &\qquad -\frac{3m-1}{2}\eta\widehat q +2w\partial_{\eta\eta}w\widehat q=0,\\
 &\widehat q|_{\eta=1}=0,
 \qquad
 \left.\left(\partial_\eta\widehat q
 -\frac{m}{w^2}\widehat q\right)\right|_{\eta=0}=0,\\
 &\widehat q|_{\tau=0}=\partial_\xi w_0,
 \qquad
 \left.\widehat q\right|_{\xi\to-\infty}=0.
\end{aligned}
\right.
\end{equation}
Let \(L_2\) denote the differential operator in the first line of
\eqref{eq:tangential-direct-m}, including its zero-order term.
Set
\[
 \psi_2=
 \begin{cases}
  \delta e^{\kappa_2\tau}w^s,
  &0\leq\tau\leq\tau_0,\\
  K_2\delta\rho_m(\tau)w^s,
  &\tau_0<\tau<+\infty.
 \end{cases}
\]

For \(0\leq\tau\leq\tau_0\), substitution gives
\begin{align*}
 L_2 \psi_2 =&\delta e^{\kappa_2\tau}
 \left[(w^2-(w^s)^2)\partial_{\eta\eta}w^s
 +2w\partial_{\eta\eta}w\,w^s
 -\kappa_2[1+(m-1)\tau\eta]w^s\right]\\
 \leq & \psi_2 \left[\frac{w^2-(w^s)^2}{(w^s)^2} w^s \partial_{\eta\eta}w^s +2w\partial_{\eta\eta}w
 -\kappa_2\right].
\end{align*}
Theorem~\ref{thm:ws-regularity} and the bootstrap bounds \eqref{eq:bootstrap-m} imply that
\[
\frac{w^2-(w^s)^2}{(w^s)^2} w^s \partial_{\eta\eta}w^s +2w\partial_{\eta\eta}w
\]
is uniformly bounded. Choose \(\kappa_2>0\) sufficiently large. Then
\[
L_2 \psi_2 \leq-\delta e^{\kappa_2\tau}w^s, \qquad 0\leq\tau\leq\tau_0.
\]
At \(\eta=0\),
\[
 \left.\left(\partial_\eta-\frac{m}{w^2}\right)
 \psi_2\right|_{\eta=0}
 =-m\delta e^{\kappa_2\tau}
 \left(\frac1{w^s}+\frac{w^s}{w^2}\right)<0.
\]
Set \(\Phi_{2,\pm}=\psi_2\pm\widehat q\). The comparison functions satisfy
\[
\left\{
\begin{aligned}
 &L_2\Phi_{2,\pm}
 \leq-\delta e^{\kappa_2\tau}w^s<0,\\
 &\left.\Phi_{2,\pm}\right|_{\eta=1}=0,\qquad
 \left.\left(\partial_\eta\Phi_{2,\pm}
 -\frac{m}{w^2}\Phi_{2,\pm}\right)\right|_{\eta=0}<0,\\
 &\left.\Phi_{2,\pm}\right|_{\xi\to-\infty}
 =\delta e^{\kappa_2\tau}w^s\geq0,\\
 &\left.\Phi_{2,\pm}\right|_{\tau=0}
 =\delta w_0^s\pm\partial_\xi w_0\geq0.
\end{aligned}
\right.
\]
The zero-order coefficient is bounded on the finite time slab.  The maximum
principle, after an exponential conjugation, gives
\begin{equation}\label{eq:tangent-finite-m}
 |\partial_\xi\widehat w|=|\widehat q|
 \leq\delta e^{\kappa_2\tau}w^s,
 \qquad0\leq\tau\leq\tau_0.
\end{equation}

For \(\tau\geq\tau_0\), \eqref{eq:bootstrap-m},
\eqref{eq:stability-decay-choice-m}, and Proposition~\ref{prop:stability-zero}
give, after increasing \(\tau_0\) if necessary,
\[
 w\partial_{\eta\eta}w
 =\frac{w}{w^s}w^s\partial_{\eta\eta}w^s
 +w\partial_{\eta\eta}\widehat w
 \leq-\frac{C_*}{2}+CN_4\delta\rho_m(\tau)
 \leq-\frac{C_*}{4}
\]
and
\[
 \left|(w^2-(w^s)^2)\partial_{\eta\eta}w^s\right| \leq\frac{C_*}{8}w^s.
\]
Substitution into \eqref{eq:tangential-direct-m} yields
\begin{align*}
 L_2 \psi_2 &=K_2\delta\rho_m w^s
 \left[
 \frac{w^2-(w^s)^2}{(w^s)^2}w^s\partial_{\eta\eta}w^s +2w \partial_{\eta\eta} w
 -[1+(m-1)\tau\eta]\frac{\rho_m'}{\rho_m}
 \right]\\
 &\leq \psi_2 \left(\frac{C_*}{8}-\frac{C_*}{2}+\frac{C_*}{4}\right)\leq 0.
\end{align*}
At \(\eta=0\),
\[
 \left.\left(\partial_\eta-\frac{m}{w^2}\right)
 \psi_2\right|_{\eta=0}
 =-mK_2\delta\rho_m
 \left(\frac1{w^s}+\frac{w^s}{w^2}\right)<0.
\]
The comparison functions satisfy
\[
\left\{
 \begin{aligned}
 &L_2\Phi_{2,\pm} \leq 0,\\
 &\left.\Phi_{2,\pm}\right|_{\eta=1}=0,\qquad
 \left.\left(\partial_\eta\Phi_{2,\pm}
 -\frac{m}{w^2}\Phi_{2,\pm}\right)\right|_{\eta=0}<0,\\
 &\left.\Phi_{2,\pm}\right|_{\xi\to-\infty}
 =K_2\delta\rho_m w^s\geq0,\\
 &\left.\Phi_{2,\pm}\right|_{\tau=\tau_0}
 \geq\left[K_2\rho_m(\tau_0)-e^{\kappa_2\tau_0}\right]
 \delta w^s=0 
 \end{aligned}
\right.
\]
for $K_2=e^{\kappa_2\tau_0} \rho_m^{-1}(\tau_0)$. The maximum principle and \eqref{eq:tangent-finite-m} give
\[
 |\partial_\xi\widehat w|
 \leq K_2\delta\rho_m(\tau)w^s,
 \qquad\tau\geq\tau_0.
\]
Together with the definition of \(K_2\), this proves
\eqref{eq:tangential-output-m}.
\end{proof}

\begin{proposition}\label{prop:stability-normal}
Let \(w\) be a smooth solution of \eqref{eq:prandtl-crocco-m} satisfying
\eqref{eq:bootstrap-m}.  There exist constants
\[
 K_3=K_3(m,E^*,\tau_0)>0,\qquad K_4=K_4(m,E^*)>0,
\]
\[
 K_5=K_5(m,E^*,\tau_0)>0,\qquad K_6=K_6(m,E^*)>0
\]
such that
\begin{equation}\label{eq:normal-output-m}
 |w\partial_\eta\widehat w|
 \leq (K_3+K_4N_3)\delta\rho_m(\tau)
 (-\overline w_m\overline w_m'),
\end{equation}
and
\begin{equation}\label{eq:time-output-m}
 [1+(m-1)\tau]|\partial_\tau\widehat w|
 \leq[K_5+K_6\rho_m(\tau_0)N_2]
 \delta\rho_m(\tau)w^s.
\end{equation}
\end{proposition}

\begin{proof}
\medskip\noindent\textbf{Step 1. Equations for the normal and time derivatives.}
Set
\[
 \widehat z=w\partial_\eta\widehat w,
 \qquad \widehat p=\partial_\tau\widehat w.
\]
Differentiating \eqref{eq:perturbation-direct-m} with respect to \(\eta\)
and multiplying by \(w\) give
\begin{equation}\label{eq:normal-identity-m}
\left\{
\begin{aligned}
 &w^2\partial_{\eta\eta}\widehat z
 +m(\eta^2-1)\partial_\eta\widehat z
 -\eta\partial_\xi\widehat z
 -[1+(m-1)\tau\eta]\partial_\tau\widehat z\\
 &\qquad+
 \left[2w\partial_{\eta\eta}w^s-(m-1)\eta\right]\widehat z
 =g_3,\\
 &\widehat z|_{\eta=1}=0,
 \qquad \widehat z|_{\eta=0}
 =\left.\frac{m}{w^s}\widehat w\right|_{\eta=0},\\
 &\widehat z|_{\tau=0}=w_0\partial_\eta(w_0-w_0^s),
 \qquad \left.\widehat z\right|_{\xi\to-\infty}=0,
\end{aligned}
\right.
\end{equation}
where
\begin{equation*}
 g_3=w\partial_\xi\widehat w+(m-1)\tau w\widehat p-w\left[
 2\partial_\eta w^s\partial_{\eta\eta}w^s
 +(w+w^s)\partial_{\eta\eta\eta}w^s
 -\frac{3m-1}{2}\right]\widehat w.
\end{equation*}
Differentiating \eqref{eq:perturbation-direct-m} with respect to \(\tau\)
and collecting the terms containing \(\widehat p\) on the left give
\begin{equation}\label{eq:time-identity-m}
\left\{
\begin{aligned}
 &w^2\partial_{\eta\eta}\widehat p
 +m(\eta^2-1)\partial_\eta\widehat p
 -\eta\partial_\xi\widehat p
 -[1+(m-1)\tau\eta]\partial_\tau\widehat p\\
 &\quad+
 \left[2w\partial_{\eta\eta}w-\frac{5m-3}{2}\eta
 +\frac{2[1+(m-1)\tau\eta]\partial_\tau w^s}{w}\right]
 \widehat p=g_4,\\
 &\widehat p|_{\eta=1}=0, \qquad \left.\left(\partial_\eta\widehat p
 -\frac{m}{w^2}\widehat p\right)\right|_{\eta=0}
 =-\left.\frac{m(w+w^s)\widehat w}{w^2(w^s)^2}
 \partial_\tau w^s\right|_{\eta=0},\\
 &\widehat p|_{\tau=0}
 =\partial_\tau w_0 -\partial_\tau w^s_0,
 \qquad \left.\widehat p\right|_{\xi\to-\infty}=0,
\end{aligned}
\right.
\end{equation}
where
\begin{align*}
 g_4={}&-\frac{2\partial_\tau w^s}{w}
 \left\{\eta\partial_\xi\widehat w
 -m(\eta^2-1)\partial_\eta\widehat w
 -\left[(w+w^s)\partial_{\eta\eta}w^s
 -\frac{3m-1}{2}\eta\right]\widehat w\right\}\\
 &-\left[2\partial_\tau w^s\partial_{\eta\eta}w^s
 +(w+w^s)\partial_{\tau\eta\eta}w^s\right]\widehat w.
\end{align*}
Let \(L_3\) and \(L_4\) denote the differential operators on the left-hand
sides of \eqref{eq:normal-identity-m} and \eqref{eq:time-identity-m},
respectively.

\medskip\noindent\textbf{Step 2. The finite time interval.}
For \(0\leq \tau \leq \tau_0\), separate the coupling terms as
\[
 g_3=(m-1)\tau w\widehat p+\widetilde g_3,
 \qquad
 g_4=
 \frac{2m(1-\eta^2)(-\partial_\tau w^s)}{w^2}
 \widehat z+\widetilde g_4.
\]
The two coupling coefficients are nonnegative. Propositions~
\ref{prop:stability-zero} and~\ref{prop:stability-tangential}, and
\eqref{eq:bootstrap-m} give
\[
 |\widetilde g_3|
 \leq C_{\tau_0}\delta
 (-\overline w_m\overline w_m'),
 \qquad
 |\widetilde g_4|
 \leq C_{\tau_0}\delta w.
\]
Take
\[
 \psi_3=D\delta e^{\kappa\tau}
 (-\overline w_m\overline w_m'),
 \qquad
 \psi_4=D\delta e^{\kappa\tau}w,
\]
where \(D\geq1\) and \(\kappa>0\) will be chosen below. Direct substitution gives
\begin{align*}
 L_3\psi_3={}&D\delta e^{\kappa\tau}
 \biggl\{
 w^2(-\overline w_m\overline w_m')''
 +m(\eta^2-1)(-\overline w_m\overline w_m')'\\
 &+
 \bigl[
 2w\partial_{\eta\eta}w^s-(m-1)\eta
 -\kappa(1+(m-1)\tau\eta)
 \bigr]
 (-\overline w_m\overline w_m')
 \biggr\}\\
 ={}&D\delta e^{\kappa\tau}
 \biggl\{(w^2-\overline w_m^2) (-\overline w_m\overline w_m')'' -\frac{3m-1}{2} \overline w_m^2\\
 &+[2w\partial_{\eta\eta}w^s-\kappa(1+(m-1)\tau\eta)] (-\overline w_m\overline w_m')\biggr\}\\
 \leq{}&-D\delta e^{\kappa\tau}(\kappa-C_{\tau_0})
 (-\overline w_m\overline w_m'),
\end{align*}
and, using the equation for \(w\),
\begin{align*}
 L_4\psi_4
 ={}&D\delta e^{\kappa\tau}w
 \left[2w\partial_{\eta\eta}w-(m-1)\eta
 +\frac{2[1+(m-1)\tau\eta]\partial_\tau w^s}{w}
 -\kappa[1+(m-1)\tau\eta]\right]\\
 \leq{}&-D\delta e^{\kappa\tau}(\kappa-C_{\tau_0})w.
\end{align*}
The profile estimates also give
\[
\begin{aligned}
 (m-1)\tau w\psi_4
 &\leq C_{\tau_0}D\delta e^{\kappa\tau}
 (-\overline w_m\overline w_m'),\\
 \frac{2m(1-\eta^2)(-\partial_\tau w^s)}{w^2}\psi_3
 &\leq C_{\tau_0}D\delta e^{\kappa\tau}w.
\end{aligned}
\]
Choosing $\kappa\geq \max\{\kappa_1, 3C_{\tau_0}+1\}$ gives
\begin{align*}
  &L_3\psi_3+(m-1)\tau w\psi_4\pm\widetilde g_3\leq-D\delta e^{\kappa\tau} (-\overline w_m\overline w_m')<0,\\
  &L_4\psi_4+\frac{2m(1-\eta^2)(-\partial_\tau w^s)}{w^2}\psi_3\pm\widetilde g_4\leq-D\delta e^{\kappa\tau}w<0.
\end{align*}
Choose \(D=D(m,E^*,\tau_0)\) sufficiently large so that
\[
 \psi_3|_{\tau=0}\geq|\widehat z|_{\tau=0},
 \qquad
 \psi_4|_{\tau=0}\geq|\widehat p|_{\tau=0},
\]
and so that the wall inequalities below hold.

Set
\[
 \Phi_{3,\pm}=\psi_3\pm\widehat z,
 \qquad
 \Phi_{4,\pm}=\psi_4\pm\widehat p.
\]
The equations for \(\widehat z\) and \(\widehat p\) give
\[
\left\{
\begin{aligned}
 &L_3\Phi_{3,+} +(m-1)\tau w\Phi_{4,-}<0,\\
 &L_3\Phi_{3,-}+(m-1)\tau w\Phi_{4,+}<0,\\
 &L_4\Phi_{4,+}+\frac{2m(1-\eta^2)(-\partial_\tau w^s)}{w^2}\Phi_{3,-}<0,\\
 &L_4\Phi_{4,-}+\frac{2m(1-\eta^2)(-\partial_\tau w^s)}{w^2}\Phi_{3,+}<0.
\end{aligned}
\right.
\]
Moreover,
\[
\left\{
\begin{aligned}
 &\left.\Phi_{3,\pm}\right|_{\eta=1}=0, \qquad  \left.\Phi_{3,\pm}\right|_{\eta=0}\geq m\delta\bigl(De^{\kappa\tau}-e^{\kappa_1\tau}\bigr)\geq0,\\
 &\left.\Phi_{3,\pm}\right|_{\tau=0}\geq0, \qquad \left.\Phi_{3,\pm}\right|_{\xi\to-\infty}
 =\psi_3\geq0,\\
 &\left.\Phi_{4,\pm}\right|_{\eta=1}=0, \qquad \left.\left(\partial_\eta\Phi_{4,\pm}-\frac{m}{w^2}\Phi_{4,\pm} \right)\right|_{\eta=0} \leq
\left.
\frac{\delta}{w}
\left(
-2mDe^{\kappa\tau}
+C_{\tau_0}e^{\kappa_1\tau}
\right)\right|_{\eta=0}\leq 0,\\
 &\left.\Phi_{4,\pm}\right|_{\tau=0}\geq0, \qquad \left.\Phi_{4,\pm}\right|_{\xi\to-\infty} =\psi_4\geq0.
\end{aligned}
\right.
\]
Assume, to the contrary, that one of the four comparison functions becomes negative. Let \(\tau_*\) be the first time at which any one of them loses nonnegativity. The boundary inequalities imply that the first contact is interior in \(\eta\) and may occur either in the interior of the \(\xi\)-interval or at its outflow endpoint \(\xi=\ln L\). Suppose, for example, that
\[
\Phi_{3,+}(\tau_*,\xi_*,\eta_*)=0.
\]
By the definition of \(\tau_*\), all four comparison functions are nonnegative at \(\tau=\tau_*\); in particular,
\[
\Phi_{4,-}(\tau_*,\xi_*,\eta_*)\geq0.
\]
At the fixed time \(\tau_*\), the contact point is a spatial interior minimum of \(\Phi_{3,+}\). Hence
\[
\partial_{\eta\eta}\Phi_{3,+}\geq0,
\qquad
\partial_\eta\Phi_{3,+}
=0,\qquad \partial_\xi\Phi_{3,+}\leq 0.
\]
Since \(\tau_*\) is the first contact time,
\[
\partial_\tau\Phi_{3,+}\leq0.
\]
The zero-order term vanishes because \(\Phi_{3,+}=0\). Therefore, at the contact point,
\[
L_3\Phi_{3,+}=w^2\partial_{\eta\eta}\Phi_{3,+}-\eta_* \partial_\xi\Phi_{3,+} -[1+(m-1)\tau_*\eta_*]\partial_\tau\Phi_{3,+}\geq0.
\]
Since \((m-1)\tau_*w\geq0\) and \(\Phi_{4,-}\geq0\), it follows that
\[
L_3\Phi_{3,+}
+(m-1)\tau_*w\Phi_{4,-}\geq0,
\]
contradicting the first inequality of the comparison system. The other three cases are identical. Therefore,
\[
 |\widehat z|\leq\psi_3,
 \qquad
 |\widehat p|\leq\psi_4,
 \qquad 0\leq\tau\leq\tau_0.
\]
Consequently,
\begin{equation}\label{eq:finite-slab-normal-time-m}
 \frac{|\widehat z|}
 {-\overline w_m\overline w_m'}
 +\frac{|\widehat p|}{w^s}
 \leq C_{\tau_0}\delta,
 \qquad 0\leq\tau\leq\tau_0.
\end{equation}

\medskip\noindent\textbf{Step 3. The normal derivative for large time.}
Propositions~\ref{prop:stability-zero} and~
\ref{prop:stability-tangential} give
\[
 |w\partial_\xi\widehat w|
 +\left|w\left[
 2\partial_\eta w^s\partial_{\eta\eta}w^s
 +(w+w^s)\partial_{\eta\eta\eta}w^s
 -\frac{3m-1}{2}\right]\widehat w\right|
 \leq C_{\tau_0}\delta\rho_m
 (-\overline w_m\overline w_m').
\]
The bootstrap estimate for \(\widehat p\) gives
\[
 (m-1)\tau w|\widehat p|
 \leq\frac{(m-1)\tau}{1+(m-1)\tau}
 N_3\delta\rho_m ww^s
 \leq CN_3\delta\rho_m
 (-\overline w_m\overline w_m').
\]
Thus
\begin{equation}\label{eq:normal-source-bound-m}
 |g_3|\leq (C_{\tau_0}+CN_3)\delta\rho_m (-\overline w_m\overline w_m').
\end{equation}
Here the coefficient \(C\) of \(N_3\) is independent of \(\tau_0\).

Take
\[
 \psi_5=(K_3+K_4N_3)\delta\rho_m
 (-\overline w_m\overline w_m'),
\]
where \(K_3,K_4>0\) will be chosen below. 
Substitution gives
\begin{align*}
 L_3\psi_5
 ={}&(K_3+K_4N_3)\delta\rho_m\left[
 (w^2-\overline w_m^2)(-\overline w_m\overline w_m')''
 -\frac{3m-1}{2}\overline w_m^2\right.\\
 &\left.\quad
 +2w\partial_{\eta\eta}w^s(-\overline w_m\overline w_m')
 -[1+(m-1)\tau\eta]\frac{\rho_m'}{\rho_m}
 (-\overline w_m\overline w_m')\right]\\
 \leq{}&-c(K_3+K_4N_3)\delta\rho_m
 (-\overline w_m\overline w_m'),
\end{align*}
for \(\tau\geq\tau_0\), after increasing \(\tau_0\), where
\(c=c(m,E^*)>0\).  Here
\[
 |(w^2-\overline w_m^2)
 (-\overline w_m\overline w_m')''|
 \leq C\rho_m(-\overline w_m\overline w_m'),
 \qquad
 2w\partial_{\eta\eta}w^s\leq-C_*,
\]
and \eqref{eq:stability-decay-choice-m} controls the term containing
\(\rho_m'\).

Choose $K_4\geq\frac{2C}{c}$, and then choose \(K_3=K_3(m,E^*,\tau_0)\) so that
\[
 cK_3\geq2C_{\tau_0},\qquad
 K_3\rho_m(\tau_0)\geq C_{\tau_0},
 \qquad K_3\geq K_1.
\]
Then \eqref{eq:normal-source-bound-m} and the preceding calculation give
\begin{align*}
   L_3 \psi_5 \pm g_3 & \leq- c (K_3+K_4N_3)\delta\rho_m (-\overline w_m\overline w_m')+(C_{\tau_0}+CN_3)\delta\rho_m (-\overline w_m\overline w_m')\\
   & \leq [(-cK_3+C_{\tau_0})+(-cK_4+C)N_3]\delta\rho_m (-\overline w_m\overline w_m')<0.
\end{align*}
Set \(\Phi_{5,\pm}=\psi_5\pm\widehat z\). 
For \(\tau\geq\tau_0\), the comparison functions satisfy
\[
\left\{
\begin{aligned}
 &L_3\Phi_{5,\pm} = L_3 \psi_5 \pm g_3 <0,\\
 &\left.\Phi_{5,\pm}\right|_{\eta=1}=0, \qquad \left.\Phi_{5,\pm}\right|_{\eta=0} \geq m\delta\rho_m(K_3-K_1)\geq0,\\
 &\left.\Phi_{5,\pm}\right|_{\xi\to-\infty}
 =\psi_5\geq0, \qquad \left.\Phi_{5,\pm}\right|_{\tau=\tau_0}
 \geq\left[K_3\rho_m(\tau_0)-C_{\tau_0}\right]
 \delta(-\overline w_m\overline w_m')\geq0.
\end{aligned}
\right.
\]
Here the wall inequality follows from Proposition~\ref{prop:stability-zero},
and the matching condition follows from
\eqref{eq:finite-slab-normal-time-m}.  The maximum principle gives
\eqref{eq:normal-output-m}.

\medskip\noindent\textbf{Step 4. The time derivative for large time.}
Propositions~\ref{prop:stability-zero} and~\ref{prop:stability-tangential}, together with Theorem~\ref{thm:higher-order-regularity}, give
\[
\begin{aligned}
 &\frac{|\partial_\tau w^s|}{w}
 \left\{\eta|\partial_\xi\widehat w|
 +\left|(w+w^s)\partial_{\eta\eta}w^s
 -\frac{3m-1}{2}\eta\right||\widehat w|\right\}\\
 &\qquad+
 \left|2\partial_\tau w^s\partial_{\eta\eta}w^s
 +(w+w^s)\partial_{\tau\eta\eta}w^s\right||\widehat w|\\
 &\leq
 C_{\tau_0}\delta
 \frac{\rho_m(\tau)}{1+(m-1)\tau}w^s.
\end{aligned}
\]
For the remaining term, the bootstrap estimate for $\partial_\eta\widehat w$ gives
\[
 \frac{|\partial_\tau w^s|}{w}m(1-\eta^2)
 |\partial_\eta\widehat w|\leq C N_2\delta
 \frac{\rho_m^2(\tau)}{1+(m-1)\tau}w^s.
\]
Therefore,
\begin{equation}\label{eq:time-source-bound-m}
 |g_4|\leq
 [C_{\tau_0}+C\rho_m(\tau)N_2]\delta
 \frac{\rho_m(\tau)}{1+(m-1)\tau}w^s.
\end{equation}
Here the coefficient \(C\) of \(\rho_m(\tau_0)N_2\) is independent of \(\tau_0\). Moreover, Proposition~\ref{prop:stability-zero} and
Theorem~\ref{thm:ws-regularity} give
\[
 \left|
 \frac{m(w+w^s)\widehat w}{w^2(w^s)^2}
 \partial_\tau w^s\right|_{\eta=0}
 \leq C_{\tau_0}\delta
 \frac{\rho_m}{[1+(m-1)\tau]w},
\]

Take
\[
 \psi_6=[K_5+K_6\rho_m(\tau_0)N_2]\delta
 \frac{\rho_m(\tau)}{1+(m-1)\tau}w,
\]
where \(K_5,K_6>0\) will be chosen below.  Substitution gives
\begin{align*}
 L_4\psi_6
 ={}&[K_5+K_6\rho_m(\tau_0)N_2]\delta
 \frac{\rho_m}{1+(m-1)\tau}w\\
 &\quad\times\left[
 2w\partial_{\eta\eta}w
 +\frac{2[1+(m-1)\tau\eta]\partial_\tau w^s}{w}
 +\frac{(m-1)(1-\eta)}{1+(m-1)\tau}
 -[1+(m-1)\tau\eta]\frac{\rho_m'}{\rho_m}
 \right]\\
 \leq{}& [K_5+K_6\rho_m(\tau_0)N_2]\delta \frac{\rho_m}{1+(m-1)\tau}w \left(-\frac{C_*}{2}+\frac{(m-1)(1-\eta)}{1+(m-1)\tau}+\frac{C_*}{4}\right)\\
 \leq{}&-c[K_5+K_6\rho_m(\tau_0)N_2]\delta
 \frac{\rho_m}{1+(m-1)\tau}w,
\end{align*}
for \(\tau\geq\tau_0\), after increasing \(\tau_0\), where \(c=c(m,E^*)>0\).

Choose \(K_6\geq2C/c\), and then choose
\(K_5=K_5(m,E^*,\tau_0)\) so that
\[
 cK_5\geq2C_{\tau_0},\qquad
 K_5\frac{\rho_m(\tau_0)}{1+(m-1)\tau_0}
 \geq 2 C_{\tau_0},\qquad
 mK_5\geq C_{\tau_0}.
\]
Then \eqref{eq:time-source-bound-m} and the preceding calculation give
\begin{align*}
 L_4\psi_6\pm g_4
 &\leq-c[K_5+K_6\rho_m(\tau_0)N_2]\delta
 \frac{\rho_m}{1+(m-1)\tau}w +[C_{\tau_0}+C\rho_m(\tau_0)N_2]\delta
 \frac{\rho_m}{1+(m-1)\tau}w^s\\
 &\leq\left[
 (-cK_5+2C_{\tau_0})
 +(-cK_6+2C)\rho_m(\tau_0)N_2\right]\delta
 \frac{\rho_m}{1+(m-1)\tau}w\leq 0.
\end{align*}
At \(\eta=0\),
\begin{align*}
   &[K_5+K_6\rho_m(\tau_0)N_2]\delta \left(\partial_\eta-\frac{m}{w^2}\right) \left(\frac{\rho_m}{1+(m-1)\tau}w\right) \pm \left| \frac{m(w+w^s)\widehat w}{w^2(w^s)^2} \partial_\tau w^s\right|\\
   \leq & -[K_5+K_6\rho_m(\tau_0)N_2]\delta \frac{2m\rho_m}{[1+(m-1)\tau]w} + C_{\tau_0}\delta\frac{\rho_m}{[1+(m-1)\tau]w}\\
   \leq & (-2mK_5+C_{\tau_0})\delta\frac{\rho_m}{[1+(m-1)\tau]w}\leq 0.
\end{align*}

Set \(\Phi_{6,\pm}=\psi_6\pm\widehat p\).
For \(\tau\geq\tau_0\), the comparison functions satisfy
\[
\left\{
\begin{aligned}
 &L_4\Phi_{6,\pm}=L_4\psi_6\pm g_4<0,\\
 &\left.\Phi_{6,\pm}\right|_{\eta=1}=0,\qquad \left.\left(\partial_\eta\Phi_{6,\pm} -\frac{m}{w^2}\Phi_{6,\pm}\right)\right|_{\eta=0}<0,\\
 &\left.\Phi_{6,\pm}\right|_{\xi\to-\infty}=\psi_6\geq0,\qquad \left.\Phi_{6,\pm}\right|_{\tau=\tau_0}
 \geq\left[
 K_5\frac{\rho_m(\tau_0)}{1+(m-1)\tau_0}
 -2 C_{\tau_0}\right]\delta w |_{\tau=\tau_0}\geq0.
\end{aligned}
\right.
\]
The maximum principle gives \eqref{eq:time-output-m}.
\end{proof}

\begin{proposition}\label{prop:stability-second-normal}
Let \(w\) be a smooth solution of \eqref{eq:prandtl-crocco-m} satisfying
\eqref{eq:bootstrap-m}.  There exists \(C_0=C_0(m,E^*)>0\) such that
\begin{equation}\label{eq:second-output-m}
 |\partial_{\eta\eta}\widehat w|
 \leq C_0\{K_1+K_2+K_3+K_4N_3
 +K_5+K_6\rho_m(\tau_0)N_2\}
 \delta\rho_m(\tau)(-\overline w_m'').
\end{equation}
\end{proposition}

\begin{proof}
The perturbation equation gives
\begin{equation}\label{eq:second-recovery-m}
\begin{aligned}
 w^2\partial_{\eta\eta}\widehat w
={}&[1+(m-1)\tau\eta]\widehat p
 +\eta\partial_\xi\widehat w
 -m(\eta^2-1)\partial_\eta\widehat w\\
 &-\left[(w+w^s)\partial_{\eta\eta}w^s
 -\frac{3m-1}{2}\eta\right]\widehat w.
\end{aligned}
\end{equation}
Dividing \eqref{eq:second-recovery-m} by \(w^2\) gives
\begin{align*}
 |\partial_{\eta\eta}\widehat w|
 \leq{}&\frac{[1+(m-1)\tau\eta]|\partial_\tau\widehat w|}{w^2}
 +\frac{\eta|\partial_\xi\widehat w|}{w^2}
 +\frac{m|\eta^2-1|\,|\partial_\eta\widehat w|}{w^2}\\
 &+\frac{\left|(w+w^s)\partial_{\eta\eta}w^s
 -\frac{3m-1}{2}\eta\right||\widehat w|}{w^2}.
\end{align*}
Since
\[
 1+(m-1)\tau\eta\leq1+(m-1)\tau,
\]
Lemma~\ref{self-similar-properties}, Theorem~\ref{thm:ws-regularity},
and~\eqref{eq:bootstrap-m} give
\[
 \frac{w^s}{w^2}
 +\frac{m|\eta^2-1|
 (-\overline w_m\overline w_m')}{w^3}
 +\frac{\left|(w+w^s)\partial_{\eta\eta}w^s
 -\frac{3m-1}{2}\eta\right|w^s}{w^2}
 \leq C(-\overline w_m'').
\]
Consequently, Propositions~\ref{prop:stability-zero},~
\ref{prop:stability-tangential}, and~\ref{prop:stability-normal}, together
with \eqref{eq:time-output-m}, give
\begin{equation*}
  |\partial_{\eta\eta}\widehat w|\leq C_0\{K_5+K_6\rho_m(\tau_0)N_2+K_2
 +K_3+K_4N_3+K_1\} \delta\rho_m(\tau)(-\overline w_m''),
\end{equation*}
which is \eqref{eq:second-output-m}.
\end{proof}

\subsection*{Proof of Theorem~\ref{thm:semi-self-similar-stability}}

\begin{proof}
Enlarge \(C_0=C_0(m,E^*)\geq1\) so that it covers the equivalence of all
weights used below. Since $K_4$ and $K_6$ are independent of
\(\tau_0\), increase \(\tau_0\), if necessary, so that all large-time
comparisons above hold and
\begin{equation}\label{eq:bootstrap-coupling-small-m}
 16C_0^2K_4K_6\rho_m(\tau_0)\leq\frac12.
\end{equation}
With this final choice of \(\tau_0\), fix \(K_1,K_2,K_3\), and \(K_5\)
as above. Set
\begin{align*}
   N_1&=4C_0(K_1+K_2),\\
 N_2&=8C_0\{K_3+1+4C_0K_4(K_5+1)\},\\
 N_3&=4C_0\{K_5+K_6\rho_m(\tau_0)N_2+1\}.
\end{align*}
Indeed, \eqref{eq:bootstrap-coupling-small-m} gives
\begin{equation*}
   4C_0(K_3+K_4N_3+1)=\frac{N_2}{2}+16C_0^2K_4K_6\rho_m(\tau_0)N_2\leq N_2.
\end{equation*}
Finally, set
\[
 N_4=4C_0\{K_1+K_2+K_3+K_4N_3
 +K_5+K_6\rho_m(\tau_0)N_2+1\}.
\]
The estimates above give
\begin{align*}
 \left|\frac{\widehat w}{w^s}\right|
 +\left|\frac{\partial_\xi\widehat w}{w^s}\right|
 &\leq\frac{N_1}{4}\delta\rho_m(\tau),\\
 \left|\frac{\partial_\eta\widehat w}{-\overline w_m'}\right|
 &\leq\frac{N_2}{4}\delta\rho_m(\tau),\\
 \left|\frac{[1+(m-1)\tau]\partial_\tau\widehat w}{w^s}\right|
 &\leq\frac{N_3}{4}\delta\rho_m(\tau),\\
 \left|\frac{\partial_{\eta\eta}\widehat w}{-\overline w_m''}\right|
 &\leq\frac{N_4}{4}\delta\rho_m(\tau).
\end{align*}
Choose \(\delta_0=\delta_0(m,E^*)>0\) so small that
\[
 N_i\delta_0\leq\frac12,
 \qquad1\leq i\leq4.
\]
Since \(N_i\geq4\) and the initial \(\mathbb X_m\)-norm is at most
\(\delta\), the initial data lie strictly inside the bootstrap bounds.
The standard local Oleinik theory gives a smooth solution on a
nontrivial time interval on which \eqref{eq:bootstrap-m} holds.  Let \(T^*\)
be the maximal time on which these bounds hold.  The preceding estimates
improve every bootstrap bound by a factor of at least \(1/2\), so the bounds
cannot be saturated at a finite \(T^*\).  They also preserve positivity and
control the weighted derivatives in \(\mathbb X_m\); hence the standard
continuation criterion extends the solution past any finite \(T^*\).
Therefore \(T^*=+\infty\), and \eqref{eq:smallness-condition} follows.
\end{proof}

\begin{remark}\label{rem:relaxed-entrance}
The entrance condition \(w^-=w^s\) in
\eqref{eq:prandtl-crocco-m} can be relaxed.  Suppose that
\[
 w_0^-(\eta)=\lim_{\xi\to-\infty}w_0(\xi,\eta)
\]
is compatible with the semi-similar Crocco problem and satisfies
\[
 \|w_0^--w_0^s\|_{\mathbb X_m}\leq\delta.
\]
Repeating the preceding estimates with the \(\xi\)-derivatives omitted
gives a global semi-similar solution \(w^-(\tau,\eta)\) such that
\[
 \|w^-(\tau,\cdot)-w^s(\tau,\cdot)\|_{\mathbb X_m}
 \leq C\delta\rho_m(\tau).
\]
Using this estimate at \(\xi\to-\infty\) in the comparison problems of
Propositions~\ref{prop:stability-zero}--\ref{prop:stability-normal}
gives the same conclusion as in
Theorem~\ref{thm:semi-self-similar-stability}.
\end{remark}

\section{Returning to Eulerian variables}

This section translates the Crocco estimates into Eulerian variables.
For $U(x)=x^m$, introduce the scaling factor
\[
 {\Delta_m}=\Delta_m(x)=x^{\frac{1-m}{2}},
\]
and the Crocco variables
\[
 \tau=tx^{m-1},\qquad \xi=\ln x,\qquad
 r=x^{\frac{m-1}{2}}y,
\]
\[
 \eta=\frac{u(t,x,y)}{x^m},\qquad
 w(\tau,\xi,\eta)=x^{-\frac{3m-1}{2}}\partial_yu(t,x,y).
\]
The inverse Crocco map is
\begin{equation}\label{eq:general-eulerian-inverse-m}
 r=\int_0^{u/x^m}\frac{d\zeta}{w(\tau,\xi,\zeta)}.
\end{equation}
Thus, as long as $w>0$ for $0\leq\eta<1$, the map $\eta\mapsto r$ is
invertible and defines a monotone Eulerian profile.  The identities needed
below are
\[
 \partial_yu=x^m{\Delta_m^{-1}}w,\qquad
 \partial_y={\Delta_m^{-1}}w\partial_\eta,
\]
\[
 \partial_y^\ell u
 =x^m{\Delta_m^{-\ell}}(w\partial_\eta)^{\ell-1}w,
 \qquad \ell\geq1,
\]
and, at fixed $x$ and fixed Crocco coordinate,
\[
 \partial_t={\Delta_m^{-2}}\partial_\tau.
\]
At fixed $(t,y)$, a function of $(\tau,\xi,r)$ satisfies
\begin{equation}\label{eq:general-x-operator-m}
 x\partial_x=\partial_\xi+(m-1)\tau\partial_\tau
 +\frac{m-1}{2}r\partial_r.
\end{equation}
These identities follow directly from the Crocco transformation.

\subsection*{Proof of Theorem~\ref{thm:eulerian-semi-self-similar}}

\begin{proof}
Let
\[
 \eta^s=\frac{u^s(t,x,y)}{x^m},\qquad
 \overline\eta=\frac{\overline u_m(x,y)}{x^m}.
\]
At the same Eulerian point, the inverse Crocco relations are
\begin{equation}\label{eq:two-inverse-m}
 x^{\frac{m-1}{2}}y
 =\int_0^{\eta^s}\frac{d\zeta}{w^s(\tau,\zeta)}
 =\int_0^{\overline\eta}\frac{d\zeta}{\overline w_m(\zeta)}.
\end{equation}
Since $\overline w_m\leq w^s\leq C\overline w_m$, one has
$\eta^s\geq\overline\eta$.  Moreover, \eqref{eq:ws-large-time} gives
\begin{equation*}
   0\leq\int_{\overline\eta}^{\eta^s}\frac{d\zeta}{\overline w_m(\zeta)}=\int_0^{\eta^s}
 \frac{w^s(\tau,\zeta)-\overline w_m(\zeta)}
 {w^s(\tau,\zeta)\overline w_m(\zeta)}\,d\zeta\leq C\rho_m(\tau)
 \int_0^{\overline\eta}\frac{d\zeta}{\overline w_m(\zeta)}.
\end{equation*}
Since $\partial_\eta\sigma_m=[2(1-\eta)\sigma_m]^{-1}$, the steady-profile
bounds imply
\[
 \frac{2}{M_2}[\sigma_m(b)-\sigma_m(a)]
 \leq\int_a^b\frac{d\zeta}{\overline w_m(\zeta)}
 \leq\frac{2}{M_1}[\sigma_m(b)-\sigma_m(a)],
 \qquad0\leq a\leq b<1.
\]
Consequently,
\begin{equation}\label{eq:coordinate-shift-m}
 \begin{gathered}
 0\leq\sigma_m(\eta^s)-\sigma_m(\overline\eta)
 \leq C\rho_m(\tau)\sigma_m(\overline\eta),\\
 \sigma_m(\eta^s)\leq C\sigma_m(\overline\eta),
 \qquad1-\eta^s\leq1-\overline\eta.
 \end{gathered}
\end{equation}
The identity $1-\eta=\mu_m^{-1}e^{-\sigma_m(\eta)^2}$ and the inequality
$1-e^{-z}\leq z$ for $z\geq0$ give
\[
 0\leq\eta^s-\overline\eta
 \leq C\rho_m(\tau)(1-\overline\eta)
 \sigma_m(\overline\eta)^2,
\]
which proves the assertion for $\ell=0$.

For $1\leq\ell\leq5$, the Crocco transformation gives
\[
 \partial_y^\ell u^s
 =x^m{\Delta_m^{-\ell}}(w^s\partial_\eta)^{\ell-1}w^s(\tau,\eta^s),
 \qquad
 \partial_y^\ell\overline u_m
 =x^m{\Delta_m^{-\ell}}
 (\overline w_m\partial_\eta)^{\ell-1}\overline w_m(\overline\eta).
\]
Repeated differentiation of \eqref{eq:steady-self-similar-m} gives
\[
 \left|\partial_\eta^k\overline w_m(\eta)\right|
 \leq\frac{C}{(1-\eta)^{k-1}\sigma_m(\eta)},
 \qquad2\leq k\leq5.
\]
Together with \eqref{eq:ws-large-time} and the higher-order derivative
estimates \eqref{eq:ws-high-large-time} in
Theorem~\ref{thm:higher-order-regularity}, this yields
\begin{align*}
 \left|(w^s\partial_\eta)^{\ell-1}w^s
 -(\overline w_m\partial_\eta)^{\ell-1}\overline w_m\right|
 &\leq C\rho_m(\tau)(1-\eta)\sigma_m(\eta)^\ell,\\
 \left|\frac d{d\sigma_m}
 \left[(\overline w_m\partial_\eta)^{\ell-1}\overline w_m\right]
 \right|&\leq C(1-\eta)\sigma_m(\eta)^{\ell+1}.
\end{align*}
Splitting the difference at the common coordinate $\eta^s$ and using
\eqref{eq:coordinate-shift-m} therefore gives
\[
 |\partial_y^\ell(u^s-\overline u_m)|
 \leq C\rho_m(\tau)x^{\frac{(m-1)\ell}{2}}
 (x^m-\overline u_m)\overline\sigma_m^{\ell+2},
 \qquad0\leq\ell\leq5.
\]

It remains to estimate time derivatives at fixed $y$.  Set
\[
 J(\tau,\eta)=\int_0^\eta
 \frac{\partial_\tau w^s(\tau,\zeta)}{(w^s(\tau,\zeta))^2}\,d\zeta.
\]
Differentiating \eqref{eq:two-inverse-m} gives
\begin{equation}\label{eq:Qt-m}
 \partial_t\eta^s={\Delta_m^{-2}}w^s(\tau,\eta^s)J(\tau,\eta^s).
\end{equation}
The estimates for $w^s$ and $\partial_\tau w^s$ imply
\[
 [1+(m-1)\tau]|J(\tau,\eta)|
 \leq C\rho_m(\tau)\sigma_m(\eta),
\]
and successive applications of $w^s\partial_\eta$ give
\[
 [1+(m-1)\tau]|(w^s\partial_\eta)^\ell(w^sJ)(\tau,\eta)|
 \leq C\rho_m(\tau)(1-\eta)\sigma_m(\eta)^{\ell+2},
 \qquad0\leq\ell\leq3.
\]
Since $\partial_y={\Delta_m^{-1}}w^s\partial_\eta$ along the inverse map,
\eqref{eq:Qt-m} and \eqref{eq:coordinate-shift-m} give
\[
 [1+(m-1)\tau]|\partial_t\partial_y^\ell u^s|
 \leq C\rho_m(\tau)x^{\frac{(m-1)(\ell+2)}2}
 (x^m-\overline u_m)\overline\sigma_m^{\ell+2},
 \qquad0\leq\ell\leq3.
\]

For the second time derivative, set
\[
 K(\tau,\eta)=\int_0^\eta\left[
 \frac{\partial_{\tau\tau}w^s}{(w^s)^2}
 -2\frac{(\partial_\tau w^s)^2}{(w^s)^3}\right](\tau,\zeta)\,d\zeta.
\]
The mixed time-normal derivative estimates
\eqref{eq:ws-high-large-time} in
Theorem~\ref{thm:higher-order-regularity} give
\[
 [1+(m-1)\tau]^2|K(\tau,\eta)|
 \leq C\rho_m(\tau)\sigma_m(\eta),
 \qquad
 [1+(m-1)\tau]^2|\partial_\eta K(\tau,\eta)|
 \leq\frac{C\rho_m(\tau)}{(1-\eta)\sigma_m(\eta)}.
\]
Differentiating \eqref{eq:Qt-m} once more gives
\begin{equation}\label{eq:Qtt-m}
 \partial_{tt}\eta^s={\Delta_m^{-4}}
 \left[2\partial_\tau w^sJ+w^s\partial_\eta w^sJ^2+w^sK\right](\tau,\eta^s).
\end{equation}
Using the bounds above for $J$, $K$, and the derivatives of $w^s$, and applying
$w^s\partial_\eta$ at most once to the bracketed expression, we obtain
\[
 [1+(m-1)\tau]^2
 \left|(w^s\partial_\eta)^\ell
 \left(2\partial_\tau w^sJ+w^s\partial_\eta w^sJ^2+w^sK\right)\right|
 \leq C\rho_m(\tau)(1-\eta)\sigma_m(\eta)^{\ell+4},
 \qquad0\leq\ell\leq1.
\]
Equations~\eqref{eq:Qtt-m} and~\eqref{eq:coordinate-shift-m} prove the last
line of \eqref{eq:eulerian-high}.

Let $\psi^s(\tau,r)$ be normalized by
$\psi^s(\tau,0)=0$ and $\partial_r\psi^s=\eta^s$.  Then
\[
 \Psi^s(t,x,y)=x^m{\Delta_m}\psi^s(\tau,r),
 \qquad v^s=-\partial_x\Psi^s,
\]
satisfies $u^s=\partial_y\Psi^s$ and
$\partial_xu^s+\partial_yv^s=0$.  Since $w^s>0$ for
$0\leq\eta<1$ and the semi-similar Crocco equation is obtained from the
tangential Prandtl equation by the invertible Crocco transformation, the
inverse reconstruction also recovers the tangential momentum equation.  The
wall, outer-flow, and initial conditions follow from the inverse relations.

If $m>1$, then $r=x^{(m-1)/2}y\to0$ as $x\downarrow0$, and the three
right-hand sides in \eqref{eq:eulerian-high} contain, respectively, the
powers
\[
 x^{m+\ell(m-1)/2},\qquad
 x^{m+(\ell+2)(m-1)/2},\qquad
 x^{m+(\ell+4)(m-1)/2}.
\]
Moreover, $\psi^s(\tau,r)=O(r^2)$ at the wall, so
$\Psi^s=O(x^{(3m-1)/2}y^2)$ and $v^s\to0$.  If $m=1$, then
\[
 u^s(t,x,y)=x\eta^s(t,y),\qquad
 \Psi^s(t,x,y)=x\psi^s(t,y),\qquad
 v^s(t,x,y)=-\psi^s(t,y).
\]
Thus the reconstructed fields and all the stated quantities extend
continuously to $x=0$.
\end{proof}

\subsection*{Proof of Theorem~\ref{thm:eulerian-stability}}

\begin{proof}
Let
\[
 \eta=\frac{u(t,x,y)}{x^m},\qquad
 \eta^s=\frac{u^s(t,x,y)}{x^m},\qquad
 r=x^{\frac{m-1}{2}}y.
\]
At the same Eulerian point $(t,x,y)$, the inverse Crocco relations are
\begin{equation}\label{eq:full-background-inverse-m}
 r=\int_0^\eta\frac{d\zeta}{w(\tau,\ln x,\zeta)}
 =\int_0^{\eta^s}\frac{d\zeta}{w^s(\tau,\zeta)}.
\end{equation}
Theorem~\ref{thm:semi-self-similar-stability} and
$\overline w_m\leq w^s\leq C\overline w_m$ give
\begin{equation}\label{eq:full-steady-comparison-m}
 (1-C\delta\rho_m)\overline w_m\leq w
 \leq C(1+C\delta\rho_m)\overline w_m.
\end{equation}
Here and below $w$ is evaluated at $(\tau,\ln x,\zeta)$ and $w^s$ at
$(\tau,\zeta)$ unless otherwise indicated.  We decrease $\delta_0$ so that
$C\delta_0\leq\frac12$.

For any $\theta$ between $\eta$ and $\eta^s$, set
\[
    \overline r_\theta
    :=
    \int_0^\theta \frac{d\zeta}{\overline w_m(\zeta)}.
\]
The inverse relations and the preceding comparison imply, after
decreasing $\delta_0$ if necessary, that
\[
    \bigl(1-C\delta\rho_m(\tau)\bigr)r
    \leq
    \overline r_\theta
    \leq
    C\bigl(1+C\delta\rho_m(\tau)\bigr)r.
\]
At the steady Eulerian height $\overline r_\theta$, the Crocco
    coordinate is $\theta$.  Hence the steady estimate~\eqref{eq:oleinik-weighted-tail} gives
\[
    (1-\theta)\sigma_m(\theta)^j
    \leq
    C(1+\overline r_\theta)^{
        j-\frac{5m+1}{m+1}}
    \exp\left(
        -\frac{m+1}{4}\overline r_\theta^2
        +\frac{m+1}{2}\bar d_m\,\overline r_\theta
    \right),
    \qquad 0\leq j\leq 7.
\]

The two-sided comparison between $\overline r_\theta$ and $r$ implies,
uniformly for $0\leq j\leq 7$, that
\[
    (1+\overline r_\theta)^{
        j-\frac{5m+1}{m+1}}
    \leq
    C(1+r)^{
        j-\frac{5m+1}{m+1}}.
\]
Moreover, the lower comparison for $\overline r_\theta$, together
with an elementary completion of the square, yields
\begin{equation}
\label{eq:gaussian-completion-m}
    -\frac{m+1}{4}\overline r_\theta^2
    +\frac{m+1}{2}\bar d_m\,\overline r_\theta
    \leq
    -\left(
        \frac{m+1}{4}-C\delta\rho_m(\tau)
    \right)r^2
    +\frac{m+1}{2}\bar d_m\,r
    +C.
\end{equation}
Combining these estimates, and absorbing the last constant in
\eqref{eq:gaussian-completion-m} into $C$, we obtain
\begin{equation}
\label{eq:gaussian-coordinate-tail-m}
\begin{aligned}
    (1-\theta)\sigma_m(\theta)^j
    &\leq
    C(1+r)^{
        j-\frac{5m+1}{m+1}}                             \\
    &\quad\times
    \exp\left[
        -\left(
            \frac{m+1}{4}-C\delta\rho_m(\tau)
        \right)r^2
        +\frac{m+1}{2}\bar d_m\,r
    \right],
    \qquad 0\leq j\leq 7,
\end{aligned}
\end{equation}
for $\theta=\eta$, $\theta=\eta^s$, and every intermediate
coordinate.

Subtracting the two inverse relations yields
\[
 \left|\int_{\eta^s}^{\eta}\frac{d\zeta}{w^s(\tau,\zeta)}\right|
 \leq C\delta\rho_m(\tau)
 \int_0^\eta\frac{d\zeta}{w(\tau,\ln x,\zeta)}
 =C\delta\rho_m(\tau)r.
\]
The comparison between $d\zeta/w^s$ and $d\sigma_m$ therefore gives
\begin{equation}\label{eq:sigma-shift-m}
 |\sigma_m(\eta)-\sigma_m(\eta^s)|
 \leq C\delta\rho_m(\tau)r.
\end{equation}
The mean value theorem and \eqref{eq:gaussian-coordinate-tail-m} now imply
\begin{equation*}
   |u-u^s|\leq C\delta\rho_m(\tau)x^m r(1+r)\exp\left[-\left(\frac{m+1}{4}
 -C\delta\rho_m(\tau)\right)r^2+\frac{m+1}{2}\bar d_mr\right].
\end{equation*}

For the normal derivatives,
\[
 \partial_y^\ell u
 =x^m{\Delta_m^{-\ell}}(w\partial_\eta)^{\ell-1}w(\tau,\ln x,\eta),
 \qquad
 \partial_y^\ell u^s
 =x^m{\Delta_m^{-\ell}}(w^s\partial_\eta)^{\ell-1}w^s(\tau,\eta^s).
\]
The Crocco estimates give, for $1\leq\ell\leq3$,
\begin{align*}
 \left|(w\partial_\eta)^{\ell-1}w
 -(w^s\partial_\eta)^{\ell-1}w^s\right|
 &\leq C\delta\rho_m(\tau)(1-\theta)\sigma_m(\theta)^\ell,\\
 \left|\frac d{d\sigma_m}
 \left[(w^s\partial_\eta)^{\ell-1}w^s\right]\right|
 &\leq C(1-\theta)\sigma_m(\theta)^{\ell+1}.
\end{align*}
Splitting at the common coordinate $\eta$ and using
\eqref{eq:gaussian-coordinate-tail-m}--\eqref{eq:sigma-shift-m} yields the
normal-derivative estimate in the theorem.

Differentiating \eqref{eq:full-background-inverse-m} at fixed $(x,y)$ gives
\begin{align*}
 \partial_tu&=x^m{\Delta_m^{-2}}w(\tau,\ln x,\eta)
 \int_0^\eta\frac{\partial_\tau w(\tau,\ln x,\zeta)}
 {w(\tau,\ln x,\zeta)^2}\,d\zeta,\\
 \partial_tu^s&=x^m{\Delta_m^{-2}}w^s(\tau,\eta^s)
 \int_0^{\eta^s}\frac{\partial_\tau w^s(\tau,\zeta)}
 {(w^s(\tau,\zeta))^2}\,d\zeta.
\end{align*}
At the common coordinate $\theta$, the Crocco estimates imply
\begin{align*}
 \left|w\int_0^\theta\frac{\partial_\tau w}{w^2}\,d\zeta
 -w^s\int_0^\theta\frac{\partial_\tau w^s}{(w^s)^2}\,d\zeta\right|
 &\leq\frac{C\delta\rho_m(\tau)}{1+(m-1)\tau}
 (1-\theta)\sigma_m(\theta)^2,\\
 \left|\frac d{d\sigma_m}\left[
 w^s\int_0^\theta\frac{\partial_\tau w^s}{(w^s)^2}\,d\zeta\right]\right|
 &\leq\frac{C\rho_m(\tau)}{1+(m-1)\tau}
 (1-\theta)\sigma_m(\theta)^3.
\end{align*}
Splitting at $\eta$ and using
\eqref{eq:gaussian-coordinate-tail-m}--\eqref{eq:sigma-shift-m} gives
\begin{align*}
 |\partial_tu-\partial_tu^s|\leq{}&\frac{C\delta\rho_m(\tau)}{1+(m-1)\tau}
 x^m{\Delta_m^{-2}}(1+r)^4\\
 &\times\exp\left[-\left(\frac{m+1}{4}
 -C\delta\rho_m(\tau)\right)r^2
 +\frac{m+1}{2}\bar d_mr\right].
\end{align*}

For the mixed derivative, the same-coordinate estimates and the mean value
theorem give
\begin{align*}
 |\partial_\tau w(\tau,\ln x,\eta)
 -\partial_\tau w^s(\tau,\eta^s)|
 &\leq\frac{C\delta\rho_m(\tau)}{1+(m-1)\tau}(1+r)^3\\
 &\quad\times\exp\left[-\left(\frac{m+1}{4}
 -C\delta\rho_m(\tau)\right)r^2
 +\frac{m+1}{2}\bar d_mr\right],\\
 |\partial_\eta w(\tau,\ln x,\eta)
 -\partial_\eta w^s(\tau,\eta^s)|
 &\leq C\delta\rho_m(\tau)(1+r).
\end{align*}
Substituting these estimates into
\begin{align*}
 \partial_t\partial_yu-\partial_t\partial_yu^s={}&x^m{\Delta_m^{-3}}
 [\partial_\tau w(\tau,\ln x,\eta)
 -\partial_\tau w^s(\tau,\eta^s)]\\
 &+{\Delta_m^{-1}}(\partial_tu-\partial_tu^s)\partial_\eta w(\tau,\ln x,\eta)\\
 &+{\Delta_m^{-1}}\partial_tu^s[\partial_\eta w(\tau,\ln x,\eta)
 -\partial_\eta w^s(\tau,\eta^s)]
\end{align*}
gives
\begin{align*}
 |\partial_t\partial_yu-\partial_t\partial_yu^s|\leq{}&\frac{C\delta\rho_m(\tau)}{1+(m-1)\tau}
 x^m{\Delta_m^{-3}}(1+r)^5\\
 &\times\exp\left[-\left(\frac{m+1}{4}
 -C\delta\rho_m(\tau)\right)r^2
 +\frac{m+1}{2}\bar d_mr\right].
\end{align*}

Finally, the derivatives of the inverse coordinate at fixed $r$ are
\[
 \partial_\xi\eta=w(\eta)\int_0^\eta\frac{\partial_\xi w}{w^2}\,d\zeta,
 \qquad
 \partial_\tau\eta=w(\eta)\int_0^\eta\frac{\partial_\tau w}{w^2}\,d\zeta,
 \qquad \partial_r\eta=w(\eta),
\]
and the same formulas without the $\xi$-term hold for $\eta^s$.  Hence
\begin{align*}
 x\partial_xu&=x^m\left[m\eta+\partial_\xi\eta+(m-1)\tau\partial_\tau\eta
 +\frac{m-1}{2}r\partial_r\eta\right],\\
 x\partial_xu^s&=x^m\left[m\eta^s+(m-1)\tau\partial_\tau\eta^s
 +\frac{m-1}{2}r\partial_r\eta^s\right].
\end{align*}
The estimates for $\partial_\xi w$ and $\partial_\tau w$, together with
$0\leq(m-1)\tau/[1+(m-1)\tau]\leq1$, give
\begin{align*}
 |\partial_x(u-u^s)|\leq{}&C\delta\rho_m(\tau)x^{m-1}(1+r)^2\\*
 &\times\exp\left[-\left(\frac{m+1}{4}
 -C\delta\rho_m(\tau)\right)r^2
 +\frac{m+1}{2}\bar d_mr\right].
\end{align*}
Moreover,
\begin{align*}
 x\partial_x\partial_yu=x^m{\Delta_m^{-1}}\bigg\{&\frac{3m-1}{2}w+\partial_\xi w
 +(m-1)\tau\partial_\tau w\\
 &+\partial_\eta w\left[\partial_\xi\eta+(m-1)\tau\partial_\tau\eta
 +\frac{m-1}{2}r\partial_r\eta\right]\bigg\},
\end{align*}
with the corresponding formula for $x\partial_x\partial_yu^s$ obtained by deleting the
$\xi$-terms.  Therefore,
\begin{align*}
 |\partial_x\partial_y(u-u^s)|
 \leq{}&C\delta\rho_m(\tau)x^{m-1}{\Delta_m^{-1}}(1+r)^3\\*
 &\times\exp\left[-\left(\frac{m+1}{4}
 -C\delta\rho_m(\tau)\right)r^2
 +\frac{m+1}{2}\bar d_mr\right].
\end{align*}

The estimates are first proved for $x>0$.  Choose $\delta_0$ smaller, if
necessary, so that the quadratic coefficient in the Gaussian is positive,
and define
\[
 v(t,x,y)=v^s(t,x,y)
 -\int_0^y\partial_x\bigl(u(t,x,z)-u^s(t,x,z)\bigr)\,dz.
\]
The tangential estimate makes this integral convergent and gives
$\partial_xu+\partial_yv=0$ and $v=0$ at the wall.  The invertibility of the
Crocco transformation yields the tangential Prandtl equation, the initial
datum $u_0$ at $t=0$, and the outer velocity as $\eta\to1$.

If $m>1$, the four right-hand sides in
Theorem~\ref{thm:eulerian-stability} contain, respectively, the powers
\[
 x^{\frac{3m-1}{2}},\qquad
 x^{m+\ell(m-1)/2},\qquad
 x^{m-1+\ell(m-1)/2},\qquad
 x^{m+(\ell+2)(m-1)/2},
\]
which are positive in their stated ranges.  Moreover, integration of the
$\partial_xu-\partial_xu^s$ estimate gives
$|v-v^s|\leq C\delta\rho_m x^{(m-1)/2}$ for fixed $(t,y)$.
If $m=1$, then $\tau=t$, $r=y$, and $x\partial_x=\partial_\xi$.
The entrance conditions imply
\[
 w(\tau,\xi,\eta)\longrightarrow w^s(\tau,\eta),\qquad
 \partial_\xi w(\tau,\xi,\eta)\longrightarrow0
 \quad\text{as }\xi\to-\infty,
\]
together with the corresponding weighted derivatives used above.
Therefore $u=x\eta\to0$ and
$v-v^s=-\int_0^y(\partial_xu-\partial_xu^s)\,dz\to0$.
This completes the proof.
\end{proof}
\appendix

{
\section{{Steady favorable profiles for fixed $m\geq1$}}
\label{app:steady-profile}

{The classical construction of Oleinik yields a positive favorable
Falkner--Skan profile; see, for example, \cite{OS}.  In this appendix, we
derive directly from the steady equation all the estimates required in
Lemma~\ref{self-similar-properties} and conclude with a short uniqueness
proof.  Throughout, the constants may depend on the fixed exponent $m$.}

\begin{proof}[Proof of Lemma~\ref{self-similar-properties}]
{Since $\eta=f_m'(r)$ and $\overline w_m(\eta)=f_m''(r)$,
differentiating \eqref{eq:falkner-skan-family} with respect to $\eta$ yields
\eqref{eq:steady-self-similar-m}.  Evaluating
\eqref{eq:falkner-skan-family} at $r=0$ gives $f_m'''(0)+m=0$, which is
equivalent to $\overline w_m(0)\overline w_m'(0)+m=0$.
For the remainder of the proof, we abbreviate
$\overline w=\overline w_m$ and $\sigma=\sigma_m$.}

{We begin with the bounds in
\eqref{eq:steady-self-similar-properties-m}.  The wall condition implies
$\overline w'(0)<0$.  If $\overline w'$ vanished at an interior point, the
steady equation would imply}
\[
 \overline w''=\frac{3m-1}{2}
 \frac{\eta}{\overline w}>0
\]
{at that point.  Thus any zero of $\overline w'$ would correspond to a
crossing from negative to positive.  Such a crossing is impossible because
$\overline w>0$ on $[0,1)$ and $\overline w(1)=0$.  Therefore
$\overline w'<0$.}

{Set $X=1-\eta$.  For a positive test function $\phi$, denote the
left-hand side of the steady equation by}
\[
 \mathcal Q[\phi]
 :=\phi^2\phi''+m(\eta^2-1)\phi'
   -\frac{3m-1}{2}\eta\phi.
\]
{The comparison principle used below follows directly by subtraction:}
\begin{align*}
 -\mathcal Q[\phi]
 ={}&\overline w^2(\overline w-\phi)''
 +m(\eta^2-1)(\overline w-\phi)'\\
 &+\left[(\overline w+\phi)\phi''
 -\frac{3m-1}{2}\eta\right](\overline w-\phi).
\end{align*}
{For the barriers considered below, $\phi''<0$, so the zero-order
coefficient is negative.  Consequently, $\mathcal Q[\phi]\leq0$ makes
$\phi$ an upper barrier, whereas $\mathcal Q[\phi]\geq0$ makes it a lower
barrier.  The same conclusions hold at the wall provided, respectively, that}
$\phi(0)\phi'(0)+m\leq0$ or
$\phi(0)\phi'(0)+m\geq0$.  {For example, suppose that
$\overline w(0)>\phi(0)$ at a wall maximum of
$\overline w-\phi$.  Then
$\overline w'(0)=-m/\overline w(0)>-m/\phi(0)\geq\phi'(0)$,
contradicting the required sign of the right derivative at that maximum.}

{Choose $-\log\mu_m=(3m+1)/[2(m+1)]$, so that
$\sigma^2=-\log(\mu_mX)$, and define}
$\Phi=\sqrt{m+1}\,X\sigma$.  A direct differentiation gives
\begin{equation}\label{eq:appendix-upper-barrier}
 \frac{\mathcal Q[\Phi]}{\Phi}
 =\frac{m-1}{2}X
 -\frac{5m+1-2mX}{4\sigma^2}\leq0.
\end{equation}
{To verify the final inequality, observe that
$X\sigma^2=X[-\log\mu_m-\log X]\leq-\log\mu_m$ on $(0,1]$.
Moreover, $\Phi(0)\Phi'(0)+m=0$.  The comparison principle established above
therefore yields}
\begin{equation}\label{eq:appendix-upper-profile}
 \overline w(\eta)\leq\sqrt{m+1}(1-\eta)\sigma(\eta).
\end{equation}

{For $A<\sqrt{m+1}$, the same calculation shows that
$\mathcal Q[A X\sigma]>0$ when $\eta$ is sufficiently close to $1$.
After decreasing $A$ if necessary to satisfy the lower-barrier condition at
the left endpoint of this terminal interval, the comparison principle gives}
\[
 A(1-\eta)\sigma(\eta)\leq\overline w(\eta).
\]
{On the remaining compact interval, positivity and continuity give the
same inequality after possibly decreasing $A$ once more.  This proves the
first two-sided bound in \eqref{eq:steady-self-similar-properties-m}.}

{Dividing the steady equation by $\overline w$ and integrating once from
the wall gives the exact identity}
\begin{equation}\label{eq:appendix-profile-derivative}
 \overline w'
 =-\frac{m(1-\eta^2)}{\overline w}
 -\frac{m+1}{2}\int_0^\eta\frac{s}{\overline w(s)}\,ds.
\end{equation}
{Multiplying by $\overline w$ gives the equivalent identity for
$\overline w\overline w'$.  Combining the preceding two-sided bound with}
$\int[(1-\eta)\sigma(\eta)]^{-1}\,d\eta=2\sigma(\eta)$
{in \eqref{eq:appendix-profile-derivative}, we obtain}
\[
 -M_3\sigma\leq\overline w'\leq-M_4\sigma.
\]
{Substitution into the steady equation now gives
$|\overline w\overline w''|\leq C$.}

{We next determine the sign of $\overline w\overline w''$.  At the wall,}
$\overline w(0)\overline w''(0)=-m^2/\overline w(0)^2<0$.
{At any zero of $\overline w\overline w''$, differentiation of the steady
equation gives}
\[
 (\overline w\overline w'')'
 =\frac{3m-1}{2}
 +\frac{3m-1}{2}\frac{\eta^2}{1-\eta^2}
   \frac{m+1}{2m}>0.
\]
{Thus any zero would be a crossing from negative to positive.  Once such a
crossing occurred, $\overline w''$ would remain positive and $\overline w'$
would be increasing, contradicting
$\overline w'(\eta)\leq-M_4\sigma(\eta)\to-\infty$.  We conclude that
$\overline w\overline w''<0$ on $[0,1)$.}

{We now prove the refined edge estimate.  Retaining $X=1-\eta$, define}
$W(X)=\overline w(1-X)$,
$\Sigma=-\log(\mu_mX)=\sigma^2$, and $R=XW_X/W$.
{The bounds established above imply}
\[
 0<C^{-1}\leq R\leq C,\qquad
 |WW_{XX}|\leq C,\qquad WW_{XX}<0.
\]
{The steady equation, together with a direct differentiation of $R$, gives}
\begin{align}
 WW_{XX}&=\frac{3m-1}{2}(1-X)-m(2-X)R,
 \label{eq:appendix-curvature-ratio}\\
 R_\Sigma&=R(R-1)-\left(\frac{X}{W}\right)^2WW_{XX}.
 \label{eq:appendix-ratio-equation}
\end{align}
{The final term in \eqref{eq:appendix-ratio-equation} is nonnegative and
bounded above by $C/\Sigma$.  If}
$R-1>K\Sigma^{-1/2}$ for a sufficiently large value of $\Sigma$, then
$R_\Sigma\geq C^{-1}(R-1)>0$ as long as this inequality holds; since
$K\Sigma^{-1/2}$ decreases, this contradicts the upper bound for $R$.
{On the other hand, if}
$R-1<-K\Sigma^{-1/2}$, then, after increasing the initial value of $\Sigma$,
the $C/\Sigma$ term is at most one half of $C^{-1}|R-1|$, and hence
$R_\Sigma\leq\tfrac12C^{-1}(R-1)<0$.  Since $-K\Sigma^{-1/2}$ increases, this
contradicts the positive lower bound for $R$.  {Therefore,}
\[
 R=1+O(\Sigma^{-1/2}).
\]
{Substituting this estimate into
\eqref{eq:appendix-curvature-ratio} yields}
\[
 \overline w\overline w''
 =-\frac{m+1}{2}+O(\sigma^{-1}).
\]
{  Since}
$d(W_X^2)/d\Sigma=-2RWW_{XX}$, integration first gives
$W_X=\sqrt{m+1}\sqrt\Sigma+O(1)$ and
$W/X=\sqrt{m+1}\sqrt\Sigma+O(1)$.  {It follows that}
\[
 -\left(\frac{X}{W}\right)^2WW_{XX}
 =\frac{1}{2\Sigma}+O(\Sigma^{-3/2}).
\]
{Repeating the preceding comparison in
\eqref{eq:appendix-ratio-equation} with the barriers
$\pm K\Sigma^{-1}$ sharpens the estimate to $R=1+O(\Sigma^{-1})$.
Substitution into \eqref{eq:appendix-curvature-ratio}, followed by one more
integration, yields}
\begin{equation}\label{eq:appendix-fs-tail}
 \begin{aligned}
 WW_{XX}&=-\frac{m+1}{2}+O(\Sigma^{-1}),\\
 W_X&=\sqrt{m+1}\sqrt\Sigma
       \left[1+O\left(\frac{1+\log\Sigma}{\Sigma}\right)\right],\\
 \frac{W}{X}&=\sqrt{m+1}\sqrt\Sigma
       \left[1+O\left(\frac{1+\log\Sigma}{\Sigma}\right)\right],
\end{aligned}
\end{equation}
{which is \eqref{eq:steady-edge-curvature-family}. Consequently, there exist $K_m>0$ and $\eta_0<1$ such that}
\[
 \sqrt{m+1}(1-\eta)(\sigma-K_m)
 \leq\overline w(\eta),\qquad \eta_0<\eta<1.
\]
{Together with \eqref{eq:appendix-upper-profile}, this proves the terminal
two-sided estimate with $M_2=\sqrt{m+1}$.  The refined limit and strict
negativity near the edge, combined with compactness on $[0,\eta_0]$, give
$-M_5\leq\overline w\overline w''\leq-M_6$.}

{It remains to prove uniqueness.  Let $\overline w^{(1)}$ and
$\overline w^{(2)}$ be two positive solutions of
\eqref{eq:steady-self-similar-m}, and put
$z=\overline w^{(1)}-\overline w^{(2)}$.  Subtracting their equations gives}
\[
 (\overline w^{(1)})^2z''+m(\eta^2-1)z'
 +\left[(\overline w^{(1)}+\overline w^{(2)})
 (\overline w^{(2)})''
 -\frac{3m-1}{2}\eta\right]z=0,
\]
with
\[
 z(1)=0,\qquad
 z'(0)-\frac{m}{\overline w^{(1)}(0)\overline w^{(2)}(0)}z(0)=0.
\]
{The zero-order coefficient is strictly negative because
$(\overline w^{(2)})''<0$.  Therefore, a positive maximum of $z$ can occur
neither in $(0,1)$ nor at $\eta=0$, where the Robin condition forces the
right derivative to have the opposite sign.  Hence $z\leq0$.  Applying the
same argument to $-z$ yields $z\equiv0$ and completes the proof.}
\end{proof}
}

{
\section{Critical edge packets and nonuniform dependence}
\label{app:edge-counterexamples}

In this appendix, we construct counterexamples to uniform continuity of the
solution map in the weighted $C^2$ topology.  We use the notation and the
weighted edge space from Section~\ref{sec:weighted-c2-counterexample}.

{
\begin{theorem}[Nonuniform dependence in the weighted $C^2$ topology]
\label{thm:edge-nonuniformity-eta}
Fix $m\geq1$.  There exist $\overline\tau>0$ and a sequence $\delta_n\to0$,
together with positive smooth initial data $w_{0,n}^{(i)}(\eta)$, $i=1,2$,
satisfying the hypotheses of
Theorem~\ref{thm:ws-regularity}, with the constants $E_1,E_2,E_3$ independent
of $n$ and $i$.  Let $w_n^{(i)}(\tau,\eta)$ be the corresponding global
solutions.  Then, for some constant $C>0$ independent of $n$ and $i$,
\begin{equation*}
 1\leq\frac{w_n^{(i)}(\tau,\eta)}{\overline w_m(\eta)}
 \leq1+C\delta_n,
 \qquad \tau\geq0,\quad0\leq\eta<1.
\end{equation*}
Define
\begin{align*}
 \mathcal D_n(\tau):={}&
 \left\|\frac{w_n^{(2)}(\tau,\cdot)-w_n^{(1)}(\tau,\cdot)}
 {\overline w_m}\right\|_{L^\infty(0,1)}\\
 &+\left\|\frac{\partial_\eta w_n^{(2)}(\tau,\cdot)
 -\partial_\eta w_n^{(1)}(\tau,\cdot)}
 {-\overline w_m'}\right\|_{L^\infty(0,1)}\\
 &+\left\|\frac{\partial_{\eta\eta}w_n^{(2)}(\tau,\cdot)
 -\partial_{\eta\eta}w_n^{(1)}(\tau,\cdot)}
 {-\overline w_m''}\right\|_{L^\infty(0,1)}.
\end{align*}
Then
\begin{equation*}
 \mathcal D_n(0)\longrightarrow0.
\end{equation*}
Moreover, for every fixed $\tau_0\in(0,\overline\tau]$, there exists
$c_0=c_0(\tau_0)>0$ such that
\begin{equation*}
 \mathcal D_n(\tau_0)\geq c_0
\end{equation*}
for all sufficiently large $n$.  Thus the solution map is not uniformly
continuous at any fixed time $\tau_0\in(0,\overline\tau]$ on this bounded
weighted $C^2$ class.
\end{theorem}
}

We first recall the equation for the relative profile.  Set
\[
 s=\int_0^\tau\frac{d\theta}{1+(m-1)\theta},\qquad
 \alpha(s,\eta)=\eta+(1-\eta)e^{-(m-1)s},
\]
and
\[
 \chi=\log\bigl(2\sqrt{-\log(1-\eta)}\bigr),\qquad
 h(s,\chi)=\frac{w(s,\eta)}{\overline w_m(\eta)}.
\]
Then \eqref{eq:edge-relative-equation} reads
\begin{equation*}
 h_s=\nu h^2h_{\chi\chi}+(d_1h^2+d_0)h_\chi
      +g h(h^2-1),
\end{equation*}
where
\begin{equation*}
\begin{aligned}
 \nu&=\frac{\overline w_m^2\chi_\eta^2}{\alpha},\qquad
 d_1=\frac1\alpha\bigl(\overline w_m^2\chi_{\eta\eta}
           +2\overline w_m\overline w_m'\chi_\eta\bigr),\\
 d_0&=\frac{m(\eta^2-1)\chi_\eta}{\alpha},\qquad
 g=\frac{\overline w_m\overline w_m''}{\alpha}.
\end{aligned}
\end{equation*}

\begin{proof}[Proof of \eqref{eq:edge-coefficient-limits}]
Write
\[
 X=1-\eta,\qquad L=-\log X,\qquad
 \Sigma=-\log(\mu_mX)=L-\log\mu_m,
\]
and set
\[
 W(X)=\overline w_m(1-X),\qquad
 B=\frac{W}{X},\qquad R=\frac{XW_X}{W}.
\]
As $L\to+\infty$, the estimates obtained above and in
\eqref{eq:appendix-fs-tail} give
\begin{equation*}
\begin{aligned}
 B^2&=(m+1)L+O(1+\log L),\\
 R&=1+O(L^{-1}),\\
 WW_{XX}&=-\frac{m+1}{2}+O(L^{-1}).
\end{aligned}
\end{equation*}
Also,
\[
 e^{2\chi}=4L,\qquad
 \chi_\eta=\frac1{2XL},\qquad
 \chi_{\eta\eta}=\frac{L-1}{2X^2L^2},\qquad
 X_\chi=-2XL.
\]
For $0\leq s\leq T$,
\[
 \alpha=1-\bigl(1-e^{-(m-1)s}\bigr)X,\quad
 \alpha^{-1}=1+O_T(X),\qquad
 \alpha_\chi=O_T(LX).
\]
Since $X=e^{-L}$, every term containing $X$ can be absorbed into the error
terms below.

For the diffusion coefficient, we have
\begin{align*}
 \nu
 &=\frac{B^2}{4\alpha L^2}\\
 &=\frac{m+1}{4L}
   +O_T\left(\frac{1+\log L}{L^2}\right)\\
 &=(m+1)e^{-2\chi}
   +O_T\bigl((1+\chi)e^{-4\chi}\bigr).
\end{align*}
To estimate its derivative, we differentiate the exact formula for \(\nu\):
\begin{equation*}
 \frac{\nu_\chi}{\nu}
 =2\frac{W_\chi}{W}
  +2\frac{(\chi_\eta)_\chi}{\chi_\eta}
  -\frac{\alpha_\chi}{\alpha}
 =4L(1-R)-4-\frac{\alpha_\chi}{\alpha}.
\end{equation*}
Thus
\[
 \nu_\chi=O_T(e^{-2\chi})
 =O_T\bigl((1+\chi)e^{-2\chi}\bigr).
\]
Next,
\begin{align*}
 d_1
 &=\frac{B^2}{\alpha}
   \left(\frac{L-1}{2L^2}-\frac{R}{L}\right)\\
 &=-\frac{m+1}{2}
   +O_T\bigl((1+\chi)e^{-2\chi}\bigr),
\end{align*}
while
\begin{equation*}
 d_0=-\frac{m(2-X)}{2\alpha L}=O_T(e^{-2\chi}).
\end{equation*}
Finally,
\begin{equation*}
 g=\frac{WW_{XX}}{\alpha}
 =-\frac{m+1}{2}+O_T(e^{-2\chi}).
\end{equation*}
These bounds prove \eqref{eq:edge-coefficient-limits}.
\end{proof}

\begin{proof}[Proof of Theorem~\ref{thm:edge-nonuniformity-eta}]
We now construct the counterexample.  Fix $c>0$.  For sufficiently small
$\delta>0$, the implicit function theorem gives a unique small number
$A_{c,\delta}>0$ satisfying
\begin{equation}\label{eq:edge-background-robin-equation}
 2A_{c,\delta}(1+c\delta-A_{c,\delta})\overline w_m(0)^2
 =m\bigl[(1+c\delta-A_{c,\delta})^2-1\bigr].
\end{equation}
It also gives
\begin{equation}\label{eq:edge-background-A-expansion}
 A_{c,\delta}=\frac{cm}{\overline w_m(0)^2+m}\delta+O(\delta^2).
\end{equation}
Set
\[
 G_{c,\delta}(\eta)=1+c\delta-A_{c,\delta}(1-\eta)^2.
\]
By \eqref{eq:edge-background-robin-equation},
$\overline w_mG_{c,\delta}$ satisfies the Robin condition.  Moreover,
\[
 G_{c,\delta}>1,\qquad G_{c,\delta}'\geq0,\qquad
 G_{c,\delta}''<0.
\]
For any positive function $G$, substituting
$w^s(0,\eta)=\overline w_m(\eta)G(\eta)$ gives
\begin{equation}\label{eq:edge-background-residual}
 \frac{\partial_\tau w^s(0,\eta)}{\overline w_m}
 =G^2\overline w_m^2G''+2G^2\overline w_m\overline w_m'G'
   +m(\eta^2-1)G'
   +G(G^2-1)\overline w_m\overline w_m''.
\end{equation}
For $G=G_{c,\delta}$, all four terms are nonpositive.  Since
$\inf(-\overline w_m\overline w_m'')>0$, the last term together with
\eqref{eq:edge-background-A-expansion} gives
\begin{equation}\label{eq:edge-background-time-margin}
 \partial_\tau w^s(0,\eta)
 \leq-c_{\mathrm{tm}}\delta\overline w_m(\eta),
 \qquad w^s(0,\eta)=\overline w_m(\eta)G_{c,\delta}(\eta),
\end{equation}
where $c_{\mathrm{tm}}>0$ is independent of $\delta$ when $\delta$ is sufficiently small.

Choose a fixed smooth nondecreasing function $\Phi$ such that
\[
 0\leq\Phi\leq1,\qquad \Phi(z)=0\ (z\leq-1),\qquad
 \Phi(z)=1\ (z\geq1),\qquad \Phi'\geq0,
\]
and choose $z_*$ with $\Phi''(z_*)\neq0$.  Let
\[
 X_n\to+\infty,\qquad \delta_n=e^{-X_n},\qquad c_1=2,\quad c_2=3,
\]
and, for a fixed sufficiently small $\varepsilon_*>0$, set
\begin{equation}\label{eq:edge-packet-data}
 h_{0,n}^{(i)}(\chi)
 =G_{c_i,\delta_n}(\eta(\chi))
 +\varepsilon_*\delta_n
 \Phi\left(\frac{\chi-X_n}{\delta_n}\right),\qquad
 w_{0,n}^{(i)}=\overline w_mh_{0,n}^{(i)},\qquad i=1,2.
\end{equation}
The added packet is zero near the wall, so the Robin condition is unchanged.
On the support of the derivatives of the packet,
$e^{-2}\delta_n^2\leq e^{-2\chi}\leq e^2\delta_n^2$, while
\[
 (h_{0,n}^{(i)})_\chi=(G_{c_i,\delta_n})_\chi
       +\varepsilon_*\Phi'(z),\qquad
 (h_{0,n}^{(i)})_{\chi\chi}=(G_{c_i,\delta_n})_{\chi\chi}
       +\frac{\varepsilon_*}{\delta_n}\Phi''(z),\qquad
 z=\frac{\chi-X_n}{\delta_n}.
\]
Lemma~\ref{lem:natural-edge-norm-equivalence} shows that the two sequences are
uniformly bounded in $\mathcal Y_m^2$.  Also,
$(h_{0,n}^{(i)})_\chi\geq0$ and $h_{0,n}^{(i)}-1\geq b\delta_n$.  On
$\chi\geq X_n-\delta_n$, \eqref{eq:edge-coefficient-limits} and the preceding
identities give
\[
 \frac{\partial_\tau w_n^{(i)}(0,\eta)}{\overline w_m(\eta)}
 \leq C\varepsilon_*\delta_n+o(\delta_n)
       -\frac{m+1}{4}b\delta_n.
\]
For $\chi\leq X_n-\delta_n$, the packet is zero, so
\eqref{eq:edge-background-time-margin} applies.  We first choose
$\varepsilon_*$ sufficiently small and then take $n$ sufficiently large.  We
obtain
\begin{equation}\label{eq:edge-packet-time-sign}
 \partial_\tau w_{n}^{(i)}(0,\eta)\leq0,\qquad i=1,2.
\end{equation}
{
Thus, in essential contrast with the oscillatory datum in the first
counterexample, both sequences constructed here do satisfy the one-sided
initial time-monotonicity condition.  The nonuniform dependence below
therefore occurs within the time-monotone class, rather than through a failure
of that condition.
}
Proposition~\ref{prop:time-monotonicity} and
Theorem~\ref{thm:ws-regularity} then give global solutions.
After the time change in \eqref{eq:edge-common-time}, we write
\[
 w_n^{(i)}(s,\eta):=w_n^{(i)}(\tau(s),\eta),\qquad
 \tau(s)=
 \begin{cases}
  s,&m=1,\\
  \displaystyle\frac{e^{(m-1)s}-1}{m-1},&m>1.
 \end{cases}
\]
These solutions satisfy
\begin{equation}\label{eq:edge-packet-zero-order}
 \overline w_m\leq w_n^{(i)}(s,\eta)\leq w_{0,n}^{(i)}(\eta)
 \leq(1+C\delta_n)\overline w_m.
\end{equation}
Since the added packets are the same for $i=1,2$,
\eqref{eq:edge-background-A-expansion} and
Lemma~\ref{lem:natural-edge-norm-equivalence} give
\begin{equation}\label{eq:edge-packet-initial-distance}
 \|w_{0,n}^{(2)}-w_{0,n}^{(1)}\|_{\mathcal Y_m^2}\leq C\delta_n\to0.
\end{equation}
To show separation at a positive time, set
\begin{equation}\label{eq:edge-packet-rescaling}
 \zeta=\frac{\chi-X_n-\frac{m+1}{2}s}{\delta_n},\qquad
 h_n^{(i)}(s,\chi)=1+\delta_n H_n^{(i)}(s,\zeta).
\end{equation}
By \eqref{eq:edge-packet-zero-order}, the functions $H_n^{(i)}$ are uniformly
bounded.  They satisfy
\begin{align}\label{eq:edge-packet-rescaled-equation}
 (H_n)_s={}&\frac{\nu}{\delta_n^2}(1+\delta_nH_n)^2(H_n)_{\zeta\zeta}
 +\frac{\frac{m+1}{2}+d_1(1+\delta_nH_n)^2+d_0}{\delta_n}
   (H_n)_\zeta\nonumber\\
 &+g(1+\delta_nH_n)(2H_n+\delta_nH_n^2).
\end{align}
Here the coefficients are evaluated at
$\chi=X_n+\frac{m+1}{2}s+\delta_n\zeta$.  Uniformly on compact cylinders,
\[
 \frac{\nu}{\delta_n^2}\to(m+1)e^{-(m+1)s},\qquad
 \frac{\frac{m+1}{2}+d_1(1+\delta_nH_n)^2+d_0}{\delta_n}
 =-(m+1)H_n+o(1),\qquad
 g\to-\frac{m+1}{2}.
\]
{
To justify the compactness step, fix \(0<S,R<\infty\) and
consider \(Q_{S,R}=[0,S]\times[-R,R]\).  The uniform bound for \(H_n^{(i)}\),
the first convergence above, and \(1+\delta_nH_n^{(i)}\to1\) show that, for all
sufficiently large \(n\), the principal coefficient satisfies
\[
 0<\lambda_{S,R}\leq
 \frac{\nu}{\delta_n^2}(1+\delta_nH_n^{(i)})^2
 \leq\Lambda_{S,R}<\infty
 \qquad\hbox{on }Q_{S,R}.
\]
Hence the rescaled equations are uniformly parabolic on every compact
cylinder.  After the cancellation in the middle coefficient displayed above,
the exact coefficient formulas and their tail asymptotics also give uniform
local bounds for the drift and reaction terms.  Regarded as functions of
\((s,\zeta,r)\), the coefficients have uniform local Lipschitz bounds for
\(r\) in the bounded range of \(H_n^{(i)}\).  Local parabolic H\"older estimates
first give a uniform \(C^{\alpha/2,\alpha}\) bound; inserting this regularity
into the coefficients and applying the interior Schauder estimates then
yields, for some \(\alpha\in(0,1)\),
\[
 \|H_n^{(i)}\|_{C^{1+\alpha/2,\,2+\alpha}(Q_{S,R})}
 \leq C_{S,R}.
\]
At \(s=0\), the same estimate up to the initial face follows from the smooth
local convergence of the rescaled initial data; one applies the estimate on a
slightly larger spatial cylinder.  The compact embedding of these parabolic
H\"older spaces and a diagonal argument now allow us to pass to a subsequence
such that
\[
 H_n^{(i)}\longrightarrow H^{(i)}\quad\hbox{in }C^1_{\rm loc},
\]
}
and each limit satisfies
\begin{equation}\label{eq:edge-packet-limit-equation}
 H_s=(m+1)e^{-(m+1)s}H_{\zeta\zeta}
 -(m+1)HH_\zeta-(m+1)H,
\end{equation}
with
\begin{equation}\label{eq:edge-packet-limit-data}
 H^{(1)}(0,\zeta)=2+\varepsilon_*\Phi(\zeta),\qquad
 H^{(2)}(0,\zeta)=3+\varepsilon_*\Phi(\zeta).
\end{equation}
Uniqueness of bounded classical solutions shows that each of the two full sequences
converges to its corresponding limit.

Differentiating \eqref{eq:edge-packet-limit-equation} at $s=0$ and using
\eqref{eq:edge-packet-limit-data} gives
\[
 \left.\partial_s\bigl(H_\zeta^{(2)}-H_\zeta^{(1)}\bigr)\right|_{s=0}
 =-(m+1)\varepsilon_*\Phi''(\zeta).
\]
Set
\[
 F(s)=H_\zeta^{(2)}(s,z_*)-H_\zeta^{(1)}(s,z_*),
 \qquad
 K=(m+1)\varepsilon_*|\Phi''(z_*)|>0.
\]
By \eqref{eq:edge-packet-limit-data}, $F(0)=0$, and the preceding identity
gives $|F'(0)|=K$.  Hence there exists $\overline s>0$ such that
\begin{equation}\label{eq:edge-packet-small-time-separation}
 |F(s)|\geq\frac K2s,\qquad 0<s\leq\overline s.
\end{equation}
The $C^1_{\rm loc}$ convergence and the identity
\[
 (h_n^{(i)})_\chi=(H_n^{(i)})_\zeta
\]
imply that, for every fixed $s_0\in(0,\overline s]$ and all sufficiently
large $n$,
\[
 \|(h_n^{(2)})_\chi(s_0)-(h_n^{(1)})_\chi(s_0)\|_\infty
 \geq\frac K4s_0.
\]
Lemma~\ref{lem:natural-edge-norm-equivalence} then gives
\begin{equation}\label{eq:edge-packet-positive-distance}
 \|w_n^{(2)}(s_0)-w_n^{(1)}(s_0)\|_{\mathcal Y_m^2}
 \geq c_{\mathrm{sep}}s_0
\end{equation}
for all sufficiently large $n$, where $c_{\mathrm{sep}}>0$ is independent of $n$ and
$s_0$.  Define
\begin{equation*}
 \overline\tau=
 \begin{cases}
  \overline s,&m=1,\\[1mm]
  \displaystyle\frac{e^{(m-1)\overline s}-1}{m-1},&m>1.
 \end{cases}
\end{equation*}
The construction, \eqref{eq:edge-packet-time-sign}, and
Lemma~\ref{lem:natural-edge-norm-equivalence} show that the initial data
satisfy the hypotheses of Theorem~\ref{thm:ws-regularity} with uniform
constants.  After returning to the variable $\tau$,
\eqref{eq:edge-packet-zero-order} gives the pointwise bound in the theorem.
Moreover, \eqref{eq:edge-packet-initial-distance} gives
$\mathcal D_n(0)\to0$.  For any fixed
$\tau_0\in(0,\overline\tau]$, let
\begin{equation*}
 s_0=
 \begin{cases}
  \tau_0,&m=1,\\[1mm]
  \displaystyle\frac{\log[1+(m-1)\tau_0]}{m-1},&m>1.
 \end{cases}
\end{equation*}
Then $s_0\in(0,\overline s]$, and
\eqref{eq:edge-packet-positive-distance} gives
$\mathcal D_n(\tau_0)\geq c_{\mathrm{sep}}s_0$.  Taking
$c_0(\tau_0)=c_{\mathrm{sep}}s_0$ proves the theorem and the failure of uniform continuity
stated in Remark~\ref{rem:nonuniform-solution-map}.
\end{proof}
}

\section*{Acknowledgments}
Z. Xin is supported in part by the Zheng Ge Ru Foundation; the GS Charity Foundation Limited project ``Fundamental Mathematics Talent Development of IMS''; and Hong Kong RGC Earmarked Research Grants CUHK-14304325, CUHK-14301423, CUHK-14301421, and CUHK-14302819.
T. Tao is supported in part by the NSF of China under Grant 12371234.
C. Gao is supported in part by the NSF of China under Grant 12494541.
We used Eureka, an AI agent developed at the University of Science and
Technology of China, in constructing the two counterexamples presented in
this paper. Eureka also initially suggested a more delicate approach to the one-sided derivative estimate for $m>1$; the authors subsequently found a simpler argument, which is the one used in the paper. We used OpenAI's GPT-5.6 Sol to proofread and polish the manuscript.

\section*{Data Availability Statement}
No datasets were generated or analyzed in this study because the work is purely theoretical.

\section*{Conflict of Interest}
On behalf of all authors, the corresponding author states that there is no conflict of interest.


\begin{thebibliography}{99}

\bibitem{Akamatsu} T. Akamatsu, \textit{On the incompressible laminar boundary layer for impulsive motion of a semi-infinite flat plate}, Bull. JSME 10 (1967), no. 40, 641--646.

\bibitem{AWXY} R. Alexandre, Y.-G. Wang, C.-J. Xu, and T. Yang, \textit{Well-posedness of the Prandtl equation in Sobolev spaces}, J. Amer. Math. Soc. 28 (2015), 745--784.

\bibitem{B} H. Blasius, \textit{Grenzschichten in Fl\"{u}ssigkeiten mit kleiner Reibung}, Z. Math. Phys. 56 (1908), 1--37.

\bibitem{Dennis} S. C. R. Dennis, \textit{The motion of a viscous fluid past an impulsively started semi-infinite flat plate}, IMA J. Appl. Math. 10 (1972), no. 1, 105--117.

\bibitem{EE} W. E and B. Engquist, \textit{Blowup of solutions of the unsteady Prandtl's equation}, Comm. Pure Appl. Math. 50 (1997), no. 12, 1287--1293.

\bibitem{FS} V. M. Falkner and S. W. Skan, \textit{Some approximate solutions of the boundary-layer equations}, Philos. Mag. (7) 12 (1931), no. 80, 865--896.

\bibitem{GX} {C. Gao and Z. Xin, \textit{Prandtl boundary layers in an infinitely long convergent channel}, Comm. Math. Phys. 407 (2026), Paper No.~26.}

\bibitem{GZ} C. Gao and C. Zhao, \textit{Sharp asymptotic behavior of the steady pressure-free Prandtl system}, arXiv:2504.11870 (2025).

\bibitem{GZZ} C. Gao, L. Zhang, and C. Zhao, \textit{Symmetric stationary boundary layer}, arXiv:2303.17781 (2023).

\bibitem{GD} D. G\'{e}rard-Varet and E. Dormy, \textit{On the ill-posedness of the Prandtl equation}, J. Amer. Math. Soc. 23 (2010), 591--609.

\bibitem{GM} D. G\'{e}rard-Varet and N. Masmoudi, \textit{Well-posedness for the Prandtl system without analyticity or monotonicity}, Ann. Sci. \'{E}c. Norm. Sup\'{e}r. (4) 48 (2015), no. 6, 1273--1325.

\bibitem{GWZ} {Y. Guo, Y. Wang, and Z. Zhang, \textit{Dynamic stability for steady Prandtl solutions}, Ann. PDE 9 (2023), no.~2, Paper No.~16, 33 pp., doi:10.1007/s40818-023-00160-x.}

\bibitem{Hartman2002} P. Hartman, \textit{Ordinary Differential Equations}, 2nd ed., Classics in Applied Mathematics, vol. 38, Society for Industrial and Applied Mathematics, Philadelphia, PA, 2002; unabridged, corrected republication of the 1982 Birkh\"{a}user edition.

\bibitem{Hassan} H. A. Hassan, \textit{On unsteady laminar boundary layers}, J. Fluid Mech. 9 (1960), no. 2, 300--304.

\bibitem{Hayasi1} N. Hayasi, \textit{On similar solutions of the unsteady quasi-two-dimensional incompressible laminar boundary-layer equations}, J. Phys. Soc. Japan 16 (1961), no. 11, 2316--2328.

\bibitem{Hayasi2} N. Hayasi, \textit{On semi-similar solutions of the unsteady quasi-two-dimensional incompressible laminar boundary-layer equations}, J. Phys. Soc. Japan 17 (1962), no. 1, 194--203.

\bibitem{Iyer1} S. Iyer, \textit{On global-in-x stability of Blasius profiles}, Arch. Ration. Mech. Anal. 237 (2020), no. 2, 951--998.

\bibitem{Iyer2} S. Iyer, \textit{Stability of the favorable Falkner--Skan profiles for the stationary Prandtl equations}, Adv. Math. 487 (2026), Paper No. 110749, 78 pp.

\bibitem{JLY} H. Jia, Z. Lei, and C. Yuan, \textit{Sharp asymptotic stability of Blasius profile in the steady Prandtl equation}, Adv. Math. 480 (2025), part C, Paper No. 110533, 65 pp.

\bibitem{JLY-dyn} H. Jia, Z. Lei, and C. Yuan, \textit{Global well-posedness and regularity of the dynamical Prandtl equation}, arXiv:2606.13134 (2026).

\bibitem{LY1} W.-X. Li and T. Yang, \textit{Well-posedness in Gevrey function spaces for the Prandtl equations with non-degenerate critical points}, J. Eur. Math. Soc. (JEMS) 22 (2020), no. 3, 717--775.

\bibitem{LMY} W.-X. Li, N. Masmoudi, and T. Yang, \textit{Well-posedness in Gevrey function space for 3D Prandtl equations without structural assumption}, Comm. Pure Appl. Math. 75 (2022), no. 8, 1755--1797.

\bibitem{LWY2} C.-J. Liu, Y.-G. Wang, and T. Yang, \textit{On the ill-posedness of the Prandtl equations in three-dimensional space}, Arch. Ration. Mech. Anal. 220 (2016), no. 1, 83--108.

\bibitem{LWY1} C.-J. Liu, Y.-G. Wang, and T. Yang, \textit{A well-posedness theory for the Prandtl equations in three space variables}, Adv. Math. 308 (2017), 1074--1126.

\bibitem{Maekawa} Y. Maekawa, \textit{On the inviscid limit problem of the vorticity equations for viscous incompressible flows in the half-plane}, Comm. Pure Appl. Math. 67 (2014), no. 7, 1045--1128.

\bibitem{MW} N. Masmoudi and T. K. Wong, \textit{Local-in-time existence and uniqueness of solutions to the Prandtl equations by energy methods}, Comm. Pure Appl. Math. 68 (2015), 1683--1741.

\bibitem{Nanbu} {K. Nanbu, \textit{Unsteady Falkner--Skan flow}, Z. Angew. Math. Phys. 22 (1971), 1167--1172, doi:10.1007/BF01590885.}

\bibitem{OS} O. A. Oleinik and V. N. Samokhin, \textit{Mathematical models in boundary layer theory}, Applied Mathematics and Mathematical Computation, vol.~15, Chapman \& Hall/CRC, Boca Raton, FL, 1999.

\bibitem{Peletier} L. A. Peletier, \textit{On the asymptotic behaviour of velocity profiles in laminar boundary layers}, Arch. Ration. Mech. Anal. 45 (1972), no. 2, 110--119.

\bibitem{SC} M. Sammartino and R. E. Caflisch, \textit{Zero viscosity limit for the analytic solutions of the Navier-Stokes equations on a half-space. I. Existence for the Euler and Prandtl equations}, Comm. Math. Phys. 192 (1998), 433--461.

\bibitem{SG} H. Schlichting and K. Gersten, \textit{Boundary-Layer Theory}, 9th ed., Springer, Berlin, Heidelberg, 2017.

\bibitem{Serrin} J. Serrin, \textit{Asymptotic behavior of velocity profiles in the Prandtl boundary layer theory}, Proc. Roy. Soc. London Ser. A 299 (1967), 491--507.

\bibitem{Stewartson1} K. Stewartson, \textit{On the impulsive motion of a flat plate in a viscous fluid}, Q. J. Mech. Appl. Math. 4 (1951), no. 2, 182--198.

\bibitem{Stewartson2} K. Stewartson, \textit{The theory of unsteady laminar boundary layers}, Advances in Applied Mechanics 6 (1960), 1--37.

\bibitem{Sun} B. H. Sun, \textit{Similarity solutions of a class of unsteady laminar boundary layer}, Physics of Fluids 36 (2024), 083616.

\bibitem{Tokuda} N. Tokuda, \textit{On the impulsive motion of a flat plate in a viscous fluid}, J. Fluid Mech. 33 (1968), no. 4, 657--675.

\bibitem{DS} L. L. van Dommelen and S. F. Shen, \textit{The spontaneous generation of the singularity in a separating laminar boundary layer}, J. Comput. Phys. 38 (1980), 125--140.

\bibitem{WZ2} Y. Wang and Z. Zhang, \textit{Asymptotic behavior of the steady Prandtl equation}, Math. Ann. 387 (2023), nos.~3--4, 1289--1331.

\bibitem{Wang1985} {J. C. T. Wang, \textit{Renewed studies on the unsteady boundary layers governed by singular parabolic equations}, J. Fluid Mech. 155 (1985), 413--427, doi:10.1017/S0022112085001872.}

\bibitem{W} H. Weyl, \textit{On the differential equations of the simplest boundary-layer problems}, Ann.\ of Math. (2) 43 (1942), 381--407.

\bibitem{WilliamsJohnson} {J. C. Williams III and W. D. Johnson, \textit{Semi-similar solutions to unsteady boundary layer flows including separation}, AIAA J. 12 (1974), no.~10, 1388--1393, doi:10.2514/3.49494.}

\bibitem{WilliamsRhyne} {J. C. Williams III and T. B. Rhyne, \textit{Boundary layer development on a wedge impulsively set into motion}, SIAM J. Appl. Math. 38 (1980), no.~2, 215--224, doi:10.1137/0138019.}

\bibitem{XZ1} Z. Xin and L. Zhang, \textit{On the global existence of solutions to the Prandtl system}, Adv. Math. 181 (2004), 88--133.

\bibitem{XZZ} Z. Xin, L. Zhang, and J. Zhao, \textit{Global well-posedness and regularity of weak solutions to the Prandtl's system}, SIAM J. Math. Anal. 56 (2024), no. 3, 3042--3081.

\end{thebibliography}
\end{document}